\documentclass[a4paper]{article}
\usepackage[utf8]{inputenc}

\usepackage[T1]{fontenc}

\usepackage{xfrac}

\usepackage{amsmath}
\usepackage{amsthm}
\usepackage{amstext}
\usepackage{amssymb}
\usepackage{tikz}
\usetikzlibrary{cd}
\usetikzlibrary{matrix,arrows}
\usepackage{thmtools, thm-restate}

\usepackage{mathtools}
\usepackage{enumitem}
\usepackage{comment}

\usepackage[colorlinks, citecolor=blue]{hyperref}    % better urls in bibliography

\usepackage{cleveref}

\newtheoremstyle{mythm}
  {\topsep} % Space above
  {\topsep} % Space below
  {\itshape} % Body font
  {} % Indent amount
  {\bfseries} % Theorem head font
  {.} % Punctuation after theorem head
  {.5em} % Space after theorem head
  {} % Theorem head spec (can be left empty, meaning `normal')

\newtheoremstyle{thmintrostyle}
  {\topsep} % Space above
  {\topsep} % Space below
  {\itshape} % Body font
  {} % Indent amount
  {\bfseries} % Theorem head font
  {.} % Punctuation after theorem head
  {.5em} % Space after theorem head
  {}%{\thmname{#1}\thmnumber{ #2}} % Theorem head spec (can be left empty, meaning `normal')

\theoremstyle{plain}
\newtheorem{thm}{Theorem}[section]
\newtheorem{prop}[thm]{Proposition}
\newtheorem{lemma}[thm]{Lemma}

\newtheorem{cor}[thm]{Corollary}

\newtheorem{fact}[thm]{Fact}
\newtheorem{claim}[thm]{Claim}

\newtheorem{remark}[thm]{Remark}
\newtheorem{obs}[thm]{Observation}

\newtheorem*{remark*}{Remark}
\newtheorem*{falseexample*}{Naive reasoning}

\Crefname{prop}{Proposition}{Propositions}

\theoremstyle{thmintrostyle}
\newtheorem{thmintro}{Theorem}[] 

\newtheorem{corintro}[thmintro]{Corollary} 
\newtheorem{propintro}[thmintro]{Proposition}

\theoremstyle{definition}
\newtheorem{definition}[thm]{Definition}

\makeatletter
\newcommand{\proofpoint}[1]{%
  \par
  \addvspace{\medskipamount}%
  \noindent\emph{\underline{Proof of point #1:}}\par\nobreak
  \addvspace{\smallskipamount}%
  \@afterheading
}
\makeatother

\makeatletter
\newcommand{\proofpart}[2]{%
  \par
  \addvspace{\medskipamount}%
  \noindent\emph{#1 : #2}\par\nobreak
  \addvspace{\smallskipamount}%
  \@afterheading
}
\makeatother

\newcommand{\Addresses}{{
		\bigskip
		\footnotesize
		
		\textsc{Alfréd Rényi Institute of Mathematics, Budapest, Hungary}\par\nopagebreak
		\texttt{fabio.gironella@renyi.hu, fabio.gironella.math@gmail.com}
		
}}

\newcommand{\nat}{\mathbb{N}}
\newcommand{\integ}{\mathbb{Z}}
\newcommand{\rat}{\mathbb{Q}}
\newcommand{\real}{\mathbb{R}}
\newcommand{\compl}{\mathbb{C}}

\newcommand{\cercle}{\mathbb{S}^1}
\newcommand{\disk}{D}
\newcommand{\torus}{\mathbb{T}^2}
\newcommand{\tore}{\mathbb{T}^3}
\newcommand{\sphere}[1]{\mathbb{S}^{#1}}
\newcommand{\polynomials}[1]{\real\left[#1\right]}

\newcommand{\normalbundle}[1]{\mathcal{N}#1}
\newcommand{\vol}{\text{vol}}

\newcommand{\Diff}{\text{Diff}}

\newcommand{\What}{\widehat{W}}
\newcommand{\Xhat}{\widehat{X}}

\newcommand{\Mhat}{\widehat{M}}
\newcommand{\xihat}{\widehat{\xi}}

\newcommand{\xik}{\xi_{k}}
\newcommand{\alphak}{\alpha_{k}}
\newcommand{\xiks}{\xik^{s}}
\newcommand{\alphaks}{\alphak^{s}}

\newcommand{\connH}{\mathcal{H}}
\newcommand{\Sigmatilde}{\widetilde{\Sigma}}
\newcommand{\rhotilde}{\widetilde{\rho}}

\newcommand{\Vhat}{\widehat{V}}
\newcommand{\phat}{\widehat{p}}
\newcommand{\etahat}{\widehat{\eta}}
\newcommand{\etazerohat}{\widehat{\eta}_0}

\newcommand{\Sigmahat}{\widehat{\Sigma}}

\newcommand{\Sigmatheta}{\Sigma_{\theta}}

\newcommand{\diskone}{\disk^{2}}
\newcommand{\diskdelta}{\diskone_{\delta}}

\newcommand{\omegaV}{\omega_{V}}

\newcommand{\omegahat}{\widehat{\omega}}
\newcommand{\omegahateps}{\widehat{\omega}_{\epsilon}}
\newcommand{\omegaVhat}{\widehat{\omega}_{V}}
\newcommand{\omegaVhateps}{\widehat{\omega}_{\epsilon,V}}

\newcommand{\piomega}{\pi^*\omega}

\newcommand{\alphahat}{\widehat{\alpha}}
\newcommand{\alphahateps}{\alphahat_{\epsilon}}

\newcommand{\polytau}{\polynomials{\tau}}

\newcommand{\calUhat}{\widehat{\mathcal{U}}}

\newcommand{\Mtor}{M\times\torus}

\newcommand{\omegat}{\omega_{\torus}}
\newcommand{\aeps}{\alpha_{\epsilon}}

\newcommand{\aone}{\alpha_{1}}

\newcommand{\omegaM}{\omega_{M}}
\newcommand{\omegam}{\omega_{\vert_{TM}}}
\newcommand{\dun}{\phi_1 d\theta_1}
\newcommand{\ddeux}{\phi_2 d\theta_2}
\newcommand{\ddun}{d\phi_1\wedge d\theta_1}
\newcommand{\dddeux}{d\phi_2\wedge d\theta_2}
\newcommand{\eqpiece}{(\omegaM+\tau d\beta)^{n-1}}
\newcommand{\dxdy}{d\theta_1\wedge d\theta_2}
\newcommand{\surfg}{\Sigma_{g}}
\newcommand{\Msurfg}{M\times\surfg}

\newcommand{\omegag}{\omega_{g}}

\newcommand{\Xsurfg}{X\times\surfg}
\newcommand{\Mline}{\overline{M}}
\newcommand{\xiline}{\overline{\xi}}

\newcommand{\betaB}{\beta_{B}}

\newcommand{\hone}{h_{1}}
\newcommand{\htwo}{h_{2}}
\newcommand{\phione}{\phi_{1}}
\newcommand{\phitwo}{\phi_{2}}

\newcommand{\neigh}{\mathcal{N}}

\newcommand{\ReebB}{R_{B}}
\newcommand{\Reebba}{R_{\alpha}}

\newcommand{\dex}{\partial_{x}}
\newcommand{\dey}{\partial_{y}}

\newcommand{\bigslant}[2]{{\raisebox{.2em}{$#1$}\left/\raisebox{-.2em}{$#2$}\right.}}

\newcommand{\diskoneeps}{\diskone_{1-\epsilon}}
\newcommand{\diskoneepshalf}{\diskone_{1-{\epsilon}/{2}}}

\newcommand{\norm}[1]{\left\Vert #1 \right\Vert}
\newcommand{\lie}{\mathcal{L}}
\newcommand{\lieX}{\lie_{X}}
\newcommand{\lieY}{\lie_{Y}}
\newcommand{\lieXtheta}{\lie_{\Xtheta}}
\newcommand{\lieYtheta}{\lie_{\Ytheta}}

\newcommand{\costheta}{\cos\left(\theta\right)\,}
\newcommand{\sintheta}{\sin\left(\theta\right)\,}
\newcommand{\costhetaprime}{\cos\left(\theta'\right)\,}
\newcommand{\sinthetaprime}{\sin\left(\theta'\right)\,}

\newcommand{\contvolform}{\alpha\wedge d\alpha^{n-1}}
\newcommand{\Ki}{\left(K_{i}\right)_{i}}

\newcommand{\Xtheta}{X_{\theta}}
\newcommand{\Ytheta}{Y_{\theta}}

\newcommand{\Xthetaprime}{X_{\theta'}}

\newcommand{\Sigmathetaprime}{\Sigma_{\theta'}}

\newcommand{\ftheta}{f_{\theta}}
\newcommand{\gtheta}{g_{\theta}}

\newcommand{\iX}{\iota_{X}}
\newcommand{\iY}{\iota_{Y}}

\newcommand{\alphaX}{\alpha\left(X\right)}
\newcommand{\alphaY}{\alpha\left(Y\right)}
\newcommand{\dalphaX}{d\left(\alphaX\right)}
\newcommand{\dalphaY}{d\left(\alphaY\right)}
\newcommand{\alphaXY}{\alpha\left([X,Y]\right)}

\newcommand{\dalphaXY}{d\alpha\left(X,Y\right)}

\newcommand{\alphaXtheta}{\alpha\left(\Xtheta\right)}
\newcommand{\alphaXthetaprime}{\alpha\left(\Xthetaprime\right)}

\newcommand{\alphaYtheta}{\alpha\left(\Ytheta\right)}

\newcommand{\dalphaXtheta}{d\left(\alphaXtheta\right)}
\newcommand{\dalphaYtheta}{d\left(\alphaYtheta\right)}

\newcommand{\alphaXYtheta}{\alpha\left([\Xtheta,\Ytheta]\right)}
\newcommand{\iXthetaalpha}{\iota_{\Xtheta}\alpha}

\newcommand{\iYthetaalpha}{\iota_{\Ytheta}\alpha}
\newcommand{\iXthetadalpha}{\iota_{\Xtheta}d\alpha}
\newcommand{\iYthetadalpha}{\iota_{\Ytheta}d\alpha}
\newcommand{\dalphaXYtheta}{d\alpha\left(\Xtheta,\Ytheta\right)}

\makeatletter
\newcommand{\lowermath}[1]{\mathpalette{\lowerm@th{#1}}}
\newcommand{\lowerm@th}[3]{\raisebox{- #1}{$#2#3$}}
\makeatother

\newcommand{\flowxyp}{\psi_{y\cdot X+x \cdot Y}^1(p)}

\newcommand{\Kz}{K_{z}}
\newcommand{\intpart}{\textit{Int}}

\newcommand{\shortderivr}[1]{\frac{\partial #1}{\partial r}}

\def\co{\colon\thinspace}
\def\coeq{\coloneqq\thinspace}

\newcommand{\Mod}[1]{\ \mathrm{mod}\ #1}

\DeclareMathOperator{\Id}{Id}

\DeclareMathOperator{\Image}{Im}

\DeclareMathOperator{\SO}{SO}

\newcommand{\Aline}{\overline{A}}
\newcommand{\etaline}{\overline{\eta}}
\newcommand{\Xline}{\overline{X}}
\newcommand{\Yline}{\overline{Y}}

\newcommand{\fol}{\mathcal{F}}
\newcommand{\pihat}{\widehat{\pi}}

\DeclareMathOperator{\pr}{pr}

\newcommand{\VF}{\mathfrak{X}}
\newcommand{\VFfib}{\VF_{fib}}
\newcommand{\xifib}{\xi_{fib}}
\newcommand{\dnabla}{d_\nabla}

\newcommand{\calOhat}{\widehat{\mathcal{O}}}

\title{On some examples and constructions \\ of contact manifolds}
\author{Fabio Gironella}
\date{}

\begin{document}

\maketitle
\begin{abstract}
	The first goal of this paper is to construct examples of higher dimensional contact manifolds with specific properties. 
	Our main results in this direction are the existence of tight virtually overtwisted closed contact manifolds in all dimensions and the fact that every closed contact 3-manifold, which is not (smoothly) a rational homology sphere, contact--embeds with trivial normal bundle inside a hypertight closed contact 5-manifold.
	
	This uses known construction procedures by Bourgeois (on products with tori) and Geiges (on branched covering spaces). We pass from these procedures to definitions; 
	this allows to prove a uniqueness statement in the case of contact branched coverings, and to study the global properties (such as tightness and fillability) of the results of both constructions without relying on any auxiliary choice in the procedures.
	
	A second goal allowed by these definitions is to study relations between these constructions and the notions of supporting open book, as introduced by Giroux, and of contact fiber bundle, as introduced by Lerman.
	For instance, we give a definition of Bourgeois contact structures on flat contact fiber bundles which is local, (strictly) includes the results of the Bourgeois construction, and allows to recover an isotopy class of supporting open books on the fibers. This last point relies on a reinterpretation, inspired by an idea by Giroux, of supporting open books in terms of pairs of contact vector fields.
	% \keywords{First keyword \and Second keyword \and More}
\end{abstract}

\section{Introduction}
\label{SecIntro}

This paper is concerned with the systematic study of some explicit constructions of high dimensional co--oriented contact structures, i.e. of hyperplane fields $\xi$ on oriented smooth manifolds $M^{2n-1}$ which are given by the kernel of $\alpha\in\Omega^{1}(M)$ such that $\alpha\wedge d\alpha^{n-1}$ is a positive volume form on $M$.
More precisely, the focus is on the constructions due to Geiges \cite{Gei97} and Bourgeois \cite{Bou02}.
\\
In the first article, developing ideas from Gromov \cite{GroPartDiff}, Geiges transposes some constructions from the symplectic world to the contact setting, introducing in particular the notion of contact branched coverings. 
Contact fiber sums and contact reductions are also constructed, but we will not deal with them in the following (see Gironella \cite[Section $5.3$]{MyPhDThesis} for the case of contact fiber sums).
\\ 
In the paper \cite{Bou02}, taking inspiration from Lutz \cite{Lut79}, Bourgeois proves that, given a closed contact manifold $(M^{2n-1},\xi)$ and an open book decomposition $(B,\varphi)$ of $M$ supporting $\xi$, there is a contact structure $\eta$ on $M\times\torus$ that is invariant under the natural $\torus$-action, that restricts to $\xi$ on each submanifold $M\times\{pt\}$ and that naturally deforms to the hyperplane field $\xi\oplus T\torus$ on $M\times\torus$. 
Recall that, according to Giroux \cite{Gir02}, for any contact manifold $(M^{2n-1},\xi)$, one can always find an open book decomposition $(B,\varphi)$ on $M$ \emph{supporting} $\xi$, i.e. such that $B$ is a positive contact submanifold and there is $\alpha\in\Omega^1(M)$ defining $\xi$ such that $d\alpha$ is a positive symplectic form on the fibers of $\varphi\co M\setminus B \to \cercle$.
\\

The main motivation behind both \cite{Gei97,Bou02} was the problem of the existence of contact structures, i.e. the question of which high dimensional manifolds admit a contact structure.
This (big) problem in contact topology has now been solved by Borman--Eliashberg--Murphy \cite{BorEliMur15}: high--dimensional contact structures exist whenever the corresponding formal objects, i.e. almost contact structures, exists. 
As a consequence, the aim has now shifted from providing examples to providing  \emph{``interesting''} examples of contact structures.

The papers \cite{Gei97,Bou02} fit well in this perspective because they actually give rather explicit contact manifolds, which can be studied in some detail and which (under the right conditions) manifest interesting properties of tightness, fillability, overtwistedness, etc.
For instance, these two papers provided the first explicit methods of building PS-overtwisted (hence overtwisted, according to the posterior Casals--Murphy--Presas \cite{CMP15} and Huang \cite{Hua16}) contact manifolds in high dimensions. 
The interested reader can consult Presas \cite{Pre07} for the case of the construction in \cite{Bou02} and Niederkr\"uger--Presas \cite[page 724]{NiePre10} for the case of contact branched coverings; see also Niederkr\"uger \cite[Theorem I.5.1]{NieThesis}, attributed to Presas, which uses contact fiber sums. 
Compare also with \Cref{ClaimDoubleCovOT} in \Cref{SubSecExVirtOTMfld} below.

The aim of this article is hence to construct contact manifolds with particular properties starting from \cite{Gei97,Bou02}. 
In order to do so, we need to pass from the construction procedures by Geiges and Bourgeois to definitions.
We can then study the properties of these contact structures, without the need to rely on any auxiliary choice made in their actual constructions in \cite{Gei97,Bou02}.\\

As far as contact branched coverings are concerned, we point out that the uniqueness problem is not explicitly addressed in \cite{Gei97}, i.e. it is not shown that the objects obtained are independent of the auxiliary choices made to build them. 
We hence propose in this paper a definition of contact branched coverings that allows to naturally obtain a uniqueness (up to isotopy) statement.
\\
A definition and a uniqueness statement can also be given in the case of contact fiber sums; see Gironella \cite[Section $5.3$]{MyPhDThesis}.

We remark that in the literature there is already a definition of contact branched coverings that goes in this direction. 
Indeed, \"Ozt\"urk--Niederkr\"uger \cite{OztNie07} define this notion in terms of contact deformations verifying an additional condition at the branching locus. 
Removing this further constraint, we show here the following:
\begin{propintro}
	\label{PropIntroDefContBranchCov}
	Let $(V^{2n-1},\eta)$ be a contact manifold and $\pi\co\Vhat\rightarrow V$ be a smooth branched covering map with downstairs branching locus $M$.
	Suppose that $\eta\cap TM$ is a contact structure on $M$.
	Then:
	\begin{enumerate}
		\item \label{Point1PropIntroDefContBranchCov} there is a $[0,1]$-family of hyperplane fields $\etahat_t$ on $\Vhat$ 
		such that $\etahat_0=\pi^*\eta$ and $\etahat_t$ is a contact structure for all $t\in(0,1]$;
		\item \label{Point2PropIntroDefContBranchCov} if $\etahat_t$ and $\etahat'_t$ are as in point \ref{Point1PropIntroDefContBranchCov}, then $\etahat_r$ is isotopic 
		to $\etahat'_s$ for all $r,s\in(0,1]$. 
	\end{enumerate}
	Moreover, in point \ref{Point1PropIntroDefContBranchCov}, $\etahat_t$ can be chosen invariant under local deck transformations of $\pi$ for all $t\in(0,1]$. 
	Similarly, the isotopy in point  \ref{Point2PropIntroDefContBranchCov} can be chosen among contact structures invariant under local deck transformations, provided that $\etahat_t$ and $\etahat'_t$ are invariant too.
\end{propintro}
We will hence call \emph{contact branched covering} a contact structure $\etahat$ on $\Vhat$ that is the endpoint of any path $\etahat_t$ as above.
Notice that \Cref{PropIntroDefContBranchCov} tells exactly that this object exists and is well defined up to isotopy.

At this point, we are able to give precise statements about the properties of contact branched coverings. For instance, we prove the following:
\begin{thmintro}
	\label{ThmIntroWeakFillBranchCov}
	Consider a smooth branched covering $\pi\co\Vhat\rightarrow V$ and a contact structure $\xi$ on $V$ and let $\etahat$ be a contact branched covering of $\eta$.
	Suppose that $(V,\eta)$ is weakly filled by $(W,\Omega)$ in such a way that the downstairs branching locus $M$ of $\pi$ is filled by a symplectic submanifold $X$ of $(W,\Omega)$. 
	Suppose also that $\pi$ extends to a smooth branched covering $\widehat{\pi}\co \What\rightarrow W$ branched over $X$.
	Then, there is a symplectic structure $\widehat{\Omega}$ on $\What$ weakly filling $\etahat$ on $\Vhat=\partial \What$.
\end{thmintro}

We then devote a part of the paper to an analysis and a generalization of the Bourgeois construction in \cite{Bou02}.

As already recalled above, one can look at the examples in \cite{Bou02} in two different and ``orthogonal'' ways, namely via the projections $M\times\torus\rightarrow M$ and $M\times\torus\rightarrow\torus$.
The first one tells that these examples are $\torus$-invariant contact structures on the total space of the $\torus$-bundle $M\times\torus\rightarrow M$. 
We will not deal with this point of view here and we invite the interested reader to consult Gironella \cite[Chapter 7]{MyPhDThesis}, where the links between the construction in \cite{Bou02} and the study of $\torus$-invariant contact structures in Lutz \cite{Lut79} are analyzed in detail.
The second point of view shows that the examples in \cite{Bou02} are contact structures on $M\times\torus$ which moreover induce a contact structure on each fiber of $M\times\torus\rightarrow\torus$, i.e., using the language introduced by Lerman in \cite{Ler04}, which are \emph{contact fiber bundles} on $M\times\torus\rightarrow\torus$.
We point out that this contact bundle structure on the examples from \cite{Bou02} has already been exploited successfully in Presas \cite{Pre07}, van Koert--Niederkr\"uger \cite{NieVKo07}, Niederkr\"uger--Presas \cite{NiePre10}, Etnyre--Pancholi \cite{EtnPan11,EtnPan16} to obtain high dimensional contact manifolds with remarkable properties.
This suggests that this second point of view might be the best one to analyze and generalize the construction in \cite{Bou02}.

In this paper we then use the theory of contact fiber bundles from Lerman \cite{Ler04} in order to generalize the Bourgeois construction and define the notion of \emph{Bourgeois contact structures}.
More precisely, on a fiber bundle $\pi\co V^{2n+1}\rightarrow\Sigma^{2}$ equipped with a reference contact fiber bundle $\eta_0$, every contact fiber bundle $\eta$ admits a potential form $A$ with respect to $\eta_0$, with a well defined curvature form $R_A$.
In the case where the reference contact bundle $\eta_0$ is flat, we call Bourgeois contact structure any contact fiber bundle structure on $\pi\co V\rightarrow\Sigma$ that is also a contact structure on $V$ and verifies $\frac{1}{\epsilon}R_{\epsilon A}\rightarrow 0$ for $\epsilon\rightarrow0$.

Beside the need to pass from the construction procedure in \cite{Bou02} to a definition, another motivation behind the introduction of this notion is the following: the condition on the curvature is, on one hand, weak enough to be satisfied by a class of contact structures strictly containing the results of the construction in \cite{Bou02} and, on the other hand, strong enough to ensure some nice properties, for instance from the points of view of weak fillings and adapted open book decompositions (other properties will also be analyzed in \Cref{SubSubSecBourgContStrRev}).
\\
As far as the weak-fillability is concerned, we prove the following:
\begin{propintro}
	\label{PropIntroBourgContStrWeakFill}
	Let $(M^{2n-1},\xi)$ be a contact manifold and $\eta$ be a Bourgeois contact structure on the trivial fiber bundle $M\times\torus\rightarrow\torus$, that restricts to $\xi$ on $M\times\{pt\}=M$. 
	If $(M,\xi)$ is weakly filled by $(X^{2n},\omega)$, then $(\Mtor,\eta)$ is weakly filled by $(X\times\torus, \omega+\omegat)$, where $\omegat$ is an area form on $\torus$.
\end{propintro}
We point out that the result is already known in the case of the Bourgeois construction \cite{Bou02}. 
Indeed, the statement and the idea of the proof already appeared in Massot--Niederkr\"uger--Wendl \cite[Example 1.1]{MNW13}; see also Lisi--Marinkovi\'c--Niederkr\"uger \cite[Theorem A.a]{LisMarNie18} for an explicit proof.

From the point of view of adapted open books, Bourgeois contact structures implicitly carry some information on open book decompositions supporting the contact structures on each fiber: 
\begin{propintro}
	\label{PropIntroIsotopyClassOBD}
	Let $\eta$ be a Bourgeois contact structure on $\pi\co V\rightarrow \Sigma$. Then, there is a map $\psi_\eta$ that associate to each point $b\in\Sigma$ an isotopy class of adapted open book decompositions on $(M_b,\xi_b)\coeq \left(\pi^{-1}\left(b\right),\eta\cap T \left(\pi^{-1}\left(b\right)\right) \right)$.
	Moreover, if $\gamma(t)$, with $t\in(-\epsilon,\epsilon)$, is a path in an open set $U$ of $\Sigma$ over which $\pi$ is trivialized, i.e. over which $\pi$ becomes the projection on the first factor $\pr_U \co U\times M \rightarrow U$, then the path of isotopy classes $\psi_\eta\circ\gamma(t)$ comes from a path of open books $(B_t,\varphi_t)$ of $\{\gamma(t)\}\times M$
	such that its image via $\pr_M\co U\times M \rightarrow M$ is an isotopy of open books on $M$.\\
	In the case of the examples from \cite{Bou02}, via the global $\pr_M\co M\times\torus\rightarrow M$, the map $\psi_\eta$ gives the isotopy class of the open book $(B,\varphi)$ used in the construction.
\end{propintro}
%In other words, if on one hand the Lutz construction gives a converse to the Bourgeois one from the point of view of $\torus$-invariant contact structures on the principal bundle $M\times\torus\rightarrow M$,
%on the other hand this statement gives a converse to the Bourgeois construction from the point of view of contact fiber bundles on $M\times\torus\rightarrow\torus$.

In order to prove \Cref{PropIntroIsotopyClassOBD}, we give a reinterpretation of adapted open books in terms of pairs of contact vector fields:
\begin{thmintro}
	\label{ThmOBDCoupleContVF}
	On a contact manifold $(M^{2n-1},\xi)$, a supporting open book decomposition gives a pair of contact vector fields $X,Y$, such that $[X,Y]$ is everywhere transverse to $\xi$.
	Viceversa, such a pair of contact vector fields allows to recover a supporting open book decomposition.
	%On a contact manifold $(M^{2n-1},\xi)$, each couple of contact vector fields $X,Y$, such that $[X,Y]$ is everywhere transverse to $\xi$, determines an explicit open book decomposition on $M$ supporting $\xi$. 
	%Viceversa, an open book supporting $\xi$ allows to recover a couple $X,Y$ as above.
\end{thmintro}
The first part of this result has been stated by Giroux in talks for the Yashafest in June 2007 and for the AIM workshop of May 2012 (see Giroux \cite[Claim on page 19]{GirouxAIM}).
A more detailed statement and a detailed proof of \Cref{ThmOBDCoupleContVF} are given in \Cref{SecOBDContVF}. 
We point out that this result does not only serve to prove \Cref{PropIntroIsotopyClassOBD} but also gives another point of view on adapted open book decompositions, which is of independent interest. 
\\ 

These reinterpretations and generalizations of \cite{Gei97,Bou02} lead us to examples of high dimensional contact manifolds with interesting tightness, fillability or overtwistedness properties. 
As a byproduct, we obtain two new results, one concerning tight virtually overtwisted contact structures and one concerning codimension $2$ embeddings with trivial normal bundle of contact $3$-manifolds.\\

As far as the first result is concerned, we recall that a tight contact structure $\xi$ on $M$ is called \emph{virtually overtwisted} if its pullback $\widehat{\xi}$ on a finite cover $\widehat{M}$ of $M$ is overtwisted.
In this paper, we prove the following:
\begin{thmintro}
	\label{ThmExistenceVirtOT}
	Virtually overtwisted structures exist in all odd dimensions $\geq3$.
\end{thmintro}
The proof of this result is by induction on the dimension.
As far as the initialization step is concerned, the existence of tight virtually overtwisted contact structures is known in dimension $3$ since Gompf \cite{Gom98}. The interested reader can also consult Giroux \cite{Gir00} and Honda \cite{Hon00}, which present a classification result for this type of contact structures on particular $3$--manifolds.
The inductive step uses Propositions \ref{PropIntroBourgContStrWeakFill} and \ref{PropIntroDefContBranchCov} above, i.e. the fact that both the construction in \cite{Bou02} and contact branched coverings preserve the weak fillability condition, and relies on the existence of supporting open books proven by Giroux \cite{Gir02}, on the Bourgeois construction \cite{Bou02} and on the ``large'' neighborhood criterion for overtwistedness proven in \cite[Theorem 3.1]{CMP15}.
\\

Another application concerns the following question: for a given contact manifold $\left(M,\xi=\ker\alpha\right)$, is there $\epsilon>0$ such that $\left(M\times\disk^2_\epsilon,\ker\left(\alpha+r^2 d\varphi\right)\right)$ is tight? Here, $\disk^2_\epsilon$ is the disk of radius $\epsilon$ and centered at the origin in $\real^2$, and $(r,\varphi)$ are its polar coordinates.
\\
This is linked to the problem of finding codimension $2$ contact-embeddings with trivial normal bundle in tight ambient manifolds. Indeed, having trivial normal bundle and trivial conformal symplectic normal bundle is equivalent in codimension $2$. 
Hence, according to the contact neighborhood theorem (see for instance Geiges \cite[Theorem 2.5.15]{Gei08}), if $(M^{2n-1},\xi=\ker\alpha)$ embeds into $(V^{2n+1},\eta)$ with trivial normal bundle then it admits a neighborhood $\left(M\times\disk^2_{r_{0}},\ker\left(\alpha+r^2 d\varphi\right)\right)$, for a certain $r_0>0$. 
In particular, if $(V,\eta)$ is tight, so is this neighborhood.

\sloppy{Historically, the first motivation for addressing the above question on the ``size'' of the neighborhood of a codimension $2$ submanifold is given by Niederkr\"uger--Presas \cite{NiePre10}, where it is shown that ``big'' neighborhoods of contact overtwisted submanifolds obstruct fillability of the ambient manifold. 
	As reported in Niederkr\"uger \cite{NieThesis}, this led Niederkr\"{u}ger and Presas to conjecture that the presence of a chart contactomorphic to a product of an overtwisted $\real^3$ and a ``large'' neighborhood in $\real^{2n}$ with the standard Liouville form could be the correct generalization of overtwistedness to dimensions greater than $3$.
	After the introduction in Borman--Eliashberg--Murphy \cite{BorEliMur15} of a definition of overtwisted structures in all dimensions, Casals--Murphy--Presas \cite{CMP15} confirmed this conjecture, proving that the presence of such a chart in a contact manifold is indeed equivalent to it being overtwisted.
	More precisely, this follows from \cite[Theorem 3.1]{CMP15}, which states that, if $\left(M,\xi=\ker\alpha\right)$ is overtwisted, then $\left(M\times\disk^2_R,\ker\left(\alpha+r^2 d\varphi\right)\right)$ is also overtwisted, provided that $R>0$ is sufficiently large.
	In particular, this motivates the above question on the existence, for a given contact manifold $\left(M,\xi=\ker\alpha\right)$, of an $\epsilon>0$ such that $\left(M\times\disk^2_\epsilon,\ker\left(\alpha+r^2 d\varphi\right)\right)$ is tight.} 

The problem of finding codimension $2$ embeddings in tight manifolds has already been explicitly addressed for instance by Casals--Presas--Sandon \cite{CPS14}, Etnyre--Furukawa \cite{EtnFur17} and Etnyre--Lekili \cite{EtnLek17}. 
More precisely, \cite{CPS14} proves that each $3$-dimensional overtwisted manifold can be contact-embedded with trivial normal bundle into an exact symplectically fillable closed contact $5$-manifold. In \cite{EtnFur17}, the authors shows how to embed many contact $3$-manifolds into the standard contact $5$-sphere. Finally, it is proven in \cite{EtnLek17} that each $3$-dimensional contact manifold contact-embeds in the (unique) non-trivial $\sphere{3}$-bundle over $\sphere{2}$ equipped with a Stein fillable contact structure.
\\
In this paper, we prove the following result:
\begin{restatable}{thmintro}{ThmEmbeddingsRestated}
	\label{ThmEmbeddings}
	Each $3$-dimensional contact manifold $(M,\xi)$ with $H_1\left(M;\rat\right)\neq\{0\}$ embeds with trivial normal bundle in a hypertight closed $(V^5,\eta)$. 
\end{restatable}
\begin{corintro}
	\label{CorIntroEmbeddings}
	For each $(M^3,\xi=\ker\alpha)$ with $H_1\left(M;\rat\right)\neq\{0\}$, there is $\epsilon>0$ such that $\left(M\times\disk^2_\epsilon,\ker\left(\alpha+r^2 d\varphi\right)\right)$ is tight.
\end{corintro}

We recall that a contact structure is called \emph{hypertight} if it admits a defining form with no contractible closed Reeb orbit.
Recall also that each hypertight contact manifold is in particular tight, according to Hofer \cite{Hof93}, Albers--Hofer \cite{AlbHof09} and Casals--Murphy--Presas \cite{CMP15}.

Remark that, by the Poincaré's duality and the universal coefficients theorem, the condition $H_1\left(M;\rat\right)=\{0\}$ is equivalent to $M$ being a rational homology sphere.
An analogue of \Cref{ThmEmbeddings}, with $(V^5,\eta)$ symplectically fillable, is actually already known both in the case of every contact structure on $\sphere{3}$ and in the case of overtwisted structures on any rational homology sphere.
Indeed, the case of overtwisted rational homology spheres (which includes the overtwisted $\sphere{3}$'s) is covered in Casals--Presas--Sandon \cite[Proposition 11]{CPS14}, and the standard tight $3$-sphere (which is the unique tight contact structure on $\sphere{3}$ up to isotopy according to Eliashberg \cite{Eli92}) naturally embeds in the strongly fillable standard contact $5$-sphere with trivial normal bundle.

The main ingredients we use in the proof of \Cref{ThmEmbeddings} are the existence of adapted open book decompositions for contact $3$-manifolds, due to Giroux, and a detailed study of the dynamics of the Reeb flow of the contact forms constructed in \cite{Bou02}. \\
More precisely, under the assumption $H_1\left(M;\rat\right)\neq\{0\}$, we will show that, up to positive stabilizations, each open book decomposition $(B,\varphi)$ of $M$ can be supposed to have binding components of infinite order in $H_1(M;\integ)$. 
We will then show that this allows us to get hypertight contact forms on $M\times\torus$ using \cite{Bou02}. 
Finally, $(M,\xi)$ naturally embeds in the contact manifold constructed by Bourgeois as a fiber of the fibration $M\times\torus\rightarrow\torus$ given by the projection on the second factor.

We point out that an analogue of \Cref{ThmEmbeddings} for any $M^3$ and with $(V^5,\eta)$ tight (and not necessarily hypertight) follows from Bowden--Gironella--Moreno \cite{BowGirMor}, where it is shown, among other things, that the Bourgeois construction \cite{Bou02} on any $3$--dimensional manifold results in a contact structure on its product with $\torus$ which is tight, no matter what the original contact structure and supporting open books are.

As far as \Cref{CorIntroEmbeddings} is concerned, notice that it has recently been generalized to all dimensions in Hern{\'a}ndez-Corbato -- Mart{\'{\i}}n-Merch{\'a}n -- Presas \cite{HerMarPre18} (without any assumption on $H_1\left(M;\rat\right)$), with completely different techniques. 
More precisely, there the authors deduce such a generalization from \cite[Theorem $10$]{HerMarPre18}, stating that every contact $(2n-1)$-manifold embeds with trivial conformal symplectic normal bundle in a Stein-fillable contact $(2n+2m-1)$-manifold.
This result relies on the h-principle from Cieliebak--Eliashberg \cite{CieEliBook}, and is an analogue of \Cref{ThmEmbeddings} in all dimensions, with less control on the codimension.

\paragraph{Outline}
In \Cref{SecOpContCat}, we give the announced new approach to contact branched coverings, thus proving in particular \Cref{PropIntroDefContBranchCov}. 
We also analyze the stability of the weak fillability condition under contact branched covering, thus proving \Cref{ThmIntroWeakFillBranchCov}. 
\\
\Cref{SecOBDContVF} describes the equivalent formulation, based on an idea by Giroux \cite{GirouxAIM}, of open book decompositions supporting contact structures in terms of pairs of contact vector fields and it contains the proof of \Cref{ThmOBDCoupleContVF}.
\\
%\Cref{SecLutzBourgeois} recalls the construction by Bourgeois in \cite{Bou02} and describes the study of invariant contact structures on principal $\torus-$bundles made in \cite{Lut79}. 
%\\
Then, we rephrase and generalize in \Cref{SecContFibBund} the construction by Bourgeois using the notion of contact fiber bundle introduced in Lerman \cite{Ler04}. In particular, we give the definition of Bourgeois contact structures and prove \Cref{PropIntroIsotopyClassOBD}.
\\
\Cref{SecVirtOTHighDim} contains the study of the weak fillability of Bourgeois contact structures, hence the proof of \Cref{PropIntroBourgContStrWeakFill}, and the proof of \Cref{ThmExistenceVirtOT}.
\\
Lastly, in \Cref{SecTightNeighDim3} we analyze the Reeb dynamics of the contact forms in Bourgeois \cite{Bou02} and we prove \Cref{ThmEmbeddings} and \Cref{CorIntroEmbeddings}.

%%%%%%%%%%%%%%%%%%%%%%%%%%%%%%%%%%%%%%%%%%%%%%%%%%%%%%%%%%%%%%%%%%%%%%%%%%%%%%%%%%%%%%%%%%%%%%%%%%%%%%%%%%%%%%%%%%%%%%%%%%%%%%%%%%%%%%%%%%%%%%%%%%%%%%%%%%%%%%%%%%%%%%%%%%%%%%%%%%%%%%%%%%%%%%%%%%%%%%%%%%%%%%%%%%%%%%%%%%%%%%%%%%%%%%%%%%%%%%%%%%%%%%%%%%%%

\section*{Acknowledgements}

This work is part of my PhD thesis \cite{MyPhDThesis}, written at Centre de mathématiques Laurent Schwartz of École polytechnique (Palaiseau, France).
I would like to thank P.\,Massot, who encouraged me to look at known constructions from different perspectives and greatly improved this manuscript with many suggestions.
I am also grateful to C.\,Margerin and K.\,Niederkr\"uger for useful discussions concerning, respectively,  contact connections and the Bourgeois construction, which resulted in a big improvement in the exposition in \Cref{SecContFibBund}, as well as to J.\,Bowden for pointing out a mistake in the assumptions of the previous version of \Cref{ThmEmbeddings}.
Finally, I also wish to thank the anonymous referee for many useful comments that greatly improved the presentation of the results with respect to the original version of the paper. 

%%%%%%%%%%%%%%%%%%%%%%%%%%%%%%%%%%%%%%%%%%%%%%%%%%%%%%%%%%%%%%%%%%%%%%%%%%%%%%%%%%%%%%%%%%%%%%%%%%%%%%%%%%%%%%%%%%%%%%%%%%%%%%%%%%%%%%%%%%%%%%%%%%%%%%%%%%%%%%%%%%%%%%%%%%%%%%%%%%%%%%%%%%%%%%%%%%%%%%%%%%%%%%%%%%%%%%%%%%%%%%%%%%%%%%%%%%%%%%%%%%%%%%%%%%%%

\section{Contact branched coverings}
\label{SecOpContCat}

In \Cref{SubSubSecExistUniqContBranchCov}, we give a definition of contact branched coverings that allows to naturally obtain uniqueness statements; we will in particular prove \Cref{PropIntroDefContBranchCov} stated in the introduction.
We point out that the proofs in this section are mainly a reformulation of those in Geiges \cite{Gei97}.
\\
An analogous analysis can be carried out in the case of contact fiber sums, but, as it is not necessary for our purposes, it will not be presented here and we redirect the interested reader to Gironella \cite[Section $5.3$]{MyPhDThesis}.

Then, \Cref{SubSecEffBranchCovWeakFill} contains a proof of \Cref{ThmIntroWeakFillBranchCov} stated in the introduction, i.e. of the fact that, under some natural assumptions, contact branched coverings of a weakly fillable contact manifold are also weakly fillable.

%%%%%%%%%%%%%%%%%%%%%%%%%%%%%%%%%%%%%%%%%%%%%%%%%%%%%%%%%%%%%%%%%%%%%%%%%%%%%%%%%%%%%%%%%%%%%%%%%%%%%%%%%%%%%%%%%%%%%%%%%%%%%%%%%%%%%%%%%%%%%%%%%%%%%%%%%%%%%%%%%%%%%%%%

\subsection{Definition and uniqueness}
\label{SubSubSecExistUniqContBranchCov}

Suppose $\pi:\Vhat^{2n+1}\rightarrow V^{2n+1}$ is a branched covering map of manifolds without boundary, branched along the codimension $2$ submanifold $M^{2n-1}\subset V$. 
Let $\Mhat^{2n-1}$ be the locus of points of $\Vhat$ with branching index $>1$ and $M$ its projection $\pi(\Mhat)$.
In the following, we will also refer to $\Mhat^{2n-1}$ as \emph{upstairs branching set} and to $M$ as \emph{downstairs branching set}.
Consider now $\eta$ a contact structure on $V$ such that $\xi\coeq\eta\cap TM$ is a contact structure on $M$.

The pullback $\pi^*\eta$ is a well defined hyperplane field on $\Vhat$, because if we fix a contact form $\alpha$ for $\eta$ then $\pi^*\alpha$ %, which defines $\pi^*\eta$, 
is nowhere vanishing. 
Though, $\pi^*\eta$ is not a contact structure, because at each point $\phat$ of $\Mhat$ we have $\pi^*(\alpha\wedge d\alpha^{n})_{\vert \phat}=0$.
Nonetheless, the restriction of $\pi^*\eta$ to $\Mhat$ is a honest contact structure on $\Mhat$. 
We then want to show that $\pi^*\eta$ gives a ``natural'' way to construct contact structures on $\Vhat$.
\\ 

We start by considering a more general setting. 
Let $Y^{2n+1}$ be a smooth manifold, $Z^{2n-1}$ a codimension-$2$ submanifold and $\eta$ a hyperplane field on $Y$.

\begin{definition}
	\label{DefAdaptConf}
	We say that $\eta$ is \emph{adjusted} to $Z$ if it is a contact structure away from $Z$ and $\eta\cap TZ$ is a contact structure on $Z$.
	If that's the case, we also call \emph{contactization} of $\eta$ a contact structure $\xi$ such that there is a smooth path $\{\eta_s\}_{s\in[0,1]}$ of hyperplane fields, all adjusted to $Z$, starting at $\eta_0=\eta$ and ending at $\eta_1=\xi$, such that $\eta_s$ is a contact structure for all $s\in(0,1]$.
\end{definition}

\begin{prop}
	\label{PropExistUniqContactization}
	Let $\eta$ be a hyperplane field on $Y$ adjusted to $Z$. Contactizations of $\eta$ exist and are all isotopic.
\end{prop}

Recall from Eliashberg--Thurston \cite[Section 1.1.6]{EliThu98} that a \emph{confoliation} is a hyperplane field $\zeta=\ker\alpha$ that admits a complex structure $J:\zeta\rightarrow\zeta$ \emph{tamed} by $d\alpha\vert_{\zeta}$, i.e. such that $d\alpha(X,JX)\geq0$ for all vector fields $X$ tangent to $\zeta$.
\\
We point out that, in our situation we can talk directly about confoliations adjusted to a certain codimension $2$ submanifold. Indeed, if $\eta$ is a hyperplane field on $Y$ adjusted to a $2$-codimensional submanifold $Z$, then $\eta$ is in particular a confoliation.
This follows from \Cref{PropExistUniqContactization} and the following:
\begin{fact}
	\label{RmkHyperplaneFieldConfol1}
	Let $(\eta_n)_{n\in\nat}$ be a sequence of contact structures on a compact manifold $Y^{2n+1}$ which $C^1$-converges to a hyperplane field $\eta$ on $Y$.
	Then, $\eta=\ker\alpha$ admits a complex structure $J$ tamed by $d\alpha\vert_{\eta}$.
\end{fact}

\begin{proof}[Idea of proof (\Cref{RmkHyperplaneFieldConfol1})]
	A first attempt could be to take, for each $k\in\nat$, a complex structure $J_k$ on $\eta_k=\ker\alpha_k$ tamed by $d\alpha_k\vert_{\eta_k}$ (which exists because $\eta_k$ is a contact structure) and to define $J$ as ``the limit'' of the sequence $(J_k)_{k\in\nat}$. 
	However, such a limit does not necessarily exist for a general choice of $J_k$.
	\\
	The solution is hence to ensure the orthogonality of each of the $J_k$ with respect to an auxiliary Riemannian metric $g$, using the polar decomposition of matrices. 
	By the compactness of the space of vector bundle isomorphisms of $TY$ preserving the metric $g$, one can now find a subsequence $(J_{k_{j}})_{j\in\nat}$ converging to a certain $J$, which is hence a complex structure on $\eta$ tamed by $d\alpha\vert_\eta$.
\end{proof}

\Cref{PropExistUniqContactization} is a consequence of the following lemma, which deals with the more general situation of any number of parameters:	
\begin{lemma}
	\label{LemmaExistUniqContactizationInFamilies}
	Given $K$ a compact set and $\left(\eta_k\right)_{k\in K}$ a smooth $K$-family of confoliations on $V$ adjusted to $M$, there is a smooth family of confoliations $\left(\eta^s_k\right)_{s\in [0,1],\,k\in K}$ such that $\left(\eta^s_k\right)_{s\in [0,1]}$ is contactization of $\eta_k$, for each $k\in K$. 
	Moreover, if $\eta_k$ is contact for all $k$ in a closed subset $H\subset K$, then $\eta^s_k$ can be chosen so that $\eta^s_k=\eta_k$ for all $k\in H$ and $s\in[0,1]$. 
\end{lemma}

\begin{proof}[Proof (Proposition \ref{PropExistUniqContactization})]
	The existence of contactizations follows directly from \Cref{LemmaExistUniqContactizationInFamilies} with $K$ a point. 
	We then prove their uniqueness up to isotopy.
	\\
	Given two contactizations $\xi,\xi'$ of $\eta$, we have by definition two associated paths of adjusted confoliations $\eta_t,\eta'_t$, with $t\in[0,1]$, such that $\eta_0=\eta'_0=\eta$, $\eta_1=\xi$, $\eta'_1=\xi'$ and $\eta_t,\eta'_t$ contact for $t\in(0,1]$. Then, the path
	\begin{equation}
	t \mapsto \widehat{\eta}_t\coeq
	\begin{cases}
	\eta_{1-2t} & \text{if } t\in[0,\sfrac{1}{2}] \\
	\eta'_{2t-1}  & \text{if } t\in[\sfrac{1}{2},1]
	\end{cases}
	\end{equation}
	is a continuous path of adjusted confoliations from $\widehat{\eta}_0=\xi$ to $\widehat{\eta}_1=\xi'$. 
	Moreover, up to perturbing it smoothly at $t=\sfrac{1}{2}$, we can suppose that $\widehat{\eta}_t$ is smooth in $t$.
	Then, applying \Cref{LemmaExistUniqContactizationInFamilies} to $\widehat{\eta}_t$, with $K=[0,1]$ and $H=\{0,1\}$, we get a family $\left(\widehat{\eta}^s_t\right)_{s\in [0,1],\,t\in [0,1]}$ of adjusted confoliations such that $\widehat{\eta}^s_0=\xi$, $\widehat{\eta}^s_1=\xi'$ for all $s\in[0,1]$ and such that $\widehat{\eta}^s_t$ is contact for $s>0$. 
	The subfamily $\widehat{\eta}^1_t$ is then a path of contact structures from $\xi$ to $\xi'$, and it can be turned into an isotopy by Gray's theorem.
\end{proof}

\begin{proof}[Proof (Lemma \ref{LemmaExistUniqContactizationInFamilies})]
	This proof follows almost step by step the construction and the computations made in Geiges \cite[Section 2]{Gei97}.
	
	Because of the $C^1-$openness of the contact condition, there is an open subset $U$ of $K$ which contains $H$ and such that $\xi_k$ is contact for all $k\in U$. 
	We then consider a smooth cut-off function $\rho: K \rightarrow [0,1]$, equal to $0$ on $H$ and equal to $1$ on the complement of $U$.
	\\
	Take now an auxiliary Riemannian metric on $V$ and consider the circle bundle $S\left(\normalbundle{M}\right)$ given by the vectors of norm $1$ in the normal bundle $\normalbundle{M}$ of $M$ inside $V$.
	Let $\gamma$ be a connection form on $S\left(\normalbundle{M}\right)$, i.e. a nowhere vanishing $1-$form defining a hyperplane field which is transversal to the fibers of the fibration $S\left(\normalbundle{M}\right)\rightarrow M$. 
	Using the natural retraction $\real^2\setminus\{0\}\rightarrow\cercle$, $\gamma$ can also be seen as a $1-$form on $\normalbundle{M}\setminus M$.
	Moreover, the form $r^2\gamma$, where $r$ is the radial coordinate in $\normalbundle{M}\setminus M$, extends smoothly to $\normalbundle{M}$.  
	\\
	We consider then a non-increasing cut-off smooth function $g=g(r)$ which is $1$ near $r=0$ and vanishes for $r>1$
	and we identify $\normalbundle{M}$ with a neighborhood of $M$ inside $V$. If $\alphak$ is a smooth $K$-family of $1$-forms defining $\xik$, set 
	$$ \alphaks \, := \, \alphak \, + \, s \epsilon \rho(k) g(r) r^2 \gamma \text{ .} $$
	Here $\epsilon$ is a positive real constant which will be chosen very small later.  Suppose, without loss of generality, that $\epsilon\leq 1$. 
	Remark that $\xiks:=\alphaks$ is a well defined hyperplane field. Moreover, it is adjusted to $M$, for all values of $s,k$.
	
	We then need to show that, for an $\epsilon$ small enough, $\xiks$ is actually a contact structure on $V$ for all $s>0$, $k\in K$.
	We can compute
	\begin{align*} 
	\alphaks \wedge \left(d \alphaks\right)^{n} \,   = & \,\,
	\alphak \wedge \left(d \alphak\right)^n \, + \, \\
	&  + \, n s \epsilon \left[r g'\left(r\right) + 2 g\left(r\right)\right] \rho \left( k \right) \alphak \wedge \left(d \alphak\right)^{n-1} \wedge r dr \wedge \gamma \,\, + \,  \\
	&  + \, s \epsilon r^2 g\left(r\right) \rho \left(k\right) h \, \vol
	\end{align*}
	where $\vol$ is the Riemannian volume form on $V$ and $h$ is a function of $p\in V$, $k\in K$, $s\in [0,1]$, $\epsilon\in \real_{+}$ and is polynomial in $\epsilon$.
	\\
	Consider the smooth functions $P_k,Q_k\co V \rightarrow\real$ such that $\alphak \wedge \left(d \alphak\right)^n = P_k\, \vol$ and $n\left[r g'\left(r\right) + 2 g\left(r\right)\right] \alphak \wedge \left(d \alphak\right)^{n-1}\wedge rdr\wedge \gamma = Q_k \,\vol$. 
	Let also $R_k(\epsilon)\coeq r^2 g\left(r\right) h(\epsilon,k)$. 
	Then, $$\alphaks \wedge \left(d \alphaks\right)^{n} = \left\{P_k+s\epsilon \rho\left(k\right)\left[Q_k+R_k\left(\epsilon\right)\right]\right\}\, \vol\text{ .} $$
	\noindent
	Now, $Q_k>0$ and $R_k(\epsilon)=0$ along $\Mhat$, for all $k\in K$ and $\epsilon\in [0,1]$ (remark we allow here $\epsilon=0$). Hence, by compactness of $\Mhat$ and $[0,1]$, there is an open neighborhood $\mathcal{O}$ of $\Mhat$ inside $\Vhat$ such that $Q_k+R_k(\epsilon)>0$ on $\mathcal{O}$ for all $\epsilon\in[0,1]$.
	\\
	$P_k$ is independent of $\epsilon,s$ and is non-negative everywhere on $\Vhat$ for all $k$. 
	Moreover, $P_k$ is positive on the complement of $\mathcal{O}$ for all $k\in K$, and even on all $\Vhat$ if $k\in U \subset K$ (remember $\xi_k$ is contact if $k\in U$). \\
	Then, $P_k+s\epsilon \rho\left(k\right)\left[Q_k+R_k\left(\epsilon\right)\right]>0$ on $\mathcal{O}$, for all $k\in K$ and all $\epsilon\in(0,1]$.
	\noindent
	Finally, for $\epsilon$ very small, $P_k$ dominates $s\epsilon \rho (k)\left[Q_k+R_k\left(\epsilon\right)\right]$ wherever it is positive, because the latter is bounded above in norm (recall we are working with $\epsilon\leq 1$). Hence, by compactness of $\Vhat\setminus \mathcal{O}$, $P_k+s\epsilon \rho\left(k\right)\left[Q_k+R_k\left(\epsilon\right)\right]$ is also positive on the complement of $\mathcal{O}$ for all $k\in K$, for $\epsilon>0$ small enough.
\end{proof}

Coming back to the specific case of branched coverings, the hyperplane field $\pi^*\eta$ on $\Vhat$ is adjusted to $\Mhat$ (and is then in particular a confoliation).
\begin{definition}
	\label{DefContPullBack}
	We say that a contact structure on $\Vhat$ is a \emph{contact branched covering of $\eta$} if it is a contactization of $\pi^*\eta$ and it is invariant under all the diffeomorphisms of $\Vhat$ covering the identity of $V$. 
\end{definition}

We point out that, by definition of contactization, if $\etahat$ is a contact branched covering of $\eta$, the upstairs branching locus $\Mhat$ is naturally a contact submanifold in $(\Vhat,\etahat)$.
Then, \Cref{PropExistUniqContactization} easily implies the following:
\begin{prop}
	Let $\Vhat\rightarrow V$ be a smooth branched covering and $\eta$ a contact structure on $V$. Then, contact branched coverings of $\eta$ on $\Vhat$ exist and are all isotopic (among contact branched coverings).
\end{prop}
We point out that, in order to deduce this result from \Cref{PropExistUniqContactization}, the contactization in the statement \Cref{PropExistUniqContactization} has to be invariant under deck transformations of $\pi$, as requested in \Cref{DefContPullBack}, and the isotopy has to be among invariant contactizations. 
From the explicit formula in the proof of \Cref{LemmaExistUniqContactizationInFamilies} above, it's clear that both these conditions can be easily arranged.  

Remark also that \Cref{PropIntroDefContBranchCov} stated in the introduction is a simple consequence of Gray's theorem and the fact that contact branched coverings exist and are unique up to isotopy.
Indeed, the $[0,1]-$families of hyperplane fields in points \ref{Point1PropIntroDefContBranchCov} and \ref{Point2PropIntroDefContBranchCov} in the statement of \Cref{PropIntroDefContBranchCov} are automatically adjusted to the upstairs branching locus for small parameters $t\geq 0$.

\subsection{Effects of branched coverings on weak fillings}
\label{SubSecEffBranchCovWeakFill}

We will use in this section the notion of \emph{weak fillability} introduced in Massot--Niederkr\"uger--Wendl \cite{MNW13}, in the following computation-friendly form: 
\begin{definition}[{\cite{MNW13}}]
	\label{DefWeakFill}
	We say that \emph{$(W,\omega)$ weakly fills $(V,\eta)$}, or that \emph{$\omega$ weakly dominates $\xi$}, if, for one (hence every) $1$-form $\alpha$ defining $\eta$, $\alpha\wedge\left(\omega + \tau d\alpha\right)^{n}$ is a positive volume form on $V$ for all $\tau\geq0$.
\end{definition}

Consider now a branched covering $\pi:\What^{2n+2}\rightarrow W^{2n+2}$ of even dimensional manifolds with non-empty boundaries $\Vhat^{2n+1}=\partial \What$ and $V^{2n+1} = \partial W$. 
Let also $\Xhat^{2n}$ be the upstairs branching set, $X$ the downstairs branch set, $M,\Mhat$ the boundaries of $X,\Xhat$ respectively and $\pi'$ the restriction $\pi\vert_{\Vhat}:\Vhat\rightarrow V$.
Here's a more detailed version of \Cref{ThmIntroWeakFillBranchCov} from \Cref{SecIntro}:
\begin{thm}
	\label{PropBranchCovPresWeakFill}
	Suppose we are in the following situation:
	\begin{enumerate}[label=(\alph*)]
		\item\label{HypPropBranchCov2} $\eta$ is a contact structure on $V$ and $\xi\coeq \eta\cap TM$ is contact on $M$;
		\item\label{HypPropBranchCov3} $\etahat$ on $\Vhat$ is a contact branched covering of $(V,\eta)$; 
		\item\label{HypPropBranchCov4} $\omega$ on $W$ weakly dominates $\eta$ on $V$; 
		\item\label{HypPropBranchCov5} $X$ is a symplectic submanifold of $(W,\omega)$ and it weakly fills $(M,\xi)$.
	\end{enumerate}
	\noindent
	Then, $\What$ admits a symplectic form $\omegahat$ that weakly dominates $\etahat$ on $\Vhat$. 
\end{thm}

Notice that, because $\pi'\vert_{\Mhat}\co \Mhat \rightarrow M$ is a (unbranched) covering map, $\xihat \coeq \left(\pi'\vert_{\Mhat}\right)^* \xi=\etahat\cap T\Mhat$ is a contact structure on $\Mhat$. 

\begin{proof}
	Consider the normal bundle of $\Xhat$ inside $\What$ and view it as a neighborhood $\calUhat$ of $\Xhat$. 
	Similarly for a neighborhood $\calOhat$ of $\Mhat$ in $\Vhat$. 
	In particular, we have a norm function on $\calUhat$ and $\calOhat$, and we can denote by $\calUhat_r,\calOhat_r$ the set of vectors of norm less than $r$.
	
	Fix now an arbitrary smooth function $f:\What\rightarrow \real_{\geq 0}$, compactly supported in $\calUhat_{1}$, depending only on $r$, non-increasing in it, and equal to $1$ on a neighborhood of $\Xhat$.
	Denote also by $g$ its restriction to $\Vhat = \partial \What$. 
	Notice that in particular $f'(r)=0$, hence $g'(r)=0$, for $r=0$. 
	
	Let now $\delta$ be a connection $1$-form on the circle bundle $S \calUhat$ given by the vectors of norm $1$ in $\calUhat$. Denote also by $\gamma$ the restriction of $\delta$ to the sub-bundle $S \calOhat$ given by the vectors of norm $1$ in $\calOhat$.
	Notice that $\gamma$ is in particular a connection form on $S\calOhat$.
	The explicit formula in the proof of \Cref{LemmaExistUniqContactizationInFamilies} then shows that, up to isotopy, we can assume that the contact branched covering $\etahat$ is the kernel of $\alphahat_\epsilon:=\pi^*\alpha+  \epsilon g(r) r^2 \gamma$, for every $\epsilon$ smaller than or equal to a certain constant $\epsilon_0>0$. 
	
	As far as the symplectic structure on $\What$ is concerned, consider the closed $2$-form $\omegahateps:=\piomega+\epsilon \, d\left(f(r) r^2 \delta\right)$ on $\What$, where $\epsilon>0$.
	\begin{claim}
		\label{LemmaOmegaEpsSymplStr}
		There is $\epsilon_1>0$ such that $\omegahateps$ is symplectic on $\What$ for all $0<\epsilon<\epsilon_1$.
	\end{claim} 
	\begin{proof}[Proof (\Cref{LemmaOmegaEpsSymplStr})]
		We have $\omegahateps = \pi^*\omega +  \epsilon \left(2f + r f'\right) r dr \wedge \delta + \epsilon f r^2 d\delta$,
		%	\begin{equation*}
		% 	\omegahateps = \pi^*\omega +  \epsilon \left(2f + r f'\right) r dr \wedge \delta + \epsilon f r^2 d\delta \text{ ,}
		% 	\end{equation*}
		so that
		\begin{align*}
		\omegahateps^{n+1} =& \left[\pi^*\omega + \epsilon \left(2f + r f'\right) r dr \wedge \delta + \epsilon f r^2 d\delta\right]^{n+1} \\
		=& \,\, \pi^* \omega^{n+1} + \left(n+1\right) \epsilon \left(2f+rf'\right) \pi^* \omega^n \wedge r dr \wedge \delta \\
		& + \epsilon r^2 f h \vol \text{ ,} 
		%=& \pi^* \omega^{n+1} + n \epsilon \left(2f+rf'\right) \pi^* \omega^n \wedge r dr \wedge \delta \\
		%& + \left(n+1\right)\epsilon f r^2 d\delta\wedge\pi^* \omega^n + \epsilon^2 r^2 f h \vol \text{ ,} 
		\end{align*} 	
		where $\vol$ is a volume form on $W$ and $h$ is a smooth function depending on $p\in\What$ and on $\epsilon>0$. 
		Using that $\pi^*\omega$ is symplectic on the complement of $\Xhat$ and that the restriction of $\omega$ to $X$ is symplectic on $X$, we can then conclude, as we did in the proof of \Cref{LemmaExistUniqContactizationInFamilies}, that $\omegahateps^{n+1}>0$ for $\epsilon$ small enough.
	\end{proof}
	
	We then want to show that $\omegahateps$ weakly dominates $\etahat= \ker(\alphahateps)$, provided that $\epsilon>0$ is small enough (and in particular such that $\epsilon<\overline{\epsilon}:=\min\left(\epsilon_0,\epsilon_1\right)$).
	In other words, we need to check that, if $\epsilon$ is small enough, the following is satisfied:  
	\begin{equation*}
	\alphahateps\wedge\left(\omegaVhateps + \tau d\alphahateps\right)^n>0\, , \;\;\,\forall\tau\geq0\text{ ,}
	\end{equation*}
	where $\omegaVhateps$ denotes the pullback of $\omegahateps$ via the inclusion $\Vhat\hookrightarrow \What$, i.e.
	\begin{equation*}
	\omegaVhateps \, = \, \pi^* \omegaV + \epsilon d\left(g r^2 \gamma\right) \, =\, \pi^* \omegaV + \epsilon \left(2g + r g'\right) rdr\wedge\gamma + \epsilon g r^2 d\gamma \text{ .}
	\end{equation*}
	\noindent
	Using that $d \alphahateps  = \pi^* \alpha + \epsilon \left(2g+rg'\right) r dr \wedge \gamma + \epsilon r^2 g d\gamma$, we can compute
	\begin{align*}
	\alphahateps \wedge & \left(\omegaVhat + \tau d\alphahateps\right)^n \, \\
	= & \,\, \left(\pi^* \alpha + \epsilon g r^2 \gamma \right) \wedge \left[\pi^*\omegaV + \tau \pi^* d\alpha  \right. \\
	& \, + \left.  \epsilon \left(1+\tau\right) \left(rg' + 2 g\right) rdr\wedge \gamma + \epsilon \left(1+\tau\right) g r^2 d\gamma\right]^{n} \\
	= &  \,\, \pi^*\left[\alpha\wedge \left(\omegaV+\tau d\alpha\right)^n\right] \, \\ 
	& \, + n \epsilon \left(1+\tau\right) \left(rg' + 2 g\right) \pi^*\left[\alpha\wedge \left(\omegaV+\tau d\alpha\right)^{n-1}\right]\wedge rdr\wedge \gamma \\ 
	& \, + \epsilon g r^2 h \vol \text{ ,} 
	\end{align*}
	where $\vol$ is a volume form on $\Vhat$ and $h$ is a smooth function of $\widehat{p}\in\Vhat$, $\epsilon$ and $\tau$, which is moreover polynomial in $\epsilon$ and in $\tau$, with $\deg_\tau h \leq n$.
	
	Denote now by $P_0(\tau)$ and $P_1(\tau)$ the polynomials in $\tau$, with coefficients in the ring of functions $\Vhat\rightarrow \real$, defined respectively by the identities 
	\begin{align*}
	P_0(\tau) \, \vol & = \pi^*\left[\alpha\wedge \left(\omegaV+\tau d\alpha\right)^n\right] \text{ ,} \\
	P_1(\tau) \, \vol & = n \left(1+\tau\right) \left(rg' + 2 g\right) \pi^*\left[\alpha\wedge \left(\omegaV+\tau d\alpha\right)^{n-1}\right]\wedge rdr\wedge \gamma \text{ .}
	\end{align*}
	Similarly, denote by $P_2(\tau,\epsilon)$ the polynomial in $\tau$ and $\epsilon$ given by $P_2(\tau,\epsilon) = g r^2 h$.
	
	\begin{claim}
		\label{LemmaPosNullP0}
		For all $\tau\geq0$, $P_0(\tau)$ is non-negative everywhere on $\Vhat$ and positive away from $\Mhat$.
	\end{claim}
	\begin{proof}[Proof (\Cref{LemmaPosNullP0})]
		This follows from the fact that $(W,\omega)$ is a weak filling of $(V,\eta)$ and that $\pi\vert_{\Vhat}$ is a  branched cover with (upstairs) branching locus $\Mhat$.
	\end{proof}
	\begin{claim}
		\label{LemmaPosNearBranchLocus}
		There are constants $0<\epsilon'_0<\overline{\epsilon}$ and $r_0>0$, such that $P_1(\tau)+P_2(\tau,\epsilon)>0$ on $\calOhat_{r_{0}}$ %$\mathcal{M}_{r_{0}}$ 
		for all $0\leq \epsilon<\epsilon'_0$ and all $\tau \geq 0$.
	\end{claim}
	Notice that we allow $\epsilon=0$ in its statement.
	The proof of \Cref{LemmaPosNearBranchLocus} will follow after the end of the proof of \Cref{PropBranchCovPresWeakFill}. % of \Cref{LemmaPosNearBranchLocus}.
	\\
	According to Claims \ref{LemmaPosNullP0} and \ref{LemmaPosNearBranchLocus}, we have that $\alphahateps \wedge \left(\omegaVhat + \tau d\alphahateps\right)^n$ is a positive volume form on $\calOhat_{r_{0}}$, %$\mathcal{M}_{r_{0}}$ 
	for all $0<\epsilon<\epsilon'_0$ and all $\tau \geq 0$. (Notice that here $\epsilon\neq 0$.)
	Then, the following result, whose proof is also postponed, concludes the proof of \Cref{PropBranchCovPresWeakFill}:
	\begin{claim}
		\label{LemmaPosAwayBranchLocus}
		There is $0<\epsilon'_1<\epsilon'_0$ such that $P_0(\tau) + \epsilon \left[ P_1\left(\tau\right) + P_2\left(\tau,\epsilon\right) \right] >0 $ on the complement of 
		$\calOhat_{\sfrac{r_{0}}{2}}$, % $\mathcal{M}_{\sfrac{r_{0}}{2}}$, 
		for all $0\leq \epsilon<\epsilon'_1$ and all $\tau \geq 0$.
		\qedhere
	\end{claim}
\end{proof}

We now give a proof of Claims \ref{LemmaPosNearBranchLocus} and \ref{LemmaPosAwayBranchLocus} above. 
They are corollaries of the following fact, whose proof is easy and omitted:
\begin{fact}
	\label{LemmaMinPoly}
	Consider a smooth manifold $S$ and a continuous function $p:S\times\real_{\geq0}\rightarrow \real$ such that, for each $s\in S$, $p_s:\real_{\geq0}\rightarrow\real$ defined by $p_s(\tau):=p(s,\tau)$ is polynomial in $\tau$. 
	Suppose there is $s_0\in S$ and a neighborhood $U$ of $s_0$ such that for all $s\in U$ the followings are satisfied:
	\begin{enumerate}
		\item \label{Item1LemmaMinPoly} $\deg_{\tau}\left(p_{s_{0}}\right) \geq \deg_{\tau}\left(p_{s}\right)$;
		\item \label{Item2LemmaMinPoly} the leading coefficient of $p_{s_0}$ is positive.
	\end{enumerate}
	Then, there is a neighborhood $O$ of $s_0$ contained in $U$ such that, for all $s\in O$, the minimum $m_s$ of $p_s$ exists and it depends continuously on $s$.
	In particular, if moreover $m_{s_{0}}>0$, then $m_s>0$ for $s$ sufficiently near to $s_0$. 
\end{fact}

\begin{proof}[Proof (\Cref{LemmaPosNearBranchLocus})]
	We would like to use \Cref{LemmaMinPoly}, with $S := \Vhat \times [0,\overline{\epsilon})$ and $P:=P_1+P_2 : S \times \real_{\geq 0}\rightarrow \real$, i.e. $P_{q,\epsilon}(\tau)$ is given by $\left[P_1\left(\tau\right) + P_2\left(\tau,\epsilon\right)\right]\left(q\right)$ for $(q,\epsilon)\in S = \Vhat \times [0,\overline{\epsilon})$; notice that we allow $\epsilon=0$ here.\\
	Consider the compact set $K := \Mhat \times \{0\}$ in $S$. If $(q,0)\in K$, then
	\begin{align*}
	P_{\left(q,0\right)}\cdot \vol_{\left(q,0\right)} \,& =\, \left[P_1\left(\tau\right)_{q}+P_2\left(\tau,0\right)_{q}\right] \, \vol_{\left(q,0\right)} \\
	& = \, P_1 \left(\tau\right)_{q} \, \vol_{\left(q,0\right)} \\
	& = \,2n\left(1+\tau\right)\left\{\pi^*\left[\alpha\wedge\left(\omegaV+\tau d\alpha\right)^{n-1}\right]\wedge rdr\wedge \gamma\right\}_{q} \text{ ,}
	\end{align*}
	which is positive because the restriction of $\omega$ to $X$ weakly dominates $\xi$ on $M=\partial X$. 
	Thus, for $(q,0)\in K$, $P_{\left(q,0\right)}$ has positive leading coefficient and $m_{\left(q,0\right)}>0$.
	Moreover, for each $(q,0)\in K$, $\deg_{\tau}\left(P_{\left(q,0\right)}\right)= n \geq \deg_{\tau}\left(P_{s}\right)$ for all $s\in S = \Vhat \times [0,\overline{\epsilon})$.
	One can then apply \Cref{LemmaMinPoly}, which, by compactness of $K$, tells that there is a neighborhood $\mathcal{U}$ of $K$ in $S$ such that $m_s$ exists and is positive for all $s\in \mathcal{U}$. 
	Now, $\mathcal{U}$ contains an open set of the form $\{r < r_0 , \epsilon < \epsilon'_0\}\subset S = \Vhat \times [0,\overline{\epsilon})$, which concludes.
\end{proof}

\begin{proof}[Proof (\Cref{LemmaPosAwayBranchLocus})]
	We use again \Cref{LemmaMinPoly}. Here, $S := %\mathcal{M}_{\sfrac{r_{0}}{2}}^{c} \times [0,\epsilon'_0)$,
	\calOhat_{\sfrac{r_{0}}{2}}^{c} \times [0,\epsilon'_0)$, where %$\mathcal{M}_{\sfrac{r_{0}}{2}}^{c}$ 
	$\calOhat_{\sfrac{r_{0}}{2}}^{c}$ is the complement of %$\mathcal{M}_{\sfrac{r_{0}}{2}}$
	$\calOhat_{\sfrac{r_{0}}{2}}$ in $\Vhat$ and $r_0,\epsilon'_0$ are given by \Cref{LemmaPosNearBranchLocus}. 
	Also, $P : S \times \real_{\geq 0}\rightarrow \real$ is here defined as
	\begin{equation*}
	P_{\left(p,\epsilon\right)}(\tau) = P_0(\tau)\vert_{p} + \epsilon\left[P_1\left(\tau\right) + P_2\left(\tau,\epsilon\right)\right]\left(p\right)
	\end{equation*} 
	for $(p,\epsilon)\in S$. 
	Notice that once again we allow $\epsilon=0$.
	\\
	Then, if $K:=\calOhat_{\sfrac{r_{0}}{2}}^{c} \times \{0\}$, 
	% $K:=\mathcal{M}_{\sfrac{r_{0}}{2}}^{c} \times \{0\}$,
	$P_{(q,0)} = P_1(\tau)_{q}$ for all $(q,0)\in K$, hence it is positive by \Cref{LemmaPosNullP0}. 
	In particular, $P_{(q,0)}$ has positive leading coefficient and positive minimum $m_{\left(q,0\right)}$ for all $(q,0)\in K$.
	Moreover, $\deg_{\tau}\left(P_{\left(q,0\right)}\right)= n \geq \deg_{\tau}\left(P_{(p,\epsilon)}\right)$, 
	% for all $q,p \in \mathcal{M}_{\sfrac{r_{0}}{2}}^{c}$ 
	for all $q,p \in \calOhat_{\sfrac{r_{0}}{2}}^{c}$ 
	and $\epsilon \in [0,\epsilon'_0)$.
	\Cref{LemmaMinPoly} then implies, by compactness of $K$, that $P_{\left(p,\epsilon\right)}$ admits a minimum $m_{\left(p,\epsilon\right)}$, which is moreover positive in a neighborhood of $K$.
\end{proof}

%%%%%%%%%%%%%%%%%%%%%%%%%%%%%%%%%%%%%%%%%%%%%%%%%%%%%%%%%%%%%%%%%%%%%%%%%%%%%%%%%%%%%%%%%%%%%%%%%%%%%%%%%%%%%%%%%%%%%%%%%%%%%%%%%%%%%%%%%%%%%%%%%%%%%%%%%%%%%%%%%%%%%%%%%%%%%%%%%%%%%%%%%%%%%%%%%%%%%%%%%%%%%%%%%%%%%%%%%%%%%%%%%%%%%%%%%%%%%%%%%%%%%%%%%%%%

\section{Open books and contact vector fields}
\label{SecOBDContVF}

In this section we prove the reinterpretation of adapted open book decompositions in terms of contact vector fields described in \Cref{ThmOBDCoupleContVF}. 
A part of this result has been stated by Giroux during the Yashafest in June 2007 and the AIM workshop in May 2012; see Giroux \cite[Claim on page 19]{GirouxAIM}.
\\
More precisely, in \Cref{SubSecOBDToContactVF} we describe how to obtain a pair of contact vector fields with Lie bracket everywhere transverse to $\xi$ from the data of an open book decomposition supporting a contact structure. 
This is the part of \Cref{ThmOBDCoupleContVF} that has already been stated in \cite{GirouxAIM}. 
\Cref{SubSecContVFToOBD} deals with the converse, i.e.\ contains the proof of the fact that it is also possible to recover a supporting open book from such a pair of contact vector fields.

%In \Cref{SubSecContVFToOBD} we prove the first part of \Cref{ThmOBDCoupleContVF} stated in \Cref{SecIntro},
%i.e. we describe how to recover the data of an open book decomposition adapted to a contact structure $\xi$ from the data of two contact vector fields with Lie bracket everywhere transverse to $\xi$. 
%In \Cref{SubSecOBDToContactVF}, we show that it is also possible to recover such a couple of contact vector fields from an adapted open book, as claimed in \cite{GirouxAIM}.

%%%%%%%%%%%%%%%%%%%%%%%%%%%%%%%%%%%%%%%%%%%%%%%%%%%%%%%%%%%%%%%%%%
%%%%%%%%%%%%%%%%%%%%%%%%%%%%%%%%%%%%%%%%%%%%%%%%%%%%%%%%%%%%%%%%%%

\subsection{From open books to contact vector fields}
\label{SubSecOBDToContactVF}

We have the following more precise version of the first part of \Cref{ThmOBDCoupleContVF}:
\begin{prop}[stated in Giroux \cite{GirouxAIM}]
	\label{PropGiroux2}
	Let $(B,\varphi)$ be an open book decomposition of $M^{2n-1}$ supporting $\xi$. 
	Denote by $\alpha$ a contact form defining $\xi$ such that $d\alpha$ is symplectic on the fibers of $\varphi$. 
	Then, there is a smooth function $\phi\co M \rightarrow \real^2$ defining $(B,\varphi)$ such that the contact vector fields $X$ and $Y$, associated via $\alpha$ respectively to the contact Hamiltonians $\phione$ and $-\phitwo$, have Lie bracket $[X,Y]$ negatively transverse to $\xi$.
\end{prop}
Recall from Giroux \cite{Gir02} that an open book decomposition $(B,\varphi)$ on $M$ is said to support a contact structure $\xi$ if $B$ is a positive contact submanifold and there is a defining $1$--form $\alpha$ for $\xi$ such that $d\alpha$ is positively symplectic on the fibers of $\varphi\co M\setminus B \to \cercle$.
In the statement of \Cref{PropGiroux2} above, by \emph{``$\phi\co M \rightarrow \real^2$ defining $(B,\varphi)$''} we mean that $\phi$ is transverse to $0\in\real^2$, that $B=\phi^{-1}(0)$, and that $\sfrac{\phi}{\norm{\phi}}\co M\setminus\phi^{-1}(0)\rightarrow\cercle$ coincides with $\varphi$. % (so it is in particular a submersion).

\begin{proof}[Proof (\Cref{PropGiroux2})]
	%This proof is strongly inspired from the computations in \cite{Bou02}, which will be recalled in detail in \Cref{SubSecBourgConstr}.
	%
	Let $\overline{\phi}=(\overline{\phi}_1,\overline{\phi}_2)\co M \rightarrow\real^2$ be a smooth function defining $(B,\varphi)$.
	Consider then $\epsilon>0$ such that $\alpha\wedge d\alpha^{n-2}\wedge d\overline{\phione} \wedge d\overline{\phi}_2$ is positive on $\{\norm{\overline{\phi}}<\epsilon\}$. 
	Such an $\epsilon$ exists because $\alpha$ induces a contact form on $B = \phi^{-1}(0)$.
	\\
	%	Define now $\chi\co \real_{\geq 0} \rightarrow \real_{\geq 0}$ as follows: $\chi(x)$ is non-decreasing in $x$, is equal to $x$ for $x<\sfrac{\epsilon}{2}$ and equal to $1$ for $x\geq\epsilon$.
	%	Denote then $f\coeq \sfrac{\chi(\norm{\phi})}{\norm{\phi}}\co M \rightarrow \real_{+}$ and define $\phi\coeq f\overline{\phi} \co M\rightarrow\real^2$; then, $\phi$ defines $(B,\varphi)$ too.
	Consider now a smooth function $f\co M \rightarrow \real_{>0}$, depending only on $\norm{\phi}$ in a non-decreasing way, equal to $1$ for $\norm{\phi}<\sfrac{\epsilon}{2}$ and equal to $\sfrac{1}{\norm{\phi}}$ for $\norm{\phi}>\epsilon$.
	Let then $\phi\coeq f\overline{\phi} \co M\rightarrow\real^2$; in particular, $\phi$ defines $(B,\varphi)$ too.
	Consider also $\rho\coeq \norm{\phi}$ and $\theta \coeq \sfrac{\phi}{\rho}  \co M\setminus B \rightarrow \cercle$, and	notice that $\theta=\varphi$.
	We claim that 
	$$ \Omega  \coeq n\rho^2 d\theta \wedge d\alpha^{n-1} + n(n-1)\rho d\rho \wedge d\theta\wedge\alpha\wedge d\alpha^{n-2} $$
	is a volume form on $M$.
	Indeed, the first term is non-negative everywhere and positive away from $B$, because $d\alpha$ is symplectic on the fibers of $\theta=\varphi$, and the second term is positive along $B$ and non-negative everywhere, by choice of $f$. 
	
	We then denote by $X,Y$ the contact vector fields associated, respectively, to the contact hamiltonians $\phione,-\phitwo$ via the contact form $\alpha$ given in the statement.
	Because $\rho^2 d\theta = \phione d\phitwo - \phitwo d\phione$ and $\rho d\rho\wedge d\theta = d\phione \wedge d\phitwo$, we have
	\begin{multline}
	\label{EqOmega}
	\Omega \; = \; n\left[-\alphaX \dalphaY + \alphaY \dalphaX\right]\wedge d\alpha^{n-1} \\
	-\, n(n-1) \dalphaX\wedge \dalphaY \wedge \alpha\wedge d\alpha^{n-2} \text{ .}
	\end{multline}
	
	\noindent
	Notice now that $\alpha\wedge d\left(\alphaY\right)\wedge d\alpha^{n-1}=0 $ on $M$, because $\dim M = 2n-1$. %for dimensional reasons. 
	Hence, $\iota_{X}\left[\alpha\wedge d\left(\alphaY\right)\wedge d\alpha^{n-1}\right]=0$, which,
	using the graded Leibniz rule for the interior product, gives
	%\begin{multline*}
	%\alpha(X)\cdot d\left(\alphaY\right)\wedge d\alpha^{n-1} - d\left(\alphaY\right)(X)\cdot\alpha\wedge d\alpha^{n-1} \\
	% + (n-1) \alpha\wedge d\left(\alphaY\right) \wedge d\alpha (X,\, .) \wedge d\alpha^{n-2} = 0 \text{ ,}
	%\end{multline*}
	%i.e. 
	\begin{multline}
	\label{Eq2ProofLemma4PropGiroux1}
	%\begin{split}
	\alpha(X)\, d\left(\alphaY\right)\wedge d\alpha^{n-1} =  X\cdot\left(\alphaY\right)\,\alpha\wedge d\alpha^{n-1} \\ 
	- (n-1) \alpha\wedge d\left(\alphaY\right) \wedge d\alpha (X, \cdot) \wedge d\alpha^{n-2} \text{ .}
	%\end{split}
	\end{multline}
	(Here, we adopted the notation $Z\cdot f = df(Z)$ for a smooth function $f$ and a vector field $Z$.)
	Exchanging the roles of $X$ and $Y$ in \Cref{Eq2ProofLemma4PropGiroux1}, we also get
	\begin{multline}
	\label{Eq3ProofLemma4PropGiroux1}
	%\begin{split}
	\alpha(Y)\, d\left(\alphaX\right)\wedge d\alpha^{n-1} =  Y\cdot\left(\alphaX\right)\,\alpha\wedge d\alpha^{n-1} \\ 
	- (n-1) \alpha\wedge d\left(\alphaX\right) \wedge d\alpha (Y, \cdot) \wedge d\alpha^{n-2} \text{ .}
	%\end{split}
	\end{multline}
	As $X$ and $Y$ are contact vector fields for $\xi$, one also has %there are functions $f,g\co M \rightarrow \real$ such that 
	\begin{equation}
	\label{Eq4ProofLemma4PropGiroux1}
	d\left(\alphaX\right)\vert_\xi   \, =\,  - d\alpha (X,\, \cdot)\vert_\xi \;\text{ and } \;
	d\left(\alphaY\right)\vert_\xi   \, =\, - d\alpha (Y,\, \cdot)\vert_\xi \text{ .}
	\end{equation}
	Then, \Cref{EqOmega,Eq2ProofLemma4PropGiroux1,Eq3ProofLemma4PropGiroux1,Eq4ProofLemma4PropGiroux1} give %together with the fact that $\alpha\wedge\alpha = 0$, tell that
	\begin{multline}
	\label{Eq5ProofLemma4PropGiroux1}
	\Omega \, =  \, - \, nX\cdot\left(\alphaY\right)\,\alpha\wedge d\alpha^{n-1} 
	+ \, nY\cdot\left(\alphaX\right)\,\alpha\wedge d\alpha^{n-1} \\
	+ n(n-1) \alpha \wedge d\alpha(X,\,.)\wedge d\alpha(Y,\,.) \wedge d\alpha^{n-2} \text{ .}
	\end{multline}
	%	\begin{equation}
	%	\label{Eq5ProofLemma4PropGiroux1}
	%	\begin{split}
	%	\Omega \, = & \, - \, nX\cdot\left(\alphaY\right)\,\alpha\wedge d\alpha^{n-1} \\
	%	& + \, nY\cdot\left(\alphaX\right)\,\alpha\wedge d\alpha^{n-1} \\
	%	& + n(n-1) \alpha \wedge d\alpha(X,\,.)\wedge d\alpha(Y,\,.) \wedge d\alpha^{n-2} \text{ .}
	%	\end{split}
	%	\end{equation}
	Again for dimensional reasons, $d\alpha^{n}=0$ on $M$, so that $\iota_X\iota_Y d\alpha^n=0$, i.e.
	$$
	(n-1)d\alpha(X,\,.)\wedge d\alpha(Y,\,.)\wedge d\alpha^{n-2} \, =\, d\alpha(X,Y)\, d\alpha^{n-1} \text{ .}
	$$
	Then, \Cref{Eq5ProofLemma4PropGiroux1} finally becomes
	\begin{align*}
	\Omega \, = & \, - \, n\left[X\cdot\left(\alphaY\right) + Y\cdot\left(\alphaX\right) + d\alpha(X,Y)\right]\,\alpha\wedge d\alpha^{n-1} \\
	= & \, -n\alphaXY \, \alpha\wedge d\alpha^{n-1}\text{ .}
	\end{align*}
	%	\begin{equation*}
	%	\begin{split}
	%	\Omega \, = & \, - \, nX\cdot\left(\alphaY\right)\,\alpha\wedge d\alpha^{n-1} \\
	%	& + \,n Y\cdot\left(\alphaX\right)\,\alpha\wedge d\alpha^{n-1} \\
	%	& + nd\alpha(X,Y)\, \alpha \wedge d\alpha^{n-1} \\
	%	= & \, -n\alphaXY \, \alpha\wedge d\alpha^{n-1}\text{ .}
	%	\end{split}
	%	\end{equation*}
	As $\Omega$ is a volume form on $M$ and $\alpha$ is a positive contact form, $[X,Y]$ must then be negatively transverse to $\xi$.
\end{proof}

%%%%%%%%%%%%%%%%%%%%%%%%%%%%%%%%%%%%%%%%%%%%%%%%%%%%%%%%%%%%%%%%%%
%%%%%%%%%%%%%%%%%%%%%%%%%%%%%%%%%%%%%%%%%%%%%%%%%%%%%%%%%%%%%%%%%%

\subsection{From contact vector fields to open books}
\label{SubSecContVFToOBD}

We have the following converse to \Cref{PropGiroux2}:
\begin{prop}
	\label{PropGiroux1}
	Let $(M^{2n-1},\xi)$ be a closed contact manifold. Suppose $X$, $Y$ are two contact vector fields with Lie bracket $[X,Y]$ everywhere negatively transverse to $\xi$. Then, if we denote $\Xtheta:=\cos\theta \, X + \sin\theta \, Y$ and $\Ytheta:=X_{\theta+\pi/2}$ for $\theta\in\cercle$, we have the following:
	\begin{enumerate}[label=(\alph*)]
		\item\label{Item1PropGiroux1} The set $\Sigmatheta:=\{\Xtheta\in\xi\}$ is a non-empty regular hypersurface, which is moreover $\xi$--convex.
		\item\label{Item2PropGiroux1} For $\theta\neq\theta' \Mod{\pi}$, the intersection $K:=\Sigmatheta\cap\Sigmathetaprime$ is non-empty, transverse and doesn't depend on the choice of $\theta$, $\theta'$. 
		\item\label{Item3PropGiroux1}
		For each $\theta\in\cercle$, consider the set $$F_\theta:=\left\{\,p\in\Sigmatheta\, \vert\, \Ytheta(p)\text{ is positively transverse to }\xi_p \,\right\}\text{ ,}$$
		and define $\varphi\co M\setminus K \rightarrow\cercle$ as $\varphi(p)\coeq\theta$ if $p\in F_{-\theta}$. 
		Then, $(K,\varphi)$ is an open book decomposition of $M$, which is moreover adapted to $\xi$.  
	\end{enumerate}
\end{prop}

Recall from Giroux \cite{Gir91} that a hypersurface $\Sigma$ in $M$ is called $\xi$--convex if there is a vector field $Z$ which is contact for $\xi$ and transverse to $\Sigma$. 

The rest of \Cref{SubSecContVFToOBD} is devoted to the proof of the above result, which is a more detailed version of the second part of \Cref{ThmOBDCoupleContVF}. 
To improve readability, each claim in this section will be proved right after the conclusion of the part of the proof in which it is used.

Let $\alpha$ be a contact form for $\xi$ and denote by $f,g:M\rightarrow\real$ the smooth functions given by $\lieX\alpha=f\alpha$ and $\lieY\alpha=g\alpha$ respectively (these functions exist because $X$ and $Y$ are contact vector fields). 
For the proof of point \ref{Item3PropGiroux1} we will need to change this $\alpha$ conveniently. 
%We start by pointing out the following:

%%
%%Observe that the equality $\lieXtheta\alpha=\left[f\cos\theta +g\sin\theta \right]\alpha$ gives the following:
%%We also point out the following:
\begin{fact}
	\label{RmkBeforePropGiroux1}
	For all $\theta\in\cercle$, $X_\theta, Y_\theta$ are contact vector fields, and  $[X_\theta,Y_\theta]=[X,Y]$.
\end{fact}

%\begin{proof}[Proof (\Cref{RmkBeforePropGiroux1})]
%This follows from the fact that $\lieXtheta\alpha=\cos\theta\cdot\lieX\alpha+\sin\theta\cdot\lieY\alpha=\left[f\cos\theta +g\sin\theta \right]\alpha$ and from the fact that the Lie bracket is anti-symmetric and bilinear.
%\end{proof}

%\noindent
%We may now proceed to prove \Cref{PropGiroux1}.
\begin{proof}[Proof (\Cref{PropGiroux1}.\ref{Item1PropGiroux1})] 
	We start by proving that $\alphaXtheta$ is somewhere zero, i.e.\ that $\Sigmatheta=\{\alphaXtheta=0\}$ is non-empty.
	\\
	Suppose by contradiction this is not the case, i.e. $\alphaXtheta>0$ without loss of generality. 
	If we define $\beta\coeq\frac{1}{\alphaXtheta}\cdot \alpha$, then $\Xtheta=R_\beta$.
	By \Cref{RmkBeforePropGiroux1}, we have $\beta\left([X_\theta,Y_\theta]\right)= \beta\left([X,Y]\right)<0$. 
	On the other hand, we also have $[\Xtheta,\Ytheta]=[R_\beta, \Ytheta]$, so that 
	\begin{equation}
	\label{EqKNonEmpty}
	\begin{split}
	\beta\left([X_\theta,Y_\theta]\right) \,= & \; \beta([R_\beta,\Ytheta]) \, \\
	\overset{(i)}{=} & \, \left[ -d\beta(R_\beta,\Ytheta) \, + \, d\left(\beta\left(\Ytheta\right)\right)(R_\beta) \, - \, d\left(\beta\left(R_\beta\right)\right)(\Ytheta) \right] \,  \\
	\overset{(ii)}{=} & \, d\left(\beta\left(\Ytheta\right)\right) (R_\beta) \text{ .}
	\end{split}
	\end{equation}
	Here, for $(i)$ we used the fact that $\beta([R_\beta,\Ytheta])=-d\beta(R_\beta,\Ytheta) \, + \, d\left(\beta\left(\Ytheta\right)\right)(R_\beta) \, - \, d\left(\beta\left(R_\beta\right)\right)(\Ytheta)$ by the formula for the exterior derivative of differential forms, and for $(ii)$ we used that $d\beta(R_\beta,.)=0$ and $\beta(R_\beta)=1$.
	Now, $\beta\left(\Ytheta\right)$ is a function defined on the compact manifold $M$, hence it has at least one critical point. 
	This contradicts \Cref{EqKNonEmpty} and the fact that $\beta\left([X_\theta,Y_\theta]\right)<0$, thus proving that $\alphaXtheta$ is somewhere zero.
	
	In order to conclude the proof, it is then enough to show that
	\begin{equation}
	\label{EqDalphaLieBracket}
	\dalphaXtheta(\Ytheta)=-\alphaXY \text{ along } \Sigmatheta \text{ .}
	\end{equation}
	Indeed, this tells that $\alphaXtheta:M\rightarrow\real$ is transverse to $\{0\}\subset\real$, i.e.\ $\Sigmatheta$ is a smooth hypersurface, and that, more precisely, the contact vector field $\Ytheta$ is transverse to $\Sigmatheta$, i.e.\ the latter is $\xi-$convex.
	We then proceed to prove \Cref{EqDalphaLieBracket}.
	\\
	Using the formula for the exterior derivative, we compute 
	\begin{equation}
	\label{Eq1ProofGiroux1}
	\dalphaXYtheta =  d\left(\alphaYtheta\right)\left(\Xtheta\right)  - d\left(\alphaXtheta\right)\left(\Ytheta\right) - \alphaXYtheta \text{ .} 
	\end{equation}
	Also, by \Cref{RmkBeforePropGiroux1} there are $f_\theta,g_\theta\co M\rightarrow\real$ such that 
	\begin{equation}\label{Eq5ProofGiroux1}
	\ftheta\,\alpha=\lieXtheta\alpha=d\iXthetaalpha+\iXthetadalpha \;\text{ and } \;
	\gtheta\,\alpha=\lieYtheta\alpha=d\iYthetaalpha+\iYthetadalpha \text{ .}
	\end{equation}
	Now, evaluating these last two equations respectively on $\Ytheta$ and $\Xtheta$ gives
	\begin{equation}
	\label{Eq2ProofGiroux1}
	\begin{split} \dalphaXtheta(\Ytheta) &= \ftheta\,\alphaYtheta - \dalphaXYtheta \text{ ,} \\
	% \label{Eq3ProofGiroux1} 
	\dalphaYtheta(\Xtheta) &= \gtheta\,\alphaXtheta + \dalphaXYtheta \text{ .} \end{split}
	\end{equation}
	Substituting inside \Cref{Eq1ProofGiroux1}, we get $ \dalphaXYtheta = \gtheta\,\alphaXtheta+\dalphaXYtheta$ $-\ftheta\,\alphaYtheta + \dalphaXYtheta - \alphaXYtheta$, which, using $\alphaXtheta=0$ (we are interested in points $p\in\Sigmatheta$), gives $-\dalphaXYtheta +\ftheta\,\alphaYtheta =-\alphaXYtheta$. Replacing this identity inside \Cref{Eq2ProofGiroux1} gives $\dalphaXtheta(\Ytheta)=-\alphaXYtheta$.
	Then, again by \Cref{RmkBeforePropGiroux1}, we have %$[\Xtheta,\Ytheta]=[X, Y]$, which gives
	$\dalphaXtheta(\Ytheta)=-\alphaXYtheta=-\alphaXY$.
\end{proof}

We point out a direct consequence of \Cref{EqDalphaLieBracket} and another lemma, which we will both need later:

\begin{cor}
	\label{LemmaPoint1PropGiroux1}
	$\dalphaYtheta(\Xtheta)=\alphaXY$ on all of $\Sigma_{\theta+\pi/2}=\{\alpha(\Ytheta)=0\}$. \\ 
	In particular, along $\Sigmatheta\cap\Sigma_{\theta+\pi/2}$ (which we will show below to be independent of $\theta$ and denote by $K$), we have both $\dalphaXtheta(\Ytheta)=-\alphaXY$ and $\dalphaYtheta(\Xtheta)=\alphaXY$, which also implies $d\alpha(\Xtheta,\Ytheta)=\alphaXY<0$.
\end{cor}

\begin{lemma}
	\label{LemmaXthetaFtheta}
	$\Xtheta$ is tangent to $\Sigmatheta$. Moreover, it is transverse to $\partial F_\theta=\Sigmatheta\cap\Sigma_{\theta+\pi/2}$ and points outwards from $F_\theta$.
\end{lemma}

\begin{proof}[Proof (\Cref{LemmaXthetaFtheta})]
	Evaluating the left identity in \Cref{Eq5ProofGiroux1} on $\Xtheta$ at points $p\in \Sigmatheta$, we get $\dalphaXtheta\left(\Xtheta\right)\vert_{p}=0$, i.e.\ $\Xtheta$ is tangent to $\Sigmatheta$.
	\\ %, at points $p\in \Sigmatheta$.\\
	The second part of the statement follows from the fact that $\alphaYtheta=0$ along $\partial F_\theta = \Sigmatheta\cap\Sigma_{\theta+\pi/2}$ (by definition of $\Sigma_{\theta+\pi/2}$), and that $\dalphaYtheta\left(\Xtheta\right)<0$ along $\partial F_\theta$ by \Cref{LemmaPoint1PropGiroux1}. Indeed, this means that $\Xtheta$ points in the region where $\alphaYtheta<0$ along $\partial F_\theta$, being always tangent to $\Sigmatheta$, i.e., by definition of $F_\theta$, that it points outwards from $F_\theta$ along its boundary.
\end{proof}

%%%%%%%%%%%%%%%%%%%%%%%%%%%%%%%%%%%%%%%%%%%%%%%%%%%%%%%%%%%%%%%%%%%%%%%%%%%%%%%%%%%%

\begin{proof}[Proof (\Cref{PropGiroux1}.\ref{Item2PropGiroux1})]
	$\Sigmatheta\cap\Sigmathetaprime$ is non-empty because at the previous point we showed that $\Ytheta$ is a contact vector field transverse to $\Sigmatheta$, and we know from convex surface theory that $\{\alpha(\Ytheta)=0\}\cap\Sigmatheta\subset\Sigmatheta$ is a dividing set for the characteristic foliation $\Sigmatheta(\xi)$, and that dividing sets are always non-empty. 
	This last statement is a consequence of the fact that there are no exact symplectic forms on closed manifolds due to Stokes' identity.
	
	Let's now prove that, for $\theta\neq\theta'\bmod \pi$, $\Sigmatheta\cap\Sigmathetaprime$ is independent of $\theta,\theta'$.\\ 
	We have that $\Sigmatheta\cap\Sigmathetaprime=\{\alphaXtheta=0,\alphaXthetaprime=0\}$. 
	Now, if we consider the function $\nu=\left(\alphaX,\alphaY\right):\Sigmatheta\cap\Sigmathetaprime\rightarrow\real^2$, the equation $\alphaXtheta=\costheta \alphaX + \sintheta \alphaY=0$ tells us that, where $\nu$ is non-zero, it has to be proportional to $\left(-\sintheta,\costheta\right)$, whereas the equation  $\alphaXthetaprime=\costhetaprime \alphaX + \sinthetaprime \alphaY=0$ tells that, where $\nu$ is non-zero, it has to be proportional to $\left(-\sinthetaprime,\costhetaprime\right)$. 
	Because $\theta\neq\theta'\mod\pi$, this means $\nu\equiv0$. 
	In other words, $\Sigmatheta\cap\Sigmathetaprime$ is equal to $\nu^{-1}(0)$, i.e. it is independent of $\theta,\theta'$. We will denote it $K$, as in the statement.
	
	Finally, we prove that $K$ is a codimension $2$ submanifold of $M$. For that, it is enough to find a vector field tangent to $\Sigmathetaprime$ and transverse to $\Sigmatheta$ at every point of $K$. Because $K=\Sigmatheta\cap\Sigmathetaprime$ is independent of $\theta,\theta'$, we can suppose that $\theta=0$ and $\theta'=\pi/2$. This being said, the contact vector field $Y$ serves well to our purposes. In fact, in the proof of point \ref{Item1PropGiroux1}, we showed that $Y=X_{\pi/2}$ is transverse to $\Sigma_{0}$; moreover, it is also tangent to $\Sigma_{\pi/2}$, because $\Xtheta$ is tangent to $\Sigmatheta$ according to \Cref{LemmaXthetaFtheta}.  
\end{proof}

It now only remains to prove \Cref{PropGiroux1}.\ref{Item3PropGiroux1}. 
We use the following:
\begin{lemma}[Giroux]
	\label{LemmaGirouxOBD}
	Let $(M^{2n-1},\xi)$ be a contact manifold. Suppose there are an open book decomposition $(K,\varphi)$ of $M$ (in particular, $K$ is oriented as boundary of $\varphi^{-1}\left(\theta\right)$), a tubular neighborhood $\neigh=K\times \disk^2$ of $K$ (here $\disk^2$ is the open unit disk in $\real^2$) and a contact form $\alpha$ defining $\xi$ such that:
	\begin{enumerate}[label = (\roman*)]
		\item\label{Item1LemmaGirouxOBDrestated} $\varphi$ restricted to $\neigh\setminus K$ is the angular coordinate of the projection on the second factor $\neigh=K\times\disk^2 \rightarrow\disk^2$;
		\item\label{Item2LemmaGirouxOBDrestated} $\xi$ induces a positive contact structure on each submanifold $\Kz:=K\times\{z\}$ of $\neigh$ (notice each $\Kz$ is oriented because $K$ is);
		\item\label{Item3LemmaGirouxOBDrestated} $d\alpha$ induces a positive symplectic form on each fiber of $\varphi_{\vert M\setminus \neigh}$. 
	\end{enumerate}
	Then, the open book decomposition $(K,\varphi)$ supports the contact structure $\xi$.
\end{lemma}

\begin{proof}[Proof (\Cref{LemmaGirouxOBD})]
	Let $\alpha$ be a contact form for $\xi$ as in the statement. 
	The aim is to find a function $f:M\rightarrow\real_{+}$ such that $d(f\alpha)$ is positively symplectic on the fibers of $\varphi$. 
	%In other words, we want that $f\alpha$ induces a positive contact form on $K$ (which is trivially satisfied because it induces the same contact structure as $\alpha$) and such that $d\left(f\alpha\right)$ is a positive volume form on the pages. 
	\\
	Notice that Hypothesis \ref{Item3LemmaGirouxOBDrestated} implies that there is a very small $\epsilon>0$ such that $d\alpha$ is a symplectic form on each fiber of the restriction of $\varphi$ to $M\setminus\overline{K\times\diskoneeps}$, where $\diskoneeps$ is the disk of radius $1-\epsilon$ in $\real^2$.
	We then search the function $f$ of the following form: $f$ is a smooth function that depends only on the radius coordinate $r$ on $\disk^2$ inside $\neigh$, non$-$increasing in $r$, which is equal to $1$ on $M\setminus K\times\diskoneepshalf$ and equal to $1+e^{-kr^2}$ on $K\times\overline{\diskoneeps}$, where $k>0$ is a constant yet to determine.
	We can then compute 
	\begin{align*} 
	d\varphi\wedge d\left(f\alpha\right)^{n-1} \, = & \, d\varphi\wedge \left(\,df\wedge\alpha \, + \, f d\alpha \,\right)^{n-1}  \\
	= & \, f^{n-1} d\varphi \wedge d\alpha^{n-1} \, + \, (n-1) f^{n-2} d\varphi \wedge \shortderivr{f} dr \wedge \alpha \wedge d\alpha^{n-2}  \\
	= & \, f^{n-2}\left[ f d\varphi \wedge d\alpha^{n-1} - (n-1) \shortderivr{f}  dr \wedge d\varphi \wedge \alpha \wedge d\alpha^{n-2} \right] \text{ .} 
	\end{align*}
	Now, on $M\setminus K\times\diskoneepshalf$ we have that $f\alpha=\alpha$, hence $d\varphi\wedge d(f\alpha)^{n-1}>0$ as wanted. We then need to control its sign on $K\times\overline{\diskoneepshalf}$.\\
	Let's start by analyzing it on $K\times\overline{\diskoneeps}$. 
	Here, $\shortderivr{f}=-2kre^{-kr^2}$, so that 
	\begin{multline*}
	f d\varphi \wedge d\alpha^{n-1} - (n-1) \shortderivr{f}  dr \wedge d\varphi \wedge \alpha \wedge d\alpha^{n-2} = \\
	= e^{-kr^2} \left[ d\varphi \wedge d\alpha^{n-1} + 2\left(n-1\right) k r dr \wedge d\varphi \wedge \alpha \wedge d\alpha^{n-2} \right] \text{ .}
	\end{multline*}
	By Hypothesis \ref{Item2LemmaGirouxOBDrestated}, the form $r dr \wedge d\varphi \wedge \alpha \wedge d\alpha^{n-2}$ is positive on $\neigh$, hence on $K\times\overline{\diskoneepshalf}$, and $d\varphi \wedge d\alpha^{n-1}$ is bounded above in norm, even if we don't know its exact sign. This means that for $k>0$ big enough, the second form will dominate the first, i.e. their sum will still be positive.
	\\
	It then remains to study the sign on the open set $K\times\left(\diskoneepshalf\setminus \overline{\diskoneeps}\right)$. Here, the situation is easy because $d\varphi\wedge d\alpha^{n-1}$ is positive and $-\shortderivr{f} dr \wedge d\varphi \wedge \alpha \wedge d\alpha^{n-2}$ is non$-$negative (remember $f$ is a non$-$increasing function of $r$ in this set), so their sum is also positive.
\end{proof}

\noindent
We are now ready to give a proof of the last part of \Cref{PropGiroux1}.
In order to improve the readability, the latter is split in three main claims, which are then proved separately right after the end of the proof of \Cref{PropGiroux1}.
\begin{proof}[Proof (\Cref{PropGiroux1}.\ref{Item3PropGiroux1})] 
	Consider the smooth map $\phi:M\rightarrow\real^2$ given by $\phi(p)=\left(\alphaX_p,-\alphaY_p\right)$, and let $\varphi\coeq \sfrac{\phi}{\norm{\phi}}\co M\setminus \phi^{-1}(0)\rightarrow\cercle$.
	\begin{claim}
		\label{LemmaOBDPhi}
		$\phi$ is transverse to the origin of $\real^2$ and $\phi^{-1}(0)=K$ as subsets of $M$. 
		Also, $\varphi$ is a submersion and $\varphi^{-1}(\theta) = F_{-\theta - \pi/2}$ as subsets of $M$.
		Moreover, $\varphi^{-1}(\theta)$ is cooriented by the vector $Y_{-\theta-\pi/2}$ and $\phi^{-1}(0)$, naturally oriented as boundary of $\varphi^{-1}(\theta)$ by definition of $\varphi$, is also cooriented by the ordered couple of vectors $(Y,X)$.
	\end{claim}
	%	For readability purposes, the proof of this claim is postponed.
	In other words, \Cref{LemmaOBDPhi} tells that $(K,\varphi)$ is an open book decomposition of $M$. 
	We then need to prove that it moreover supports $\xi$.
	Notice that this is enough in order to prove point \ref{Item3PropGiroux1} of \Cref{PropGiroux1}, because the $\varphi$ in point \ref{Item3PropGiroux1} is just obtained from the $\varphi$ of \Cref{LemmaOBDPhi} by post-composing with the rotation of $\cercle$ of angle $-\frac{\pi}{2}$, so they have the same pages.

	Consider on $K,F_\theta$ the orientations such that $\phi^{-1}(0)=K$, $\varphi^{-1}(\theta) = F_{-\theta - \pi/2}$ as oriented manifolds. 
	To show that $(K,\varphi)$ is adapted to $\xi$, we then need to verify that $\xi\cap TK$ is a positive contact structure on $K$ and that there is a contact form defining $\xi$ whose differential is a positive symplectic form on each $F_{\theta}$.
	%	To prove this, we use the following result, whose proof is postponed:
	%	\begin{lemma}[Giroux]
	%		\label{LemmaGirouxOBD}
	%		Let $(M^{2n-1},\xi)$ be a contact manifold. Suppose there are an open book decomposition $(K,\varphi)$ of $M$ (in particular, $K$ is oriented as boundary of $\varphi^{-1}\left(\theta\right)$), a tubular neighborhood $\neigh=K\times \disk^2$ of $K$ (here $\disk^2$ is the open unit disk in $\real^2$) and a contact form $\alpha$ defining $\xi$ such that:
	%		\begin{enumerate}[label = (\roman*)]
	%			\item\label{Item1LemmaGirouxOBDrestated} $\varphi$ restricted to $\neigh\setminus K$ is the angular coordinate of the projection on the second factor $\neigh=K\times\disk^2 \rightarrow\disk^2$;
	%			\item\label{Item2LemmaGirouxOBDrestated} $\xi$ induces a positive contact structure on each submanifold $\Kz:=K\times\{z\}$ of $\neigh$ (notice each $\Kz$ is oriented because $K$ is);
	%			\item\label{Item3LemmaGirouxOBDrestated} $d\alpha$ induces a positive symplectic form on each fiber of $\varphi_{\vert M\setminus \neigh}$. 
	%		\end{enumerate}
	%		Then, the open book decomposition $(K,\varphi)$ is compatible with the contact structure $\xi$.
	%	\end{lemma}
	%	
	Thus, \Cref{LemmaGirouxOBD} together with the following two claims conclude the proof of \Cref{PropGiroux1}.\ref{Item3PropGiroux1}:
	\begin{claim}
		\label{Lemma3PropGiroux1}
		Let $\Psi$ be the map defined by \begin{equation*} \begin{split} \Psi \, : \,  & K\times\diskdelta \, \rightarrow \, M \\ & (p,x,y) \, \mapsto \, \flowxyp \end{split}  \; \text{ ,} \end{equation*}
		where $\psi_Z^1$ denotes the time-$1$ flow of the vector field $Z$ on $M$ and $\diskdelta$ is the $2$-disk of radius $\delta$ in $\real^2$. 
		Then, for $\delta>0$ sufficiently small, we have the following: 
		\begin{enumerate}[label = (\roman*)]
			\item \label{Item1Lemma3PropGiroux1} $\Psi$ is a diffeomorphism onto its image;
			\item \label{Item2Lemma3PropGiroux1} if we denote $\neigh:=\Psi(K\times \diskdelta)$, then we have the following commutative diagram, where $\nu$ is the composition of the projection on $\diskdelta\setminus \{0\}$ and the natural angle function $\diskdelta \setminus \left\{0\right\}\rightarrow\cercle$: 
			\begin{center}
				\begin{tikzcd}
					K\times \left(\diskdelta \setminus \left\{0\right\}\right) \ar["\Psi", rr] \ar["\nu", dr] &   & M\setminus K \ar["\varphi", dl] \\
					& \cercle &  
				\end{tikzcd}
			\end{center}
			\item \label{Item3Lemma3PropGiroux1} each $\Kz\coeq\Psi(K\times\{z\})$ is a positive contact submanifold of $(M,\xi)$.
		\end{enumerate}
	\end{claim}
	
	\begin{claim}
		\label{Lemma4PropGiroux1}
		Let $\neigh$ be the neighborhood of $K$ given by \Cref{Lemma3PropGiroux1}. Then there is a contact form $\alpha$ defining $\xi$ such that:
		\begin{enumerate}[label = (\roman*)]
			\item $\alpha$ induces a positive contact structure on each submanifold $\Kz$ of $\neigh$;
			\item $d\alpha$ is a positive symplectic form on the fibers of $\varphi\vert_{M\setminus \neigh}$. \qedhere
		\end{enumerate}
	\end{claim}
	%	\noindent
	%	This concludes the proof of point \ref{Item3PropGiroux1}.
\end{proof}

We now prove the claims used in the above proof.
\begin{proof}[Proof (\Cref{LemmaOBDPhi})]
	Clearly, $\phi^{-1}(0)=\Sigma_0\cap\Sigma_{\pi/2}=K$ as subsets of $M$. 
	
	Moreover, we can compute $d\phi(X)=\dalphaX(X)\partial_x-\dalphaY(X)\partial_y$ along $K$. 
	According to \Cref{LemmaXthetaFtheta} and \Cref{LemmaPoint1PropGiroux1}, $\dalphaX(X)=0$ and $\dalphaY(X)=\alphaXY$ along $K$, hence $d\phi(X)=-\alphaXY\partial_y$. 
	Similarly, we can compute $d\phi(Y)=-\alphaXY\partial_x$ along $K$. 
	In other words, $\phi$ is transverse to the origin of $\real^2$ and the oriented couple $(Y,X)$ gives the positive coorientation of $\phi^{-1}(0)$. %, hence $(\Ytheta,\Xtheta)$ too.
	
	To study $\varphi^{-1}(\theta)$, we argue as follows. 
	Suppose $\varphi(p)=\theta$ and write $\phi(p)\in\real^2$ in polar coordinates as $\norm{\phi(p)}\cdot(\cos\theta,\sin\theta)$. 
	Then, we can compute
	\begin{align*}
	\alpha\left(X_{-\theta-\pi/2}\right)  &\,=\,
	\alphaX \sin\theta + \alphaY \cos\theta \\
	&\, = \, \phi_1(p)\sin\theta -\phi_2\cos\theta  \\
	&\, = \, \norm{\phi(p)}\cdot\left(\cos\theta \sin \theta -\sin\theta\cos\theta\right)  \\
	&\,=\, 0 \text{ ,}
	\end{align*}
	i.e. we have that $p\in\Sigma_{-\theta-\pi/2}$.\\ 
	Hence, to show that $p\in F_{-\theta-\pi/2}$, we need to check that $Y_{-\theta-\pi/2}$ is positively transverse to $\xi$ at $p$, i.e. that $\alpha_p\left(Y_{-\theta-\pi/2}\left(p\right)\right)>0$. 
	This follows from:
	\begin{align*}
	\alpha_p\left(Y_{-\theta-\pi/2}\left(p\right)\right) \, & = \,
	\alpha_p(X(p)) \cos\theta\, -\alpha_p\left(Y\left(p\right)\right) \sin\theta  \,=\,
	\phi_1(p)\cos\theta +\phi_2\sin\theta \\ 
	& = \, \norm{\phi(p)} \left(\cos^2\theta+\sin^2\theta\right) \, =\, \norm{\phi(p)} >0 \text{ .}
	\end{align*}
	%	\begin{align*}
	%	\alpha\left(Y_{-\theta-\pi/2}\right)\vert_{p} \, & = \,
	%	\alphaX\vert_{p} \cos\theta\, -\alphaY\vert_{p} \sin\theta  \,=\,
	%	\phi_1(p)\cos\theta +\phi_2\sin\theta \\ 
	%	& = \, \norm{\phi}(p) \left(\cos^2\theta+\sin^2\theta\right) \, =\, \norm{\phi}(p) >0 \text{ .}
	%	\end{align*}
	
	We now check that $\varphi^{-1}(\theta)$ is positively cooriented by $Y_{-\theta-\pi/2}$. For this, we need to check that $d\varphi_p \left(Y_{-\theta-\pi/2}\left(p\right)\right)$ is positive.
	We can compute
	\begin{align*}
	\norm{\phi(p)} \, d\varphi_p \left(Y_{-\theta-\pi/2}\left(p\right)\right)\,  & = \,
	\left(\cos\theta \, d\phitwo - \sin\theta\, d\phione\right)_p\left(Y_{-\theta-\pi/2}\left(p\right)\right) \\
	& = \, \left[-\cos\theta \, d \left(\alpha\left(Y\right)\right)-\sin\theta\, d \left(\alpha\left(X\right)\right)\right]_p\left(Y_{-\theta-\pi/2}\left(p\right)\right) \\
	& = \, d \left(\alpha \left(X_{-\theta-\pi/2}\right)\right)_p\left(Y_{-\theta-\pi/2}\left(p\right)\right)\\
	& \overset{(*)}{=} \, - \left(\alpha \left(\left[X,Y\right]\right)\right)(p) >0 \text{ ,}
	\end{align*} 
	%	\begin{align*}
	%	\norm{\phi} d\varphi \left(Y_{-\theta-\pi/2}\right)\vert_{p}\,  & = \,
	%	\left(\cos\theta \, d\phitwo - \sin\theta\, d\phione\right)\left(Y_{-\theta-\pi/2}\right)\vert_{p} \\
	%	& = \, \left[-\cos\theta \, d \left(\alpha\left(Y\right)\right)-\sin\theta\, d \left(\alpha\left(X\right)\right)\right]\left(Y_{-\theta-\pi/2}\right)\vert_{p} \\
	%	& = \, d \left(\alpha \left(X_{-\theta-\pi/2}\right)\right)\left(Y_{-\theta-\pi/2}\right)\vert_{p} \\
	%	& \overset{(*)}{=} \, - \alpha \left(\left[X,Y\right]\right)\vert_p >0 \text{ ,}
	%	\end{align*}
	where $(*)$ comes from \Cref{EqDalphaLieBracket}. 
	%This concludes the proof of \Cref{LemmaOBDPhi}.
\end{proof}

\begin{proof}[Proof (\Cref{Lemma3PropGiroux1})]
	Let's start with point \ref{Item1Lemma3PropGiroux1}. We can explicitly evaluate the differential $d\Psi$ at points of the form $(p,0,0)$. On $K\times\{0\}$, we simply have that $d\Psi(\dex)=Y$, $d\Psi(\dey)=X$ and that $d\Psi(V)=V$ for all vector fields $V$ which are tangent to $K\times\{0\}$. 
	This shows that $\Psi$ is a local diffeomorphism at each point $(p,0,0)$. Hence, by compactness, $\Psi$ is also a diffeomophism from $K\times\diskdelta$ onto its image, provided $\delta$ is small enough. 
	
	We now prove point \ref{Item2Lemma3PropGiroux1}.
	For $\theta\in\cercle$, let $H_\theta:K\times[0,\delta)\rightarrow M$ be defined by $H_\theta(p,r)\coeq\Psi(p,r\cos\theta,r\sin\theta)$; we then have to show that $\varphi(H_\theta(p,r))=\theta$. \\
	Because $Y_{-\theta}=\sin\theta \cdot X + \cos \theta \cdot Y$, we can rewrite more explicitly $H_\theta (p,r)=\psi^{r}_{Y_{-\theta}}(p)$, i.e. $H_\theta(\cdot,r)$ is the flow of $Y_{-\theta}$ at time $r$.
	By \Cref{LemmaXthetaFtheta}, $Y_{-\theta}=-X_{-\theta-\pi/2}$ is tangent to $\Sigma_{-\theta-\pi/2}$ and entering in $F_{-\theta-\pi/2}$; in particular, for $r>0$ we have $\psi^{r}_{Y_{-\theta}}(p)\in F_{-\theta-\pi/2}$. 
	Now, by \Cref{LemmaOBDPhi}, $\varphi^{-1}\left(\theta\right)=F_{-\theta-\frac{\pi}{2}}$, which implies $\varphi(H_\theta(p,r))=\theta$, as desired.
	
	Let's finish with point \ref{Item3Lemma3PropGiroux1}.
	Because the contact condition is open, up to shrinking $\delta$, it is enough to prove that $K_0=\Psi(K\times\{0\})$ is a positive contact submanifold.
	This follows from general results from Giroux \cite{Gir91}: indeed, $\Xtheta$ defines the characteristic foliation of $\Sigmatheta$, and $K$ is transverse to it. 
\end{proof}

\begin{proof}[Proof (\Cref{Lemma4PropGiroux1})]
	We search for a function $f$ such that $\widetilde{\alpha}\coeq f\alpha$ satisfies $d\varphi\wedge d\widetilde{\alpha}^{n-1}>0$ on $M\setminus \intpart(\neigh)$.
	We start by computing
	\begin{align*}
	d\varphi\wedge d\widetilde{\alpha}^{n-1} \, & = \, f^{n-1} d\varphi\wedge d\alpha^{n-1} \, + \,  (n-1) f^{n-2} d\varphi \wedge df \wedge \alpha \wedge d\alpha^{n-2} \\
	& =  \, f^{n-2} \left[ f d\varphi \wedge d\alpha^{n-1} - \left(n-1\right) df\wedge d\varphi\wedge \alpha\wedge d\alpha^{n-2} \right] \text{ .}
	\end{align*}
	Let now $\epsilon>0$ be such that $\overline{\{\norm{\phi}<2\epsilon\}}\subset\neigh$ and choose a smooth non--increasing function $f$ of $\norm{\phi}$, equal to $1/{\epsilon}$ on the set $\{\norm{\phi}<\epsilon\}$ and equal to $1/{\norm{\phi}}$ on the set $M\setminus\{\norm{\phi}<2\epsilon\}$.
	
	We then analyze $d\varphi\wedge d\widetilde{\alpha}$ on $\neigh^{c}$. 
	Here, $f=\sfrac{1}{\norm{\phi}}$ and $df=-\sfrac{d\norm{\phi}}{\norm{\phi}^2}$, so
	\begin{equation*}
	\norm{\phi}^{n+1} d\varphi\wedge d\widetilde{\alpha}^{n-1} \, = \, \norm{\phi}^2 d\varphi \wedge d\alpha^{n-1} + \left(n-1\right) \norm{\phi} d\norm{\phi}\wedge d\varphi\wedge \alpha\wedge d\alpha^{n-2} \text{ .}
	\end{equation*}
	Moreover, recalling that $\phi=(\alphaX,-\alphaY)$, 
	one has 
	\begin{align*}
	&\norm{\phi}^2 d\varphi = \phione d\phitwo - \phitwo d\phione = -\alphaX d\left(\alphaY\right) + \alphaY d\left(\alphaX\right) \text{ ,} \\
	&\norm{\phi} d\norm{\phi}\wedge d\varphi = d\phione\wedge d\phitwo = - d\left(\alphaX\right) \wedge d\left(\alphaY\right) \text{ .}
	\end{align*}
	%$ \norm{\phi}^2 d\varphi = \phione d\phitwo - \phitwo d\phione = -\alphaX d\left(\alphaY\right) + \alphaY d\left(\alphaX\right)$ and $ \norm{\phi} d\norm{\phi}\wedge d\varphi = d\phione\wedge d\phitwo = - d\left(\alphaX\right) \wedge d\left(\alphaY\right) $, so that 
	In particular,
	\begin{equation}
	\label{Eq1ProofLemma4PropGiroux1}
	\begin{split}
	\norm{\phi}^{n+1} d\varphi\wedge d\widetilde{\alpha}^{n-1} \, = \, &\left[-\alphaX d\left(\alphaY\right) + \alphaY d\left(\alphaX\right)\right] \wedge d\alpha^{n-1} \\ 
	& - \left(n-1\right) d\left(\alphaX\right) \wedge d\left(\alphaY\right)\wedge \alpha\wedge d\alpha^{n-2} \text{ .}
	\end{split}
	\end{equation}
	
	Notice then that the right hand side is exactly the same (up to a factor $n$) as the one of \Cref{EqOmega} in the proof of \Cref{PropGiroux2}.
	Hence, the exact same computations made in that proof tell us that
	$$ \norm{\phi}^{n+1} d\varphi\wedge d\widetilde{\alpha}^{n-1} \; = \; -\alphaXY \, \alpha\wedge d\alpha^{n-1} \text{ .} $$

	\noindent
	Now, $[X,Y]$ is negatively transverse to $\xi$ by hypothesis, so $d\widetilde{\alpha}$ is positively symplectic on the fibers of $\varphi\vert_{M\setminus \neigh}$, as desired.
\end{proof}

%%%%%%%%%%%%%%%%%%%%%%%%%%%%%%%%%%%%%%%%%%%%%%%%%%%%%%%%%%%%%%%%%%
%%%%%%%%%%%%%%%%%%%%%%%%%%%%%%%%%%%%%%%%%%%%%%%%%%%%%%%%%%%%%%%%%%

%%%%%%%%%%%%%%%%%%%%%%%%%%%%%%%%%%%%%%%%%%%%%%%%%%%%%%%%%%%%%%%%%%%%%%%%%%%%%%%%%%%%
%%%%%%%%%%%%%%%%%%%%%%%%%%%%%%%%%%%%%%%%%%%%%%%%%%%%%%%%%%%%%%%%%%%%%%%%%%%%%%%%%%%%
%%%%%%%%%%%%%%%%%%%%%%%%%%%%%%%%%%%%%%%%%%%%%%%%%%%%%%%%%%%%%%%%%%%%%%%%%%%%%%%%%%%%

\section{Bourgeois structures as contact fiber bundles}
\label{SecContFibBund}

The aim of this section is to generalize the construction due to Bourgeois using the notion of contact fiber bundles introduced in Lerman \cite{Ler04}. %, under the hypothesis of the existence of a flat connection.
\\
More precisely, we start by recalling the Bourgeois construction in \Cref{SubSecBourgConstr}.
In \Cref{SubSecGenContFibBund} we recall the definitions and the main properties of contact fiber bundles.
\Cref{SubSecComparingContConnect} describes how to effectively compare two of them, which is then used to generalize the construction by Bourgeois \cite{Bou02}. 
In particular, in \Cref{SubSubSecFlatContFibBund} we take a general fibration admitting a flat contact connection and we consider on it two non-trivial subclass of all its contact connections. 
The first class is characterized in terms of deformations to the flat contact connection, in a flavor similar to the notion of contactizations introduced in \Cref{DefContPullBack}. The second one, subclass of the first, is a direct generalization of the examples from \cite{Bou02} in the setting of contact fiber bundles and is presented in \Cref{SubSubSecBourgContStrRev}. 
There \Cref{PropIntroIsotopyClassOBD} from the introduction is also proved using the results from \Cref{SecOBDContVF}.
Lastly, in \Cref{SubSubSecContactizationOpContCat} we study the stability of the first class under the operation of contact branched covering.

%%%%%%%%%%%%%%%%%%%%%%%%%%%%%%%%%%%%%%%%%%%%%%%%%%%%%%%%%%%%%%%%%%%%%%%%%%%%%%%%%%%%
%%%%%%%%%%%%%%%%%%%%%%%%%%%%%%%%%%%%%%%%%%%%%%%%%%%%%%%%%%%%%%%%%%%%%%%%%%%%%%%%%%%%

\subsection{The Bourgeois construction}
\label{SubSecBourgConstr}

%Bourgeois gives in \cite{Bou02} a construction that goes in the opposite direction with respect to that in \cite{Lut79} recalled above.
%Indeed, 
Using the notion of open book decompositions for contact manifolds $(M^{2n-1},\xi)$ from Giroux \cite{Gir02}, Bourgeois constructs in \cite{Bou02} explicit contact structures on $M\times\torus$.
%the total space of the principal $\torus$-bundle $\pi\co M\times\torus\rightarrow M$.
More precisely, the main statement of \cite{Bou02} can be rephrased as follows:
\begin{thm}[Bourgeois]
	\label{ThmBourgeoisBis}
	Let $(M^{2n-1},\xi)$ be a contact manifold and $(B,\varphi)$ an open book decomposition of $M$ supporting $\xi$.
	\begin{enumerate}[label=\alph*.]
		\item \label{Item1ThmBourgeoisBis} There is a smooth map $\phi=(\phione,\phitwo)\co M\rightarrow \real^2$ defining the open book $(B,\varphi)$ and such that $ \gamma \wedge d\gamma^{n-2}\wedge d\phione \wedge d\phitwo \geq 0$ on $M$, where $\gamma$ is any contact form defining $\xi$.
		\item \label{Item2ThmBourgeoisBis} If $\phi$ is as in point \ref{Item1ThmBourgeoisBis}, then for any choice of coordinates $(\theta_1,\theta_2)$ on $\torus$ and for any contact form $\beta$ defining $\xi$ and adapted to the open book $(B,\varphi)$, the $1$-form $\alpha\coeq \beta + \phione d\theta_1 - \phitwo d\theta_2$ is a contact form on $M\times\torus$.  
	\end{enumerate}
\end{thm}

We point out that the condition $\gamma\wedge d\gamma^{n-2} \wedge d\phione\wedge d\phitwo \geq 0$ in point \ref{Item1ThmBourgeoisBis} of \Cref{ThmBourgeoisBis} is independent of the choice of form $\gamma$ defining $\xi$.
Indeed, it is equivalent to the fact that $\xi$ induces by restriction a contact structure on $\phi^{-1}(z)$, for each $z$ regular value of $\phi$. 
\\
Notice, moreover, that the contact form $\alpha$ in point \ref{Item2ThmBourgeoisBis} clearly induces the original contact structure $\xi$ on each fiber $M\times\{pt\}$ of $\pi\co M\times\torus\rightarrow\torus$. 
%is clearly invariant under the natural $\torus$-action on the principal $\torus$-bundle $\pi\co M\times\torus \rightarrow M$.

\begin{remark}
	\label{RmkEpsBourgContStr}
	If $\phi=(\phi_1,\phi_2)$ satisfies point \ref{Item1ThmBourgeoisBis} of \Cref{ThmBourgeoisBis}, then, for all $\epsilon>0$, the same is true for $\epsilon\phi=(\epsilon\phi_1,\epsilon\phi_2)$. 
	In particular, the 1-forms $\alpha_\epsilon:=\beta+\epsilon\phi_1 d\theta_1-\epsilon\phi_2 d\theta_2$ always define positive contact structures by point \ref{Item2ThmBourgeoisBis} of \Cref{ThmBourgeoisBis}, which are moreover all isotopic by Gray's theorem. 
	Notice that $\alpha_0=\beta$ defines the hyperplane field $\xi\oplus T\torus$, which is not a contact structure on $M\times\torus$, 
	but still defines a contact structure on each fiber of the projection $\pi\co M\times\torus\rightarrow\torus$. 
	%nonetheless, in \Cref{SecContFibBund} we will call it \emph{contact fiber bundle} on $M\times\torus$ and show that it plays an important role in understanding the properties of the construction in \Cref{ThmBourgeoisBis}.
\end{remark}

%%%%%%%%%%%%%%%%%%%%%%%%%%%%%%%%%%%%%%%%%%%%%%%%%%%%%%%%%%%%%%%%%%%%%%%%%%%%%%%%%%%%
%%%%%%%%%%%%%%%%%%%%%%%%%%%%%%%%%%%%%%%%%%%%%%%%%%%%%%%%%%%%%%%%%%%%%%%%%%%%%%%%%%%%

\subsection{Generalities on contact fiber bundles}
\label{SubSecGenContFibBund}

We recall in this section the notion of contact fiber bundle introduced by Lerman in \cite{Ler04}, focusing in particular on their description using contact connections. 
We specialize here to the case of fiber bundles over (closed) surfaces as this will be the case we are interested in for the following sections.

%We recall in this section the definition of contact fiber bundles given by Lerman in \cite{Ler04}, as well as a useful criterion, also due to Lerman, that determines whether a contact fiber bundle is a contact structure on the total space of the underlying smooth fibration.

Let $\Sigma^{2}$, $M^{2n-1}$ and $V^{2n+1}$  be smooth closed manifolds and $\pi\co V\rightarrow \Sigma$ a smooth fiber bundle with fiber $M$.
Denote by $M_b$ the fiber of $\pi$ over $b\in \Sigma$.
Suppose also $V$ and $\Sigma$ oriented, and let $M_b$ be the (oriented) preimage $\pi^{-1}(b)$. 

\begin{definition}{\cite{Ler04}}
	\label{DefContFibBund}
	A \emph{contact fiber bundle} is a cooriented hyperplane field $\eta$ on $V$ 
	that induces a contact structure $\xi_b$ on each fiber $M_b$ of $\pi$.
	%such that for each fiber $M_b$ of $\pi$ the intersection $\xi_b := \eta\cap T M_b$ is a positive contact structure on $M_b$.
\end{definition}

Notice that, given a contact manifold $(M,\xi)$, both the hyperplane field $\xi\oplus T\torus$ and the contact structures on $M\times\torus$ obtained as in \Cref{ThmBourgeoisBis} are examples of contact fiber bundles on the trivial bundle $\pi\co M\times\torus\rightarrow\torus$.

%	In the following, we denote the data of a contact fiber bundle by $(\pi\co V\rightarrow \Sigma,\eta)$.

\begin{lemma}{\cite[Lemma 2.4]{Ler04}}
	\label{LemmaContFibBund}
	Let $(\pi\co V\rightarrow \Sigma,\eta)$ be a contact fiber bundle and $\alpha$ a 1-form on $V$ defining $\eta$. The distribution $\connH$ defined as the $d\alpha\vert_{\eta}$-orthogonal of $\xi_b$ in $\eta$ is an Ehresmann connection on the bundle $\pi\co V\rightarrow\Sigma$, i.e. at any point $p\in V$ we have $\eta(p) = \xi_{\pi(p)}(p) \oplus \connH(p)$.
	Moreover, its holonomy over a path $\gamma:[0,1]\rightarrow B$ is a contactomorphism between $\xi_{\gamma(0)}$ and $\xi_{\gamma(1)}$.
\end{lemma}

Vice versa, the data of $\xifib\coeq\eta\cap\ker(d\pi)=\left(\xi_b\right)_{b\in\Sigma}$ and $\connH$ obviously allows to restore the hyperplane field $\eta$.
%Hence, we will call contact fiber bundle the data $(\pi \co V\rightarrow\Sigma,\eta)$ as well as the data $(\pi\co V\rightarrow \Sigma,\xifib,\connH)$,  without any distinction.
%%Also, if $v$ is a vector of $T_b B$, we denote by $\hash{v}(p)$ its horizontal lift to $\connH(p)$, for all $p\in \pi^{-1}(b)$. \\
For this reason, we introduce the following auxiliary object:
\begin{definition}
	A \emph{fiber bundle with contact fibers} is the data $\left(\pi\colon V\rightarrow \Sigma,\xifib\right)$ of a fiber bundle $\pi\co V\rightarrow\Sigma$ and a $3$-codimensional distribution $\xifib$ on $V$ inducing, for all $b\in\Sigma$, a contact structure $\xi_b$ on the fiber $M_b$. 
\end{definition}

Recall also that any Ehresmann connection $\connH$ on a fiber bundle $\pi\co V\rightarrow \Sigma$ is equivalent to a fiber--wise projection $\omega$ of $TV$ onto $\ker(d\pi)$, i.e.\ to a \emph{connection form} $\omega\in\Omega^1(V;\ker (d\pi))$, defined on $V$ and with values in $\ker(d\pi)\subset TV$, such that $\omega\circ \omega=\omega$ and $\omega\vert_{\ker(d\pi)}=\Id\vert_{\ker(d\pi)}$. 
More precisely, given $\connH$, each vector field $Z$ on $V$ can be uniquely decomposed as $Z = Z_h + Z_v$, where $Z_h$ is \emph{horizontal}, i.e. everywhere tangent to $\connH$, and $Z_v$ is \emph{vertical}, i.e everywhere tangent to the fibers of $\pi$.
Then, for each $Z$ vector field on $V$, one can define $\omega(Z)\coeq Z_v$. Vice versa, given an $\omega\in\Omega^1(V;\ker (d\pi))$ as above, $\connH\coeq \ker(\omega)$ is an Ehresmann connection.

%%%%%%%%%%%%%%%%%%%%%%%%%%%%%%%%%%%%%%%%%%%%%%%%%%%%%%%%%%%%%%%%%%%%%%%%%%%%%%%%%%%%

\subsection{Comparing contact fiber bundles}
\label{SubSecComparingContConnect}

In this section, we are going to compare two contact fiber bundles having the same underlying structure of fiber bundles with contact fiber.
\\

We start by showing that, given a fiber bundle with contact fibers $(\pi\co V\rightarrow\Sigma,\xifib)$, one can naturally associate to it a vector bundle $\VFfib(V,\xifib)\rightarrow \Sigma$ having as fiber, over a point $b\in\Sigma$, the Frechet vector space of contact vector fields for $(M_b,\xi_b)$.
We invite the reader to consult Kriegl--Michor \cite{KriMicConvSett} for the foundations of analysis on manifolds locally modeled on Frechet vector spaces.
We can explicitly construct $\VFfib(V,\xifib)\rightarrow \Sigma$ as follows.
\\
The fiber bundle with contact fibers $(\pi\colon V \rightarrow \Sigma,\xifib)$ is equivalent to the following data: 
an open cover  $(U_i)_{i\in I}$ of $\Sigma$, 
trivial bundles with contact fibers $(\pr_{U_i}\co M\times U_i \rightarrow U_i,\xi\oplus\{0_{TU_i}\})_{i\in I}$, where $\xi$ is contact on $M$, 
and transition functions $\varphi_{i,j}\co U_i\cap U_j \rightarrow \Diff(M,\xi)$, where $\Diff(M,\xi)$ is the space of contactomorphisms of $(M,\xi)$.
Then, $\VFfib(V,\xifib)\rightarrow \Sigma$ is given by the same open cover $(U_i)_{i\in I}$ of $\Sigma$ of $\Sigma$, by a collection $(\VF(M,\xi)\times U_i\rightarrow U_i)_{i\in I}$ trivial bundles (here, $\VF(M,\xi)$ is the space of contact vector fields on $(M,\xi)$), and by transition functions $\Phi_{i,j}\co U_i\cap U_j \rightarrow GL(\VF(M,\xi))$ given by $\Phi_{i,j}(b)=d(\varphi_{i,j}(b))$ for each $b\in U_i\cap U_j$.

\begin{remark}
	\label{RmkNonLinFrameBund}
	We proposed here a very direct construction of the the vector bundle $\VFfib(V,\xifib)\rightarrow \Sigma$, in order to keep this presentation simple and self contained.
	This being said, $\VFfib(V,\xifib)\rightarrow \Sigma$ can also be interpretated as an adjoint bundle as follows.
	\\
	Analogously to \cite[Paragraph 44.4]{KriMicConvSett} defining the principal \emph{(nonlinear) frame bundle} of a smooth fiber bundle, one can associate to a given fiber bundle with contact fibers $(\pi\co V\rightarrow \Sigma, \xifib)$ (and with model fiber $(M,\xi)$) its \emph{(nonlinear) contact frame bundle} $E\rightarrow\Sigma$, which is a natural principal $\Diff(M,\xi)$-bundle associated to $(\pi,\xifib)$. 
	Then, up to isomorphism of vector bundles over $\Sigma$, $\VFfib(V,\xifib)\rightarrow \Sigma$ is just the adjoint bundle associated to $E\rightarrow\Sigma$. 
\end{remark}

We now want to show that the space of contact connections on a given fiber bundle with contact fibers $(\pi\colon V\rightarrow\Sigma,\xifib)$ has naturally the structure of an affine space over the vector space $\Omega^1(\Sigma;\VFfib(V,\xifib))$ of $1$-forms defined on $\Sigma$ and with values in the vector bundle $\VFfib(V,\xifib)\rightarrow\Sigma$.

Let $\connH_0$ be a reference contact connection on $(\pi\colon V\rightarrow\Sigma,\xifib)$. 
For simplicity, we call $(\pi\colon V\rightarrow\Sigma,\xifib,\connH_0)$ a \emph{referenced fiber bundle with contact fibers} in the following.
Denote also by $\omega_0$ the connection form associated to $\connH_0$.
\\
Notice that, given a point $p\in V$ and a vector $v\in T_p V$, the vector $\omega_0(v)-\omega(v)$ is tangent to the fiber $M_p$ of $\pi$.
Indeed, both $\omega$ and $\omega_0$ are with values in $\ker (d\pi)$.
Moreover, it only depends on the vector $u\coeq d\pi(v)$ on $T_{\pi(p)}\Sigma$, i.e. if $v,v'\in T_p V$ such that $d\pi(v)=d\pi(v')$, then $\omega_0(v)-\omega(v) = \omega_0(v')-\omega(v')$.
Indeed, for such $v,v'$, we have $d\pi(v-v')=0$, i.e. $v-v'$ is tangent to $T_p V$, hence $\omega_0(v-v')=v-v'=\omega(v-v')$, which, by linearity of $\omega$ and $\omega'$, gives the desired equality.
In other words, given a point $b\in\Sigma$, a vector $u\in T_b\Sigma$ and a vector field $Z_b$ on the fiber $M_b$ such that $d\pi(Z_b) = u$, we can define $A_u\coloneqq \omega_0(Z_b)-\omega(Z_b)$.
\\
The last thing to show is then that $A$ has actually values in the vector bundle of contact vector fields on the fibers of the fiber bundle with contact fiber $(\pi\co V\rightarrow \Sigma,\xifib)$, i.e. that $A\in\Omega^1(\Sigma;\VFfib(V,\xifib))$.
For this, notice that, if $u\in T_b\Sigma$ and $Z_b$ is a vector field on $M_b$ with $d\pi(Z_b)=u$, then $A_u=\omega_0(Z_b)-\omega(Z_b)=(Z_b)_{h} - (Z_b)_{h_0}$, where 
$(Z_b)_{h_0} \coeq Z_b - \omega_0(Z_b)$ and $(Z_b)_{h} \coeq Z_b - \omega(Z_b)$ 
%$W_0\coeq\widehat{Z_b}$ and $W\coeq\widetilde{Z_b}$ 
are the lifts of $u$ which are horizontal for, respectively, $\connH_0$ and $\connH$. 
Then, \Cref{LemmaContFibBund} tells that both the flows $\psi_0^t$ and $\psi^t$ of, respectively, $(Z_b)_{h_0}$ and $(Z_b)_{h}$ give contactomorphisms between different fibers of $(\pi,\xifib)$.
Because $\frac{d}{dt}\vert_{t=0}(\psi_{0}^t\circ \psi^{-t})=(Z_b)_{h_0}-(Z_b)_{h}$, this then directly implies that $A_u=(Z_b)_{h}-(Z_b)_{h_0}$ is a contact vector field for $(M_b,\xi_b)$.

With a little abuse of notation, for each $b\in\Sigma$ and $u\in T_b\Sigma$, we will just denote by  $\widehat{u}$ and $\widehat{u}^0$ the vector fields on $M_p$ given by the lifts of $u$ which are horizontal for, respectively, $\connH$ and $\connH_0$, i.e. the $(Z_b)_{h}$ and $(Z_b)_{h_0}$ above.

\begin{definition}
	For any referenced fiber bundle with contact fibers $(\pi\colon V\rightarrow\Sigma,\xifib, \connH_0)$ and any contact connection $\connH$ on $(\pi\colon V\rightarrow\Sigma,\xifib)$, the $1$--form $A\in\Omega^1(\Sigma;\VFfib(V,\xifib))$ defined as above is called \emph{potential} of $\connH$ with respect to $\connH_0$.
\end{definition}

\begin{remark}
	\label{RmkConnectAdjNonLinFrameBund}
	The fact that the space of contact connections on a given $(\pi\colon V\rightarrow\Sigma,\xifib)$ is an affine space over $\Omega^1(\Sigma;\VFfib(V,\xifib))$ has also the following more theoretical interpretation, in the spirit of \Cref{RmkNonLinFrameBund}.
	\\
	Analogously to what is explained in \cite[Paragraph 44.5]{KriMicConvSett} in the case of smooth fiber bundles, connection forms on $(\pi,\xifib)$ correspond bijectively to principal connections on the (nonlinear) contact frame bundle $E\rightarrow \Sigma$. Now, the space of principal connections on a principal bundle has naturally an affine structure over the vector space of $1$-forms on the base with values in the adjoint bundle. 
\end{remark}

The potential $A$ also allows to compare the curvature of $\connH$ with that of $\connH_0$.
In order to explain how, we need to introduce two more objects.

Firstly, we show that the connection $\connH_0$ on $(\pi\colon V\rightarrow\Sigma,\xifib)$ induces a covariant derivative $\nabla\co \Gamma(\VFfib(V,\xifib))\rightarrow\Omega^{1}(\Sigma; \VFfib(V,\xifib))$ on the associated vector bundle $\VFfib(V,\xifib)\rightarrow\Sigma$.
%More precisely, $\nabla$ can be defined as follows.
\\
More precisely, given a vector field $U$ on $\Sigma$ and a section $\sigma$ of $\VFfib(V,\xifib)\rightarrow\Sigma$, consider the $\connH_0$-horizontal lift $\widehat{U}^0$ of $U$ and the vector field $\overline{\sigma}$ on $V$ defined by $\overline{\sigma}(p)=\sigma_{\pi(p)}(p)$, where $\sigma_{\pi(p)}$ denotes the image of $\pi(p)$ via $\sigma$.
Notice that their Lie bracket $[\widehat{U}^0,\overline{\sigma}]$ is contained in $\ker(d\pi)$ and is, moreover, a contact vector field on each fiber of $\pi\colon V\rightarrow\Sigma$.
%that its restriction $[\widehat{X},\overline{\sigma}]_{fib}$ to each fiber of $\pi$ gives a section of $\VFfib(V,\xifib)\rightarrow\Sigma$.
An explicit computation also shows that $[\widehat{U}^0,\overline{\sigma}]$ is $C^\infty(\Sigma)$-linear in $U$ and satisfies the Leibniz rule in $\sigma$.
In other words, $\nabla_U\sigma\coloneqq [\widehat{U}^0,\overline{\sigma}]$ gives a well defined covariant derivative.
\\
We point out that $\nabla$ is flat. 
Indeed, the curvature $F\in\Gamma(\VFfib(V,\xifib))\rightarrow\Omega^2(\Sigma;\VFfib(V,\xifib))$ of $\nabla$ is, by definition, given by  $F(U,W)\sigma= \nabla_U \nabla_W\sigma - \nabla_W\nabla_U \sigma - \nabla_{[U,W]}\sigma$ for all $U,W$ vector fields on $\Sigma$ and $\sigma\in\Gamma(\VFfib(V,\xifib))$.
A direct computation using the Jacobi identity for the Lie bracket of vector fields on $V$ then shows that $F=0$, as desired.

The second object we need to introduce is the covariant exterior derivative $\dnabla\co \Omega^p(\Sigma; \VFfib(V,\xifib))\rightarrow\Omega^{p+1}(\Sigma; \VFfib(V,\xifib))$ naturally induced by $\nabla$. 
As explained for instance in Kriegl--Michor \cite[Paragraph 37.29]{KriMicConvSett}, $\dnabla$ is characterized by the formula
\begin{multline}
\label{EqDnabla}
%\begin{split}
\dnabla \omega (U_1,\ldots, U_{p+1})  \coloneqq %& 
\sum_{i=0}^p (-1)^i \nabla_{U_i}(\omega(U_1,\ldots,\widehat{U_i},\ldots,U_{p+1})) \\
%&
+\sum_{0\leq i,j\leq p} (-1)^{i+j} \omega([U_i,U_j],U_1,\ldots,\widehat{U_i},\ldots,\widehat{U_j},\ldots,U_{p+1}) 
\text{ ,}
%\end{split}
\end{multline}
for all $\omega\in\Omega^p(\Sigma; \VFfib(V,\xifib))$ and $U_1,\ldots,U_{p+1}$ vector fields on $\Sigma$.
Here, the notation $\widehat{Z}$ denotes the fact that the vector field $Z$ is omitted in the argument.
\\
Notice that the flatness of $\nabla$ implies that $\dnabla^2=0$.
This is a consequence of the formula $\dnabla \dnabla \omega = \omega\wedge F$ for all $\omega\in \Omega^p(\Sigma,\VFfib(V,\xifib))$ (see for instance \cite[Paragraph 37.29]{KriMicConvSett} for a proof of this identity).
In other words, $\dnabla$ is a differential on the chain complex $\Omega^*(\Sigma,\VFfib(V,\xifib))$.
\\

We are now ready to give an expression for the curvature $R$ of $\connH$ in terms of the curvature $R_0$ of $\connH_0$.

Let $X,Z$ be vector fields on $V$, and denote, as before, by $X_h,Z_h$ their $\connH$-horizontal component.
By definition of curvature $R\in\Omega^2(V;\ker(d\pi))$ of $\connH$, we have $R(X,Z)=\omega([X_h,Z_h])$.
Introducing the potential $A$, one can further write $\omega([X_h,Z_h])=\omega_0([A_{U} + X_{h_0},A_{W} + Z_{h_0}])-A_{[U,W]}$, where $U\coeq d\pi(X)$ and $W\coeq d\pi(Z)$.
Notice that $[A_{U},Z_{h_0}]$, $[X_{h_0},A_{W}]$ and $[A_{U},A_{W}]$ are all vertical, i.e. vector fields on $V$ which are tangent to the fibers of $\pi$. 
Because $\omega_0=\Id$ on $\ker(d\pi)$, one then has $\omega_0([A_{U} + X_{h_0},A_{W} + Z_{h_0}])=[A_{U},Z_{h_0}]+[X_{h_0},A_{W}] + [A_{U},A_{W}] + \omega_0([X_{h_0},Z_{h_0}])= [A_{U},Z_{h_0}]-[A_{W},X_{h_0}] + [A_{U},A_{W}] + R_0(X,Z)$.
Finally, remark that $X_{h_0}$ and $Z_{h_0}$ are just the $\connH_0$-horizontal lifts $\widehat{U}^0$ and $\widehat{W}^0$ of $U$ and $W$ respectively.
Hence, by \Cref{EqDnabla}, we also have $\dnabla A (U,W) = \nabla_U(A_{W}) - \nabla_{W}(A_U) - A_{[U,W]} =  [\widehat{U}^0,A_{W}]-[\widehat{W}^0,A_U] - A_{[U,W]}$.
Putting all the pieces together, we then get: for all $X,Z$ vector fields on $V$,
\begin{equation}
\label{EqCompCurv}
R (X,Z) = R_0(X,Z) + \dnabla A(\pi_*X,\pi_*Z) + [A_{\pi_*X},A_{\pi_*Z}] \text{ .}
\end{equation}
% and for $U=d\pi(Z)$, $U'=d\pi(U')$.

As it will be useful for the following section, we also point out the following fact. 
Fix a covariant derivative $\nabla^\Sigma$ on the tangent bundle of $\Sigma$.
Then, the $\nabla$ introduced above naturally extends to a unique map 
$$\overline{\nabla}\co \Omega^p(\Sigma; \VFfib(V,\xifib))\rightarrow\Omega^{p+1}(\Sigma; \VFfib(V,\xifib))$$
satisfying the following property: for all $W_1,\ldots,W_p,U$ vector fields on $\Sigma$,
\begin{equation}
\label{EqExtCovDer}
\begin{split}
(\overline{\nabla}\omega) (W_1,\ldots,W_{p},U) \,= \, & \nabla_U(\omega(W_1,\ldots,W_p)) \\
& - \sum_{i=1}^p \omega(W_1,\ldots,\nabla^\Sigma_U W_i, \ldots, W_p) \text{ .}
\end{split}
\end{equation}

%%%%%%%%%%%%%%%%%%%%%%%%%%%%%%%%%%%%%%%%%%%%%%%%%%%%%%%%%%%%%%%%%%%%%%%%%%%%%%%%%%%%

\subsection{Flat contact bundles and contact deformations}
\label{SubSubSecFlatContFibBund}

Here we call \emph{flat contact bundle} any referenced fiber bundle with contact fibers $\left(\pi\co V\rightarrow\Sigma,\eta_0=\xifib\oplus\connH_0\right)$ such that $\connH_0$ satisfies $R_0=0$.
%
%    We start by giving an analogue of \cite[Remark 3.2 and Theorem 3.6]{Ler04} in this flat setting. 
%	More precisely, \cite[Section 3.2]{Ler04} deals with the case of principal $G-$bundles using the notion of contact moment map, under the implicit assumption of the Lie group $G$ having dimension at least $1$; 
%	the case of flat contact connection below corresponds to the case of $G=\pi_1(\Sigma)$ of dimension $0$, where $\Sigma$ is the surface which is base of the contact fiber bundle.
%	For this reason, we try to keep our description as self contained as possible. 
%	\\

The first reason why flat contact bundles are interesting is because they admit a ``nice'' presentation in terms of their monodromy.
\\ 
% of $\left(\pi\co V\rightarrow\Sigma,\xifib,\connH_0\right)$.\\
Indeed, once fixed a certain fiber $(M,\xi)$ of $(\pi\co V\rightarrow\Sigma,\xifib)$ over $b\in\Sigma$, one gets a representation $\rho:\pi_1\left(\Sigma\right)\rightarrow\Diff(M,\xi)$, where $\Diff(M,\xi)$ is the space of contactomorphisms of $(M,\xi)$, as follows. 
Because $\connH_0$ is flat (hence a foliation according to Frobenius' theorem), the monodromy $\Psi_\delta$ of the connection $\connH_0$ over a (smooth immersed) curve $\delta$ in $\Sigma$ depends only on $[\delta]\in \pi_1\left(\Sigma\right)$. 
Moreover, it is also a contactomorphism of the fibers, by \Cref{LemmaContFibBund}. 
Hence, for each $c\in \pi_1\left(\Sigma\right)$, one can define $\rho(c)\coeq\Psi_\delta$, where $\delta$ is any (smooth immersed) representative of $c$.
\\
Let now $\pi_\Sigma\co \Sigmatilde\rightarrow \Sigma$ be the universal cover of $\Sigma$, and consider the map $F:M\times \Sigmatilde \rightarrow V$ covering $\pi_\Sigma$ given by $F(q,[\gamma]):=\rho_c(q)$.
Here, we see $\Sigmatilde$ as the set of arcs $\gamma$ on $\Sigma$ starting at $b$, up to homotopy. 
The differential of $F$ sends the connection $\{0\}\oplus T\Sigmatilde$ of $ M\times\Sigmatilde \rightarrow\Sigmatilde$ to the connection $\connH_0$ of $\pi:V\rightarrow \Sigma$, and the contact structure $\xi\oplus\{0\}$ on the fiber of $\Sigmatilde\times M$ over $\phat\in\Sigmatilde$ to the contact structure $\xi_p$ of the fiber $M_p$ of $V$ over $p=\pi_\Sigma(\phat)$.
\\
Moreover, if we denote by $\rhotilde$ the diagonal action of $\pi_1\left(\Sigma\right)$ on $M\times \Sigmatilde$ induced by the natural action on the second factor and by the action $\rho$ on the first factor, $F$ induces an isomorphism $f:M\times_{\rhotilde}\Sigma\rightarrow V$ of fiber bundles over $\Sigma$,	where $M\times_{\rhotilde}\Sigma$ is the quotient of $M\times\Sigma$ by $\rhotilde$. 
Notice also that on $M\times_{\rhotilde}\Sigma\rightarrow\Sigma$ there are natural $\xifib^{\rhotilde}$ and $\connH_0^{\rhotilde}$ induced, respectively, by $\xi\oplus\{0\}$ and $\{0\}\oplus T\Sigma$ on $M\times\Sigmatilde$. 
Because of the properties of $F$, we also have that the differential of $f$ sends $\xifib^{\rhotilde}$ and $\connH_0^{\rhotilde}$ respectively to $\xifib$ and $\connH_0$.
In other words, $f$ gives the desired ``nice'' presentation of $(\pi,\xifib,\connH_0)$ in terms of the monodromy $\rho$.

The second reason for restricting to the class of flat contact bundles is the following: using the notion of potential from \Cref{SubSecComparingContConnect}, given a flat $\left(\pi\co V\rightarrow\Sigma,\xifib,\connH_0\right)$, we can give an explicit criterion that tells whenever any other contact bundle on it (inducing the same $\xifib$, hence described by a contact connection $\connH$) defines a contact structure on the total space $V$.
More precisely, using \Cref{EqCompCurv} (with $R_0=0$), we can rephrase Lerman \cite[Proposition 3.1]{Ler04} in the following computational-friendly way:
\begin{prop}
	\label{PropWhenContFibBundIsContStrUsingPotentials}
	On a flat contact fiber bundle $\left(\pi\co V^{2n+1}\rightarrow\Sigma^2,\xifib,\connH_0\right)$,
	%$\flatcontbundconn$ with $\Sigma$ surface,
	a contact connection $\connH$ with potential $A$ gives a contact structure $\eta$ on the total space if and only if, for all $b$ in $\Sigma$ and all oriented basis $(u,v)$ of $T_b\Sigma$, the vector field $\dnabla A(u,v)+[A_u,A_v]$ on $M_b$ is a negative contact vector field for $(M_b,\xi_b)$. 
\end{prop}
Recall that a contact vector field is called \emph{negative} if it is everywhere negatively transverse to the contact structure.
\\

%%%%%%%%%%%%%%%%%%%%%%%%%%%%%%%%%%%%%%%%%%%%%%%%%%%%%%%%%%%%%%%%%%%%%%%%%%%%%%%%%%%%

%\subsection{Contact deformations of flat contact bundles}
%\label{SubSubSecContactizationFlatContBund}

%Fix for the rest of this section a flat contact fiber bundle $\flatcontbund$.

In \Cref{SubSubSecBourgContStrRev}, we will use \Cref{PropWhenContFibBundIsContStrUsingPotentials} to study the following objects:
\begin{definition}
	\label{DefContactizationFlatContBund}
	Let $\left(\pi\co V\rightarrow\Sigma,\eta_0=\xifib\oplus\connH_0\right)$ be a flat contact bundle.
	We say that a contact fiber bundle $\eta$ on $\pi$ is a \emph{contact deformation} of $\eta_0$ 
	%$\flatcontbund$
	if it defines a contact structure on the total space $V$ and if there is a smooth family of contact fiber bundles $\left(\eta_s\right)_{s\in[0,1]}$ starting at $\eta_0$, ending at $\eta_1:=\eta$ and satisfying:
	\begin{enumerate}
		\item \label{Item1DefContactizationFlatContBund} for all $s\in[0,1]$, $\eta_s\cap \ker(d\pi)=\xifib$;
		%\item \label{Item1DefContactizationFlatContBund} for all $p\in\Sigma$ and all $s\in[0,1]$, $\eta_s$ defines the same $\xi_p$ on the fiber $M_p$;
		\item \label{Item2DefContactizationFlatContBund}  for all $s>0$, $\eta_s$ defines a contact structure on $V$.
	\end{enumerate}
\end{definition}

By point \ref{Item1DefContactizationFlatContBund}, a contact deformation is equivalent to a path of contact connections $\connH_s$ interpolating between $\connH_0$ and $\connH$.\\

We point out that this definition is ``non-empty'', i.e. given a flat contact fiber bundle $(\pi\co V\rightarrow\Sigma,\eta_0=\xifib\oplus\connH_0)$, not all the contact fiber bundles for the same underlying fibration with contact fibers $(\pi,\xifib)$ are contact deformations of $\eta_0$.\\
For instance consider the contact fiber bundle structure on $\tore=\cercle\times\torus$ which is given by the kernel $\eta$ of $\alpha = d\theta + \cos(\theta) dx - \sin(\theta)dy$, where $\theta\in\cercle$ and $(x,y)$ are coordinates on $\torus$.
This contact fiber bundle structure is a contact deformation of the flat contact fiber bundle structure given by $\eta_0=\ker\left(d\theta\right)$: the deformation is given by $\alpha_t\coeq d\theta + t \cos\theta dx - \sin\theta dy$, with $t\in[0,1]$.\\ 
We point out that, by Giroux \cite[Lemma 10]{Gir99}, $\eta$ admits prelagrangian tori only in the isotopy class of $\{pt\}\times\torus$. 
Take now a diffeomorphism $\psi$ of $\tore$ sending $(\theta,x,y)$ to $(\theta + x, x, y)$. Then, $\psi^* \eta$ is still transverse to the $\cercle$ factor, hence it is a contact fiber bundle on the chosen fibration, and obviously it still defines a contact structure on the total space. Though, it has prelagrangian tori in an isotopy class which is different from that of the prelagrangian tori of $\eta$. According to Vogel \cite[Proposition 9.9]{Vog16b}, this implies that $\phi^* \eta$ cannot be a contact deformation of $\eta_0= \{0\}\oplus T\torus\subset T\left(\cercle\times\torus\right)$.\\

We also remark that, even though the above definition is of a very similar flavor to \Cref{DefContPullBack}, the objects they define behave differently. 
For instance, there is no uniqueness up to isotopy for contact deformations. \\
Indeed, if we take again the fiber bundle $\pi\co\tore=\cercle\times\torus\rightarrow\torus$ where we see the fibers as contact manifolds $\left(\cercle,\ker\left(d\theta\right)\right)$, then the flat contact bundle defined by $\eta_0=\ker(d\theta)$ on $\pi$ actually admits as contact deformations every contact structure on $\tore$ defined by $\alpha_n:= d\theta + \cos(n\theta)dx-\sin(n\theta)dy$. 
Though, these are not isotopic one to the other as contact fiber bundles defining contact structures on the total space. 
Indeed, they are not even isomorphic as contact structures on $\tore$, because they have different Giroux torsion (see Giroux \cite{Gir99}).

%%%%%%%%%%%%%%%%%%%%%%%%%%%%%%%%%%%%%%%%%%%%%%%%%%%%%%%%%%%%%%%%%%%%%%%%%%%%%%%%%%%%

\subsection{Bourgeois' construction revisited}
\label{SubSubSecBourgContStrRev}

The aim here is to use what we defined in the previous sections to generalize the construction by Bourgeois recalled in \Cref{SubSecBourgConstr}. Let's start by reformulating it with this new terminology. 

Let $(M^{2n-1},\xi)$ be a contact manifold and $(\pi\colon M\times\torus\rightarrow\torus,\xi\oplus T\torus)$ be a flat contact  bundle.
%We start from the trivial fiber bundle $M\times\torus\rightarrow\torus$ with fixed contact fiber $(M,\xi=\ker(\alpha))$, and we consider the flat contact fiber bundle structure $\xi\oplus T\torus$ on the total space of the fibration.
Once an open book decomposition $(B,\varphi)$ supporting $\xi$ on $M$ and a particular adapted contact form $\beta$ are fixed, consider a function $\phi = (\phi_1,\phi_2):M\rightarrow \real^2$ as in the statement of \Cref{ThmBourgeoisBis}. 
Now take the contact vector fields $X$ and $Y$ on $(M,\xi)$ associated, respectively, to the contact hamiltonians $\phi_1$ and $-\phi_2$ via the contact form $\beta$, and consider the potential $A\coeq -X\otimes d\theta_1 - Y\otimes d\theta_2$, where $(\theta_1,\theta_2)$ are coordinates on $\torus=\real^2/\integ^2$. %\sfrac{\real^2}{\integ^2}$. 
A direct computation shows that the contact fiber bundle associated to $A$ is the kernel of the contact form $\alpha = \beta + \phione d\theta_1 - \phitwo d\theta_2$ given by \Cref{ThmBourgeoisBis}. 
\\
Notice also that $\torus$ has a natural (flat and) torsion-free $\nabla^{\torus}$, inherited by the standard $\nabla^{\real^2}$ on $\real^2$, and such that $\nabla^{\torus}_{\partial_{\theta_i}}\partial_{\theta_j}=0$. Then, because $X$ and $Y$ are independent from the point of $\torus$ in the product $M\times\torus$, it is easy to check that $A$ is $\overline{\nabla}$-parallel, i.e. that $\overline{\nabla} A =0$ (see \Cref{SubSecComparingContConnect} for the definition of $\overline{\nabla}$).
%We also remark that, because $X$ and $Y$ are independent from the point of $\torus$ in the product $M\times\torus$, the $2$-form 
%$\doA$
%$\dnabla A$
%is zero everywhere. 
%Because $\alpha$ is a contact form, \Cref{PropWhenContFibBundIsContStrUsingPotentials} then tells us that %$[A,A]$ takes values in the space of \emph{negative} contact vector fields of the fibers $(M,\xi)$. In particular, 
%$[X,Y]$ is a negative contact vector field on the fiber $(M,\xi)$.

We can then give the following definition:
\begin{definition}
	\label{DefStrongBourgContStr}
	Let $(\pi\co V\rightarrow\Sigma,\eta_0=\xifib\oplus\connH_0)$ be a flat contact bundle, and consider a torsion-free covariant derivative $\nabla^\Sigma$ on $\Sigma$. We call \emph{strong Bourgeois contact structure} each contact structure on the total space $V$ given by a contact fiber bundle structure $\eta$ on $V$ with $\overline{\nabla}$-parallel potential $A$.
\end{definition}

Notice that, generalizing \Cref{RmkEpsBourgContStr}, each strong Bourgeois contact structure $\eta$ is a contact deformation of the underlying flat contact bundle $\eta_0$.
More precisely, if $A$ is the potential associated to $\eta$ with respect to $\eta_0$, the deformation is just given by the family of potentials $(sA)_{s\in [0,1]}$.  

We also point out that \Cref{DefStrongBourgContStr} is a non-trivial generalization of the Bourgeois construction, i.e. the class of strong Bourgeois contact structures is not exhausted by the examples on $M\times\torus$ from Bourgeois \cite{Bou02}: 
\begin{prop}
	\label{LemmaExistBourgContStr}
	There is a flat contact fiber bundle $(\pi\co V\rightarrow\torus,\eta_0)$ that admits a strong Bourgeois contact structure (for the standard flat $\nabla^{\torus}$ on $\torus$) and is non-trivial, i.e. not isomorphic, as flat contact fiber bundle, to $(\pi\co M\times\torus\rightarrow\torus,\xi_M\oplus T\torus)$.
\end{prop}

\Cref{LemmaExistBourgContStr} is a consequence of this generalization of \Cref{ThmBourgeoisBis}:
\begin{lemma}
	\label{LemmaBourgContStrFromInvOBD}
	Let $(M,\xi)$ be a contact manifold, $G$ a subgroup of the group of contactomorphisms of $(M,\xi)$, and $\rho:\pi_1(\torus)\rightarrow G$ a group homomorphism.
	Suppose that there is a $G$-invariant function $\phi=(\phione,\phitwo)\co M\rightarrow \real^2$ defining a ($G$-invariant) open book $(B,\varphi)$ on $M$ supporting $\xi$.\\
	Let's also denote by $\beta$ a $G$-invariant contact form for $\xi$ on $M$ such that $d\beta$ is symplectic on the fibers of $\varphi$, and by $\eta_0$ the flat contact bundle induced on $\pi \co M\times_{\rhotilde}\torus\rightarrow\Sigma$ by the flat contact bundle $\xi\oplus T\real^2$ on $M\times\real^2\rightarrow\Sigma$. Here, $\rhotilde$ is the action of $\pi_1(\torus)$ on $M\times\real^2$ given by $\rho$ on the first factor and by the natural action on the universal cover $\real^2\rightarrow \torus$ on the second factor.\\
	Then, the hyperplane field $\eta$ on $M\times_{\rhotilde}\torus$, induced by $\ker(\beta + \phione d\theta_1 - \phitwo d\theta_2)$ on $M\times\real^2$, is a strong Bourgeois contact structure on the flat contact bundle $(\pi\co M\times_{\rhotilde}\torus\rightarrow\torus,\eta_0)$ equipped with the standard flat $\nabla^{\torus}$ on $\torus$.
\end{lemma}

\begin{proof}[Proof (\cref{LemmaBourgContStrFromInvOBD})]
	The form $\beta + \phione d\theta_1 - \phitwo d\theta_2$ on $M\times\real^2$ defines a contact structure $\widehat{\eta}$ on $M\times\real^2$.
	This follows from the same computations as those in \cite{Bou02}.
	Moreover, it is invariant under the action $\rhotilde$.
	Hence, it induces a well defined contact structure $\eta$ on the codomain $M\times_{\rhotilde}\torus$ of the quotient map $q\co M\times\real^2 \rightarrow M\times_{\rhotilde}\torus$. 
	\\
	Finally, we need to prove that $\eta$ is indeed a strong Bourgeois contact structure. 
	Being $\overline{\nabla}$-parallel is a local condition, hence it is enough to prove that the potential $\widetilde{A}$ of $\widetilde{\eta}=q^*\eta$ is parallel with respect to the connection $q^*\overline{\nabla}$, pullback of $\overline{\nabla}$ to $M\times\real^2$ via $q$.
	Now, an explicit computation gives $\widetilde{A}=X\otimes dx + Y \otimes dy$, where $(x,y)$ are coordinates on $\real^2$ and $X,Y$ are contact vector fields on $(M,\xi)$ with contact hamiltonians (via $\beta$) respectively $-\phione,\phitwo$. 
	Hence, $X$ and $Y$, as functions from $M\times\torus$ to the space $\VF(M,\xi)$ of contact vector fields for $(M,\xi)$, are independent of the coordinates on $\torus$.
	This implies that $\widetilde{A}$ is $(q^*\overline{\nabla})$-parallel. 
\end{proof}

\begin{proof}[Proof (\Cref{LemmaExistBourgContStr})]
	We recall that van Koert--Niederkr\"{u}ger \cite{NieVKo05} exhibited an explicit open book decomposition for each Brieskorn manifold $W_{k}^{2n-1}\subset \compl^{n+1}$, with supporting form $\alpha_k$.
	We invite the reader to consult their article for the details. 
	What's important for us is that the adapted open book decomposition is defined by a map $\phi\co W_{k}^{2n-1}\rightarrow\real^2$ which is invariant under the action of the subgroup $\SO(n)$ of the group of strict contactomorphisms for the strict contact manifold $(W_{k}^{2n-1},\alpha_k)$. 
	More precisely, if $(z_0,\ldots,z_n)$ are the coordinates of $\compl^{n+1}$, the action of $\SO(n)$ on $W_{k}^{2n-1}\subset\compl^{n+1}$ is given by the identity on the $z_0$--coordinate and by matrix multiplication on $(z_1,\ldots,z_n)$.
	For simplicity, we denote here the couple $(W_{k}^{2n-1},\alpha_k)$ by $(M,\beta)$.
	
	Let now $\rho\co \pi_1\left(\torus\right)\rightarrow\SO(n)$ be defined by $\rho(a,b)=a \cdot f$ for each $(a,b)\in\integ^2=\pi_1\left(\torus\right)$, where $f$ is any element of $\SO(n)$ of order $2$. 
	Then, \Cref{LemmaBourgContStrFromInvOBD} 
	tells us that the $\eta$ on $M\times_{\rhotilde}\torus$, induced by $\ker(\beta + \phione d\theta_1 - \phitwo d\theta_2)$ on $M\times\real^2$, is a strong Bourgeois contact structure on the flat contact bundle $(\pi\co M\times_{\rhotilde}\torus\rightarrow\torus,\eta_0)$.
	Here, $\eta_0$ is the flat contact bundle induced by $\xi\oplus T\real^2$ on $M\times\real^2\rightarrow\real^2$.
	
	%A direct computation shows also that the potential $\widetilde{A}$ associated to the kernel of $\alpha:=\beta + \phione dx - \phitwo dy$ on $M\times\real^2$ with respect to the trivial flat contact bundle $\widetilde{\eta}_0\coeq\xi\oplus T\real^2$ is given by $\widetilde{A}=-X\otimes dx - Y\otimes dy$, where $X,Y$ are respectively the contact vector fields on $(M,\xi)$ such that $\beta(X)=\phione$ and $\beta(Y)=-\phitwo$.
	%In particular, it satisfies $d_0\widetilde{A}=0$. Hence, the potential $A$ of $\eta=\ker(\beta)$ on $V$ (with respect to the flat connection $\eta_0$ on $V$ induced by $ \widetilde{\eta}_0 $ on $M\times\real^2$) will also satisfy $\doA=0$. In other words, this is an example of Bourgeois contact structure on the flat contact fiber bundle $(V,\torus,M,\eta_0)$, as wanted.
	
	The only thing left to show is that $(\pi\co M\times_{\rhotilde}\torus\rightarrow\torus,\eta_0)$ is not isomorphic to the trivial flat contact bundle $(p\co M\times\torus\rightarrow\torus,\xi\oplus T\torus)$.\\
	The connection $\connH_0$ associated to $\eta_0$ defines a foliation $\fol_0$ by tori $\torus$ on $V$, which is also transverse to the fibers of $\pi\co V\rightarrow \torus$. 
	Moreover, because of our particular choice of $\rho\co\pi_1(\torus)\rightarrow\SO(n)$, each leaf $L$ of $\fol_0$ intersects every fiber twice.
	Now, the connection $T\torus$ on the trivial bundle $p\co M\times\torus\rightarrow\torus$ gives a foliation $\fol_1$ with leaves $\{pt\}\times\torus$, which only intersects each fiber once.
	In particular, there is no isomorphism $\Psi$ of fiber bundles (equipped with connections) over $\torus$ between $(\pi\co M\times_{\rhotilde}\torus\rightarrow\torus,\connH_0)$ and $(p\co M\times\torus\rightarrow\torus,\{0\}\oplus T\torus)$. Indeed such a $\Psi$ would send $\fol_0$ to $\fol_1$, but this contradicts the fact that their leaves intersect the fibers a different number of times.
\end{proof}

Even though strong Bourgeois contact structures are a non-trivial extension of the examples from \cite{Bou02}, we believe that this class of contact structures is still, in a certain sense, too ``rigid''. 
The first (somewhat philosophical) reason is that their definition depends on the choice of a torsion-free $\nabla^\Sigma$ on the base $\Sigma$, that is an auxiliary data with respect to the underlying flat contact bundle structure.
The second (much more concrete) reason is given by the following converse to \Cref{LemmaBourgContStrFromInvOBD}:
\begin{prop}
	\label{PropUniqStrongBourgContStr}
	Let $(M,\xi)$ be a contact manifold and consider the flat contact bundle $(\pi:M\times_{\rhotilde}\torus\rightarrow\torus,\eta_0)$, where $\eta_0$ is induced by $\xi\oplus T\real^2$ on the cover $M\times\real^2$ of $M\times_{\rhotilde}\torus$.
	Let also $\eta$ be a strong Bourgeois contact structure on $(\pi,\eta_0)$, equipped with the standard flat $\nabla^{\torus}$ on $\torus$.
	Then, there is a $\Image(\rho)$-invariant function $\phi\co M\rightarrow \real^2$ that defines an open book decomposition of $M$ supporting $\xi$, and such that the given $\eta$ is the result of the application of \Cref{LemmaBourgContStrFromInvOBD} with these choices of $(M,\xi)$, $\rho$ and $\phi$ (and $G=\Image(\rho)$).
\end{prop}
Recall from \Cref{SubSubSecFlatContFibBund} that every flat contact bundle $(\pi\co V\rightarrow\torus,\eta_0)$ is isomorphic (as flat contact bundle) to $(\pi'\co M\times_{\rhotilde}\torus\rightarrow\torus,\eta_0')$, where $\eta'_0$ is induced by $\xi\oplus T\real^2$ on the cover $M\times\real^2$ of $M\times\torus$.
Thus, \Cref{PropUniqStrongBourgContStr} says that the examples given by \Cref{LemmaBourgContStrFromInvOBD} are actually all the possible strong Bourgeois contact structures on $(\pi,\eta_0)$, equipped with the standard flat $\nabla^{\torus}$.
\begin{proof}[Proof (\Cref{PropUniqStrongBourgContStr})]
	Consider the following natural commutative diagram: 
	\begin{center}
		\begin{tikzcd}
			(M\times \real^2, \xi\oplus T\real^2) \ar[r, "Q"] \ar[d, "\pr_2"] & (M\times_{\rhotilde}\torus,\eta_0) \ar[d, "\pi"] \\
			\real^2 \ar[r, "q"] & \torus= \real^2 / \integ^2 %\sfrac{\real^2}{\integ^2}
		\end{tikzcd}
	\end{center}
	with $Q$ and $q$ the natural quotients maps (see \Cref{SubSubSecFlatContFibBund}). 
	This induces the commutative diagram
	\begin{center}
		\begin{tikzcd}
			\VFfib(M\times\real^2,\xi\oplus\{0\})=\VF(M,\xi)\times \real^2 \ar[r, "G"] \ar[d, "\pr_1"] & 
			\VFfib(M\times_{\rhotilde}\torus,\eta_0\cap \ker(d\pi)) \ar[d] \\
			\real^2 \ar[r, "q"] & \torus
		\end{tikzcd}
	\end{center}
	where $\VF(M,\xi)$ denotes the space of contact vector fields of $(M,\xi)$ and $G$ is the restriction of $dQ$ to the fibers of $\pr_2$.
	Consider the pullbacks $\widehat{\nabla}$ and $\widehat{A}$ of $\nabla$ and $A$ via $(G,q)$.
	Because of the particular choice of $\nabla^{\torus}$, the fact that $A$ is $\widehat{\nabla}$-parallel translates to the fact that $\widehat{A}$ is $\real^2$-invariant.
	A straightforward computation then shows that $d_{\widehat{\nabla}}\widehat{A}=0$, so that $R_{\widehat{A}}=[\widehat{A},\widehat{A}]$.
	\\
	Now, $\widehat{A}$ is the potential associated to a contact structure on $M\times\real^2$, because the same is true for $A$ on $M\times_{\rhotilde}\torus$.
	Let then $X\coeq \widehat{A}(\partial_x)$ and $Y\coeq \widehat{A}(\partial_y)$, for any choice coordinates $(x,y)\in\real^2$. 
	\Cref{PropWhenContFibBundIsContStrUsingPotentials} then tells that $[X,Y]$ is negatively transverse to $\xi$ everywhere on $M$. Moreover, as $\widehat{A}$ is the pullback of $A$, both $X$ and $Y$ are $\Image(\rho)$-invariant.
	\Cref{PropGiroux1} then gives the desired $\Image(\rho)$-invariant open book decomposition $\phi\co M\rightarrow \real^2$. 
	More precisely, we use here the fact that the proof of \Cref{PropGiroux1}.\ref{Item3PropGiroux1} gives an invariant function $M\rightarrow\real^2$, provided the original contact vector fields $X$ and $Y$ are both invariant.
\end{proof}

We then propose the following generalization of \Cref{DefStrongBourgContStr}:
\begin{definition}
	\label{DefBourgContStr}
	Let $(\pi\co V\rightarrow\torus,\eta_0=\xifib\oplus\connH_0)$ be a flat contact fiber bundle. We call \emph{Bourgeois contact structure} each contact structure on the total space $V$ given by a contact fiber bundle structure $\eta$ on $V\rightarrow \Sigma$ with potential $A$ that is $\dnabla$-closed, i.e. such that $\dnabla A = 0$.
\end{definition}

Notice that the condition $\dnabla A = 0$ in \Cref{DefBourgContStr} above is actually the same as the condition $\frac{1}{\epsilon} R_\epsilon\rightarrow 0$, for $\epsilon\rightarrow0$, used to introduce Bourgeois contact structures in \Cref{SecIntro}.
\\
Indeed, according to \Cref{EqCompCurv}, the	curvature $R$ of a Bourgeois contact structure is just $\dnabla A+[A,A]$, where $A$ is its potential.
In particular, this curvature has two terms which behave differently under rescaling $A\mapsto\epsilon A$, for $\epsilon>0$.
The term $\dnabla A$ is rescaling linearly in $\epsilon$, whereas $[A,A]$ is rescaling quadratically in it.
Then, if we denote by $R_\epsilon$ the curvature associated to the connection $\connH_\epsilon$ of potential $\epsilon A$ with respect to $\eta_0$, the condition $\dnabla A=0$ is equivalent to the fact that $\frac{1}{\epsilon} R_\epsilon\rightarrow 0$ for $\epsilon\rightarrow0$. %, which was the condition used to introduce Bourgeois contact structures in \Cref{SecIntro}.
\\

We point out that, as announced before \Cref{DefBourgContStr} (and as the terminology suggests), strong Bourgeois structures are also Bourgeois structures.
Indeed, by \Cref{EqExtCovDer}, we have $(\overline{\nabla}A)(W,U) = \nabla_U(A_W) - A(\nabla^{\torus}_U W)$ for all vector fields $U,W$ on $\torus$. 
Using the fact that $A$ is  $\overline{\nabla}$-parallel, we compute
\begin{align*}
\dnabla A (U,W) & = \nabla_U (A_W) - \nabla_W (A_U) - A_{[U,W]} \\
& = A(\nabla^{\torus}_U W) - A(\nabla^{\torus}_W U) - A_{[U,W]} \\
& = A_{T(U,W)} = 0 \text{ ,}
\end{align*}
where $T\in\Omega^2(\torus;T\torus)$ is the torsion of $\nabla^{\torus}$, which is by assumption zero.

We point out, however, that a direct analogue of \Cref{PropUniqStrongBourgContStr} is not true for Bourgeois contact structures.
For instance, given a strong Bourgeois contact structure $\epsilon$ with potential $A$ on any flat contact bundle $(\pi\co V\rightarrow\Sigma,\eta_0=\xifib\oplus \connH_0)$, if $A_0\in\Omega^1(\Sigma;\VFfib(V,\xifib))$ is any other $\dnabla$-closed potential, not necessarily inducing a contact structure on the total space $V$ (these are not hard to find, for instance in the case of $(\pi\co M\times\torus\rightarrow\torus,\xi\oplus T\torus)$), then, for $\epsilon>0$ small enough, the perturbation $A+\epsilon A_0$ gives a Bourgeois contact structure $\eta_\epsilon$.
Though, such $\eta_\epsilon$'s do not necessarily come from the construction in \Cref{LemmaBourgContStrFromInvOBD}. 
In other words, the class of Bourgeois contact structures is bigger than the one of strong Bourgeois contact structures.

This being said, the motivation behind \Cref{DefBourgContStr} doesn't only consist in the fact that it's a strict generalization of \Cref{DefStrongBourgContStr}.
Indeed, we now show that the condition $\dnabla A=0$ above, while being general enough to be satisfied by a class of contact structures strictly larger than those given by the construction in \Cref{LemmaBourgContStrFromInvOBD}, is also strong enough to ensure some nice properties from the points of view of contact deformations, weak fillability and adapted open book.
\\

We start by showing that each Bourgeois contact structure $\eta$ is in particular a contact deformation of the underlying flat contact bundle $\eta_0$.\\
Indeed, we have the natural path of contact bundle structures $\left(\eta_t\right)_{t\in[0,1]}$ that is given by the potential $A_{t}:= t A$ with respect to $\connH_{0}$, where $A$ is the potential of $\eta$. This has the wanted starting and ending points and gives a contact structure $\eta_{t}$ for $t>0$, according to \Cref{PropWhenContFibBundIsContStrUsingPotentials}, because $\dnabla A_{t} = t \dnabla A$ is zero and, for any $b\in\Sigma$, oriented basis $(u,v)$ of $T_b\Sigma$ and $t>0$, $[A_t,A_t](u,v)=t^2[A_u,A_v]$ is negatively transverse to $\xi_b=\eta\cap TM_b$. \\
This property is a generalization of the fact that strong Bourgeois structures (which includes the examples in \cite{Bou02}) are contact deformations of the trivial flat contact bundle on $M\times\torus$. %the end of \Cref{SubSecBourgConstr}).
\\

The study of weak fillability of Bourgeois contact structures is postponed to \Cref{SubSecBourgContStrWeakFill}. 
There, \Cref{Prop1} states that if $(M,\xi)$ is weakly fillable then a Bourgeois contact structure $\eta$ on the flat contact bundle $(\pi\co M\times\torus\rightarrow\torus,\xi\oplus T\torus)$ is weakly fillable too.
(Notice that the particular case of the contact structures obtained as in \cite{Bou02} is covered by Massot--Niederkr\"uger--Wendl \cite[Example 1.1]{MNW13} and Lisi--Marinkovi\'c--Niederkr\"uger \cite[Theorem A.a]{LisMarNie18}.)
This stability of weak fillability is also true in a more general case, as stated in \Cref{Prop1General}.
\\

%The study of weak fillability of Bourgeois contact structures is postponed to \Cref{SubSecBourgContStrWeakFill}, where \Cref{Prop1} states that if $(M,\xi)$ is weakly fillable then a Bourgeois contact structure $\eta$ on the flat contact bundle $(M\times\torus,\torus,M,\xi\oplus T\torus)$ is weakly fillable too. This stability of weak fillability is also true in a more general case, as stated in \Cref{Prop1General}.

As far as adapted open book decompositions are concerned, we prove the following: given a Bourgeois contact structure $\eta$ on the flat contact bundle $(\pi\co V\rightarrow\Sigma,\eta_0)$, we can ``naturally'' associate to each point $b$ of $\Sigma$ an open book decomposition of the fiber $M_b$ supporting the contact structure $\xi_b$.
This crucially relies on \Cref{PropGiroux1} on pairs of contact vector fields and supporting open book decomposition.
In order to give a precise statement, let's introduce some notations.
\\
Consider a smooth contact bundle $\eta$ on $X\rightarrow Y$, where $X$ is not assumed to be closed.
Denote by $\Lambda$ the space of maps $\Phi\co X\rightarrow\real^2$ such that, for each $y\in Y$:
\begin{enumerate}[label=\roman*.]
	\item \label{Item1Lambda}  the restriction $\phi_y\coeq \Phi\vert_{\pi^{-1}(y)}\co\pi^{-1}(y)\rightarrow\real^2$ is transverse to $\{0\}\subset\real^2$,
	\item \label{Item2Lambda} the map $\frac{\phi_y}{\norm{\phi_y}}\co \pi^{-1}(y) \setminus \phi_y^{-1}(0)\rightarrow \cercle$ is a fibration, 
	\item \label{Item3Lambda} $(\phi_y^{-1}(0),\frac{\phi_y}{\norm{\phi_y}})$, which is an open book decomposition of $\pi^{-1}(y)$ according to points \ref{Item1Lambda}, \ref{Item2Lambda}, is moreover adapted to the contact structure $\eta\cap T\left(\pi^{-1}(y)\right)$. 
\end{enumerate}
Notice that this space $\Lambda$ comes endowed with a natural $C^\infty$-topology induced by that on the space of functions $X\rightarrow\real^2$ in which it is contained.
Consider then the quotient $\sfrac{\Lambda}{\sim}$ of $\Lambda$ by the relation $\sim$ defined as follows: $\Phi_1,\Phi_2\in\Lambda$ are equivalent via $\sim$ if there is a positive function $f:X\rightarrow\real$ such that $\Phi_2=f\Phi_1$. Notice that $\sfrac{\Lambda}{\sim}$ inherits a natural topology as quotient of the topological space $\Lambda$.
We then call \emph{smooth $Y-$family of open books on $X$ (adjusted to $\eta$)} each element of $\sfrac{\Lambda}{\sim}$.\\ 
Remark also that if we have a contact bundle $\eta$ on a smooth fiber bundle $\pi\co X \rightarrow Y$ and $f\co Z\rightarrow Y$ is a smooth map, we can define the \emph{pullback contact bundle} $f^*\eta$ on the pullback bundle 
\begin{center}
	\begin{tikzcd}
		f^* X \coeq \left\{(z,x)\in Z\times X \vert \, f(z)= \pi(x) \right\} \ar[r, "\pr_X"] \ar[d, "\pr_Z"] & X \ar[d, "\pi"] \\
		Z \ar[r, "f"] & Y
	\end{tikzcd}
\end{center}
as the vector sub-bundle $\{W\in T (f^*X) \vert\, d (\pr_X)(W) \in \eta \}$ of $T(f^*X)$, where $\pr_X, \pr_Z$ are the projections of $Z\times X$ on the first and second factors respectively. 
This $f^*\eta$ is indeed a contact bundle because its trace on each fiber $(\pr_Z)^{-1}(z)\cap f^*X = \{z\}\times\pi^{-1}_{f(z)}$ of $\pr_Z\co f^*X \rightarrow Z$ is exactly $\{0\}\oplus \eta_{f(z)}$.
\\
Lets now go back to the specific case of Bourgeois contact structure $\eta$ on a flat contact bundle $(\pi\co V\rightarrow\Sigma,\eta_0)$.
Denote, for all $b\in\Sigma$,  $(M_b,\xi_b)$ the contact fiber over $b$, i.e. $M_b \coeq \pi^{-1}(b)$ and $\xi_b \coeq \eta_0\cap T M_b$. %, where $\pi\co V\rightarrow M$ is the given fiber bundle.
We then call \emph{fiber adapted open book} any open book $(K,\varphi)$ on a fiber $M_b$ which supports the respective contact structure $\xi_b$. \\ % We say that a couple $(K,\varphi)$ is a  if there is a point $b\in\Sigma$ such that $(K,\varphi)$ is an open book decomposition of $M_b$ supporting $\xi_b$.\\
Denote finally by $\pr\co F\Sigma\rightarrow \Sigma$ the frame tangent bundle of $\Sigma$, i.e. the (principal) bundle over $\Sigma$ with fiber over $b\in\Sigma$ given by the set of all oriented basis of $T_b\Sigma$.
We can finally state the desired result on Bourgeois structures and open books:
\begin{prop}
	\label{CorBourgContStr}
	Given a Bourgeois contact structure $\eta$ on the flat contact bundle $(\pi\co V\rightarrow\Sigma,\eta_0)$, there is a map
	$$\Psi_\eta \co F\Sigma \rightarrow \left\{ 
	%\begin{array}{c}
	\text{fiber adapted open book} %\\
	%\text{decomposition  } (V,\Sigma,M,\eta_0)
	%\end{array}
	\right\} $$
	verifying the following properties:
	\begin{enumerate}[label=\roman*.]
		\item \label{Item1CorBourgContStr} $\Psi_\eta$ sends, for all $b\in \Sigma$, each positive basis of $T_b\Sigma$ to an open book decomposition of $M_b$ adapted to $\xi_b$;
		\item \label{Item2CorBourgContStr} for each smooth path $\gamma \co [0,1]\rightarrow F\Sigma$, the composition $$\Psi_\eta \circ \gamma \co [0,1] \rightarrow \left\{ 
		%\begin{array}{c}
		\text{fiber adapted open books} 
		\right\} $$
		describes a smooth $[0,1]-$family of open books on $\gamma^*\pr^* V$ adjusted to $\gamma^*\pr^*\eta$. 
	\end{enumerate}
	\begin{center}
		\begin{tikzcd}
			\left(\gamma^*\pr^* V, \gamma^*\pr^* \eta \right) \ar[r] \ar[d] & \left(\pr^* V, \pr^* \eta\right)  \ar[r] \ar[d] & \left(V,\eta\right) \ar[d, "\pi"] \\ 
			\left[0,1\right] \ar[r, "\gamma"] & F\Sigma \ar[r, "\pr"] & \Sigma
		\end{tikzcd}
	\end{center}
\end{prop}

From the above result, we can deduce a more precise version of \Cref{PropIntroIsotopyClassOBD} stated in the introduction:
\begin{cor}
	\label{RmkCorBourgContStr}
	The map $\Psi_\eta$ in \Cref{CorBourgContStr} induces a well defined 
	$$\psi_\eta \co \Sigma \rightarrow \bigslant{\left\{ 
		%\begin{array}{c}
		\text{fiber adapted open books} %\\
		%\text{decomposition  } (V,\Sigma,M,\eta_0)
		%\end{array}
		\right\}}{\sim} \text{ ,} $$
	where $(K_0,\varphi_0)\sim(K_1,\varphi_1)$ if they are both adapted open books on a same fiber $(M_b,\xi_b)$ and there is an isotopy $(f_t)_{t\in[0,1]}$ of the fiber $M_{b}$, starting at $\phi_0=\Id$, such that $K_1=f_1(K_0)$, $\varphi_1=\varphi_0\circ f_1^{-1}$ and $\left(f_t\left(K_0\right),\varphi_0\circ f_t^{-1}\right)$ is an open book of $M_b$ adapted to $\xi_b$.
	In other words, $\eta$ uniquely determines an isotopy class of adapted open book decompositions for each fiber $(M_b,\xi_b)$ of $(\pi\co V\rightarrow\Sigma,\eta_0)$.\\
	Moreover, if $\eta=\ker\alpha$ is the Bourgeois contact structure on $(\pi\co M\times\torus\rightarrow\torus,\xi\oplus T\torus)$ given by \Cref{ThmBourgeoisBis} starting from an open book $(B,\varphi)$ for $(M,\xi)$, then the corresponding map $\psi_\eta$ sends each $b\in\torus$ to an isotopy class of adapted open books on $(M_b,\xi_b)$ that (via the natural identification $(M_b,\xi_b) \simeq (M,\xi)$ given by the projection $M\times\torus\rightarrow M$) corresponds to the isotopy class of the original open book $(B,\varphi)$ on $(M,\xi)$.
\end{cor}

\begin{proof}[Proof (\Cref{RmkCorBourgContStr})]
	Given $b\in\Sigma$, consider an ordered basis $(u,v)$ of $T_b\Sigma$ and define $\psi_\eta(b)$ as the class of $\Psi_\eta(u,v)$ under the relation $\sim$.
	Here, $\Psi_\eta$ is the map given by \Cref{PropIntroIsotopyClassOBD}. We then need to show that this is well defined.
	\\
	Suppose $(u',v')$ is another ordered basis of $T_b\Sigma$.
	We want to show that $\Psi_\eta(u,v)\sim\Psi_\eta(u',v')$. 
	Choose a curve $\gamma\co[0,1]\rightarrow F\Sigma$ with image contained in the fiber $\pr^{-1}(b)$ of $\pr\co F\Sigma\rightarrow \Sigma$ and such that $\gamma(0)=(u,v)$ and $\gamma(1)=(u',v')$.
	Then, according to point \ref{Item2CorBourgContStr} of \Cref{CorBourgContStr}, $\Psi_\eta\circ \gamma$ gives a smooth $[0,1]-$family of open books on $\gamma^*\pr^* V$ adjusted to $\gamma^*\pr^*\eta$.
	Now, $\gamma^{*}\pr^*V = [0,1]\times M_b$ and $\gamma^*\pr^*\eta = T\left(\left[0,1\right]\right) \oplus \xi_b$, so that we actually have, via the natural projection $[0,1]\times M_b\rightarrow M_b$, a smooth family of open books $(K_t,\varphi_t)_{t\in[0,1]}$ on $M_b$ supporting $\xi_b$. 
	Because a smooth path of open book decompositions comes from an isotopy as described in the statement, this actually means that $(K_0,\varphi_0)$ is isotopic to $(K_1,\varphi_1)$.
	In other words, $\Psi_\eta(u,v)\sim\Psi_\eta(u',v')$ as wanted.
	
	The last statement about the construction by Bourgeois follows directly from the definition of $\Psi_\eta$ and from point \ref{Item3PropGiroux1} of \Cref{PropGiroux1}. 
	Indeed, let $\eta=\ker(\beta+\phione d\theta_2 - \phitwo d\theta_2)$ a the Bourgeois contact structure  on the flat contact bundle $(\pi\co M\times\torus\rightarrow\torus,\xi\oplus T\torus)$ given by \Cref{ThmBourgeoisBis} starting from open book $(B,\varphi)$ of $M$ adapted to $\xi$. 
	As already observed in the beginning of \Cref{SubSubSecBourgContStrRev}, we can compute that $A_{\partial_{\theta_1}}$ and $A_{\partial_{\theta_2}}$ are respectively the contact vector fields on $(M,\xi)$ of contact hamiltonians $-\phione$ and $\phitwo$ (via $\beta$), with $\phi=(\phione,\phitwo)$ defining $(B,\varphi)$. 
	Then, we can see that point \ref{Item3PropGiroux1} of \Cref{PropGiroux1} with $X\coeq A_{\partial_{\theta_1}}$ and $Y\coeq A_{\partial_{\theta_2}}$ gives exactly the open book $(B,\varphi')$, where $\varphi'$ is obtained from $\varphi$ by composition with the antipodal map $\cercle\rightarrow\cercle$.\\
	In other words, for all $b\in \torus$, if $(\partial_{\theta_{1}},\partial_{\theta_{2}})$ is the oriented base of $T_b\torus$ coming from the choice of coordinates $(\theta_1,\theta_2)\in\torus$ as in the statement of \Cref{ThmBourgeoisBis}, then $\Psi_\eta(\partial_{\theta_{1}},\partial_{\theta_{2}})=(B,\varphi')$. In particular, $\psi_\eta(b)$ is the isotopy class of $(B,\varphi')$, which coincides with that of $(B,\varphi)$. 
\end{proof}

We now derive \Cref{CorBourgContStr} as a consequence of \Cref{PropGiroux1}: 
\begin{proof}[Proof (\Cref{CorBourgContStr})]
	Let's start by defining $\Psi_\eta$.
	Let $A$ be the potential for $\eta$ relative to the flat contact connection $\connH_0$ of $\eta_0$.
	Then, for each $b\in\Sigma$ and each positive basis $(u,v)$ of $T_b \Sigma$, $A_u$ and $ A_v$ are two vector fields on $M_b$ which are contact for $\xi_b$. 
	Moreover, according to \Cref{PropWhenContFibBundIsContStrUsingPotentials} (and by definition of Bourgeois contact structure), $[A_u,A_v]$ is negatively transverse to $\xi_b$.
	Then, \Cref{PropGiroux1} gives an open book decomposition $OBD_{(u,v)}$ for $M_b$ supporting $\xi_b$.
	Because $OBD_{(u,v)}$ is also a fiber adapted open book, we can define $\Psi_\eta(u,v)\coeq OBD_{(u,v)}$. 
	In particular, it is clear that point \ref{Item1CorBourgContStr} of \Cref{CorBourgContStr} is satisfied.
	
	Let's now prove point \ref{Item2CorBourgContStr}. 
	Consider $b\in\Sigma$ and a basis $(u,v)$ of $T_b\Sigma$. Let  $\alpha$ be a $1$-form defining $\eta$ on $V$ and denote $\alpha_b$ its restriction to the fiber $M_b$ of $V\rightarrow \Sigma$.
	From the explicit proof of point \ref{Item3PropGiroux1} of \Cref{PropGiroux1}, we can see that $\Psi_\eta(u,v)$ is the open book defined by the smooth function $$\phi_{(u,v)}\coeq\left(\alpha_b\left(A_u\right),-\alpha_b\left(A_v\right)\right) \co M_b \rightarrow \real^2 \text{ .}$$ %. \\
	By definition of pullback smooth bundle and pullback contact bundle, $\Psi_{\eta}\circ \gamma$ describes then the smooth $[0,1]$-family of open books on $\gamma^*\pr^*V$ adjusted to $\gamma^*\pr^*\eta$ which is given by the conformal class of 
	$\Psi_\gamma\co \gamma^*\pr^*V\rightarrow\real^2$
	defined, for all $(t,p)\in \gamma^*\pr^*V=\{(t,p)\in [0,1]\times V \vert \pr\circ\gamma(t)=\pi(p)\}$, by 
	\[\Psi_\gamma(t,p)\coeq \left((\mu^*\alpha)_{(t,p)}\left(A_{\gamma_1(t)}(p)\right),(\mu^*\alpha)_{(t,p)}\left(A_{\gamma_2(t)}(p)\right)\right)\text{ ,}
	\]
	where, for each $t\in[0,1]$, $\gamma_1(t)$ and $\gamma_2(t)$ are the two vectors of the (ordered) basis $\gamma(t)\in F\Sigma$ and where $\mu\co\gamma^*\pr^*V \rightarrow V$ is the restriction of the projection $\pr_V\co [0,1]\times V \rightarrow V$ to $\gamma^*\pr^* V$.
	Notice that $A_{\gamma_1(t)}(p)$ and $A_{\gamma_2(t)}(p)$ are well defined because $(t,p)\in \gamma^*\pr^* V$.
	This concludes the proof of point \ref{Item2CorBourgContStr}. % of \Cref{CorBourgContStr}.
\end{proof}

Finally, we point out a somehow peculiar property: on the trivial flat contact bundle $(\pi\co M\times\torus\rightarrow\torus,\xi\oplus T\torus)$ there is a natural way to associate a strong Bourgeois contact structure to each Bourgeois contact structure, in such a way that it gives a left inverse to the natural inclusion 
\begin{equation*}
i \co \left\{
\begin{array}{c}
\text{strong Bourgeois}\\
\text{contact structures on}\\
(\pi, \xi\oplus T\torus)
\end{array}
\right\}\hookrightarrow
\left\{
\begin{array}{c}
\text{Bourgeois}\\
\text{contact structures on} \\
(\pi, \xi\oplus T\torus)
\end{array}
\right\}\text{ .}
\end{equation*}
%on the flat contact bundle $(\pi\co M\times\torus\rightarrow\torus,\xi\oplus T\torus)$, where $\xi$ is contact on $M$, Bourgeois contact structures can also be studied from the point of view of the natural $\torus-$action on the principal bundle, in the spirit of \Cref{SecLutzBourgeois}.
Let's give a precise statement.
\\
The potential $A$ of a contact bundle $\eta$, with respect to the natural flat connection $\{0\}\oplus T\torus \subset T\left(M\times\torus\right)$ on $\pi \co M\times\torus\rightarrow\torus$, can actually be seen as a $1$-form defined on $\torus$ and with values in the vector space of contact vector fields of $(M,\xi)$, thanks to the canonical identification of each fiber of $\pi$ with $M$.
Moreover, an explicit computation gives that $A$ is $\overline{\nabla}$-parallel (with respect to the natural flat $\nabla^{\torus}$ on $\torus$) if and only if $\eta$ is invariant under the natural $\torus$-action on the total space $M\times\torus$.
Using this equivalence, we get:
\begin{prop}
	\label{LemmaBourgContStrInv}
	Let $\eta$ be a Bourgeois contact structure on the flat contact bundle $(\pi\co M\times\torus\rightarrow\torus,\xi\oplus T\torus)$ and denote by $A$ its potential.
	The average $\Aline$ of $A$ via the natural $\torus$-action is the potential of a strong Bourgeois contact structure with respect to the natural flat $\nabla^{\torus}$ on $\torus$. %i.e. a $\torus$-invariant Bourgeois contact structure $\etaline$ on $(\pi\co M\times\torus\rightarrow\torus,\xi\oplus T\torus)$.
	
	In particular, taking the average of the potential gives a well defined map
	$$
	F \co 
	\left\{
	\begin{array}{c}
	\text{Bourgeois}\\
	\text{contact structures on}\\
	(\pi, \xi\oplus T\torus)
	\end{array}
	\right\} 
	\rightarrow
	\left\{
	\begin{array}{c}
	\text{strong Bourgeois}\\
	\text{contact structures on}\\
	(\pi, \xi\oplus T\torus)
	\end{array}
	\right\} 
	\text{ ,}
	$$
	which satisfies $F\circ i = \Id$.
\end{prop} 

In other words, the space of Bourgeois contact structures on $(\pi, \xi\oplus T\torus)$ retracts to its subspace of strong Bourgeois contact structures.
It is not clear to the author whether this is actually a deformation retract or not.

\begin{proof}[Proof (\Cref{LemmaBourgContStrInv})]
	As remarked above, because $\Aline$ is $\torus$-invariant, it is also $\overline{\nabla}$-parallel, hence satisfies $\dnabla\Aline=0$.
	By \Cref{PropWhenContFibBundIsContStrUsingPotentials}, what we need to show is then that $[\Aline,\Aline]$ is with values in the \emph{negative} contact vector fields for $(M,\xi)$.
	Let's start by analyzing this condition more explicitly.
	
	Write $A=  X\otimes dx + Y \otimes dy$, with $(x,y)$ coordinates on $\torus=\left(\sfrac{\real}{2\pi\integ}\right)^2$ and $X,Y$ a $\torus$-family of vector fields on $M$ parametrized smoothly by $(x,y)$. 
	Here, for all $(x,y)\in\torus$, $[X,Y]$ is everywhere negatively transverse to $\xi$. 
	Because $\Aline=\Xline\otimes dx +\Yline\otimes dy$, it is then enough to show that their averages $\Xline,\Yline$ are such that $[\Xline,\Yline]$ is also everywhere negatively transverse to $\xi$.
	\\
	We point out that, if $Z,W$ are $\torus$-parametric vector fields on $M$, it is not true in general that the $\torus$-average of $[Z,W]$ is equal to the Lie bracket of the averages of $Z$ and $W$. 
	This being said, what we are going to show here is that this is actually true for $X,Y$, because of the additional condition $\dnabla A=0$.
	
	Notice that $X,Y$ can be seen as smooth functions from $\torus$ to the space of vector field on $M$, which has a natural structure of vector space over $\real$.
	As such, they both admit a complex Fourier series expansion
	\begin{equation}
	\label{EqFourierXY}
	X = \sum_{m,n\in\integ} e^{i\left(mx+ny\right)}X_{m,n} \; \; \text{ and }\; \; Y = \sum_{h,k\in\integ} e^{i\left(hx+ky\right)}Y_{h,k} \text{ ,}
	\end{equation}
	where, for all $m,n,h,k\in\integ$, $X_{m,n},Y_{h,k}$ are complex vector fields on $M$, i.e. sections of the complexified tangent bundle $TM\otimes_{\real}\compl\rightarrow M$. 
	Because $X,Y$ are actually real, we have the following condition on the coefficients:
	\begin{equation}
	\label{EqFourierXYCond1}
	\overline{X_{m,n}} = X_{-m,-n} \; \; \text{ and }\;\; \overline{Y_{h,k}} = Y_{-h,-k} \text{ , for all $m,n,h,k\in \integ$,}
	\end{equation}
	where $\overline{X_{m,n}}$, $\overline{Y_{h,k}}$ denote the complex conjugates of $X_{m,n}$ and $Y_{h,k}$ respectively.
	
	The condition $\dnabla A=0$ also gives some information on the Fourier coefficients.
	More precisely,
	\begin{equation}
	\label{ClaimDoAVectFields}
	\dnabla A=0 \text{ if and only if } \frac{\partial}{\partial {x}} Y - \frac{\partial}{\partial {y}} X = 0 \text{ .}
	\end{equation}
	%\begin{proof}%[Proof (\Cref{ClaimDoAVectFields})]
	Indeed, we can explicitly compute 
	\begin{align*}
	\dnabla A\left(\partial_{x},\partial_{y}\right)  & = 
	\nabla_{\partial_x}(A_{\partial_y}) - \nabla_{\partial_y}(A_{\partial_x}) - A_{\left[\partial_{x},\partial_{y}\right]}\\
	%\left[\widehat{\partial}_{x},A_{\partial_{y}}\right] - \left[\widehat{\partial}_{y},A_{\partial_{x}}\right] - A_{\left[\partial_{x},\partial_{y}\right]} \\
	& \overset{(i)}{=}  \left[\widehat{\partial}_{x},Y\right] - \left[\widehat{\partial}_{y},X\right] \\
	& \overset{(ii)}{=}  \frac{\partial}{\partial {x}} Y - \frac{\partial}{\partial {y}} X \text{ ,}
	\end{align*}
	where $(i)$ comes from the fact that $\partial_{x}$ and $\partial_{y}$ commute (and from the definition of $\nabla$), and $(ii)$ follows from the expression in coordinates of the Lie bracket.
	%\end{proof}
	
	A straightforward computation shows that the right condition in \Cref{ClaimDoAVectFields} is equivalent to: % the following condition:
	\begin{equation}
	\label{EqDoAZeroInTermOfCoeff}
	mY_{m,n} = n X_{m,n} \; \text{ for all } m,n\in\integ \text{ .}
	\end{equation}
	Notice now that the averages of $X$ and $Y$ are, respectively, $X_{0,0}$ and $Y_{0,0}$, which are in particular real vector fields on $M$.
	To avoid confusion with the conjugation, we will hence drop the notation $\Xline$ and $\Yline$ for the averages and just denote them by $X_{0,0}$ and $Y_{0,0}$ instead.
	\\
	Let $[.\, , .]_{\compl}$ be the Lie bracket induced on the complex vector space of the sections of $TM\otimes\compl\rightarrow M$ by $[.\, , .]$ on the space of vector fields on $M$.
	We then compute:
	\begin{align*}
	[X,Y] & = \left[\sum_{m,n\in\integ} e^{i\left(mx+ny\right)}X_{m,n} \, ,\; \sum_{h,k\in\integ} e^{i\left(hx+ky\right)}Y_{h,k} \right]_\compl \\
	& \overset{(a)}{=} \sum_{m,n\in\integ} \sum_{h,k\in \integ} e^{i\left[\left(m+h\right)x + \left(n+k\right)y\right]} \left[X_{m,n},Y_{h,k}\right]_\compl\\
	& \overset{(b)}{=} \sum_{r,s\in\integ} e^{i\left(rx+sy\right)} \left(\sum_{m,n\in\integ}\left[X_{m,n},Y_{r-m,s-n}\right]_\compl\right) \text{ ,}
	\end{align*}
	where the equality $(a)$ comes from the fact that the Lie bracket is $\compl-$bilinear and is taken on each fiber $M\times\{pt\}$ of $M\times\torus\rightarrow\torus$ (where the exponentials are constant), and the equality $(b)$ comes from replacing $r=m+h$ and $s=n+k$.
	
	The above computation shows that $[X,Y]$ has Fourier coefficients 
	\begin{equation}
	\label{EqFourierBracketXY}
	[X,Y]_{r,s} = \sum_{m,n\in\integ}\left[X_{m,n},Y_{r-m,s-n}\right]_\compl
	\end{equation}
	for $r,s\in\integ$.
	In particular, its average is given by 
	\begin{align*}
	[X,Y]_{0,0} & = \sum_{m,n\in\integ}\left[X_{m,n},Y_{-m,-n}\right]_\compl\\
	& \overset{(a)}{=} [X_{0,0},Y_{0,0}] + \sum_{m,n\in\integ\setminus\{0\}}\left[X_{m,n},Y_{-m,-n}\right]_\compl\\
	& \overset{(b)}{=}[X_{0,0},Y_{0,0}] + \sum_{m,n\in\integ\setminus\{0\}}\frac{m}{n}\left[Y_{m,n},\overline{Y_{m,n}}\right]_\compl\\
	& \overset{(c)}{=}[X_{0,0},Y_{0,0}] - 2 i \sum_{m,n\in\integ\setminus\{0\}}\frac{m}{n}\left[\Re Y_{m,n},\Im Y_{m,n}\right] \\
	& \overset{(d)}{=} [X_{0,0},Y_{0,0}] \text{ ,}
	\end{align*}
	where $\Re Y_{m,n}$ and $\Im Y_{m,n}$ denote respectively the real and imaginary part of $Y_{m,n}$.
	Moreover, $(a)$ comes from the fact that $X_{m,n}$ is zero if $m=0,n\neq0$ and $Y_{-m,-n}$ is zero if $n=0,m\neq0$ by \Cref{EqDoAZeroInTermOfCoeff},  $(b)$ comes from \Cref{EqFourierXYCond1,EqDoAZeroInTermOfCoeff}, $(c)$ comes from the $\compl-$bilinearity of $[.\,,.]_{\compl}$ and the anti-symmetry of $[.\, , .]$ and, finally, $(d)$ comes from the fact that $[X,Y]_{0,0}$ is a (real) tangent vector field, because average of $[X,Y]$, hence has zero imaginary part.
	
	Because $[X,Y]$ is negatively transverse to $\xi$ everywhere on $M$ for all $(x,y)\in\torus$, its average $[X,Y]_{0,0} = [X_{0,0},Y_{0,0}]$ is also negatively transverse to $\xi$ everywhere on $M$. 
	%Now, notice that the average $\Aline$ of $A$ is exactly $\Aline= -X_{0,0}\otimes dx - Y_{0,0}\otimes dy$. Then, the fact that $[X_{0,0},Y_{0,0}]$ is everywhere negatively transverse to $\xi$ means that $[\Aline,\Aline]$ is a $1-$form on $\torus$ with values in the negative contact vector fields for $(M,\xi)$. 
	This concludes the proof of \Cref{LemmaBourgContStrInv}.
\end{proof}

\begin{remark}
	In analogy with the case of Bourgeois contact structures, we could have also considered,  on a flat contact fiber bundle $(\pi\co V\rightarrow\Sigma,\eta_0)$, the class of contact structures $\eta$ on $V$ given by a potential $A$ with $[A,A]=0$. 
	
	For such an $\eta$, \Cref{PropWhenContFibBundIsContStrUsingPotentials} tells us that $\dnabla A$ is with values in the negative contact vector fields of the fibers. 
	Such a condition, though, is not compatible with the fact that the surface $\Sigma$ is closed.\\
	Indeed, by explicit computations (analogous to the ones in the proof of \Cref{ClaimBourgContStr} in the following) it can be proven that this condition on $\dnabla A$ implies the existence of an exact volume form on $\Sigma$. Now, the latter can't exist if $\Sigma$ is closed, according to Stoke's theorem.
	
	Moreover, even if we allow $\Sigma$ to have boundary, we do not recover all the informations on the fiber that we have with a Bourgeois contact structure. 
	More precisely, we can't recover in general an (isotopy class of) open book decomposition supporting the contact structure on the fiber.\\
	For instance, consider on the flat contact bundle $(M\times\Sigma\rightarrow\Sigma,\xi_M\oplus T\Sigma)$ the contact fiber bundle structure $\eta = \ker\left(\alpha+\lambda\right)$, with $\xi_M=\ker\alpha$ and $d\lambda$ symplectic on $\Sigma$ (that hence has non$-$empty boundary).
	Then, an explicit computation shows that $A= - R_\alpha \otimes \lambda$, where $R_\alpha$ is the Reeb vector field of $\alpha$. In particular, $[A,A]=0$ and $\dnabla A = - R_\alpha \otimes d\lambda$, and we do not have any way to recover an (isotopy class of) open book decomposition on $M$ from $A$.
\end{remark}

%%%%%%%%%%%%%%%%%%%%%%%%%%%%%%%%%%%%%%%%%%%%%%%%%%%%%%%%%%%%%%%%%%%%%%%%%%%%%%%%%%%%

\subsection{Contact deformations and branched coverings}
\label{SubSubSecContactizationOpContCat}

We show in this section that the class of contact fiber bundles that are contact deformations of a flat contact fiber bundle is stable under the operation of contact branched coverings:
\begin{prop}
	\label{PropStabContactizationsBranchCov}	
	Let $(\pi\co V^{2n+1}\rightarrow\Sigma,\eta_0)$ be a flat contact fiber bundle and $p:\Sigmahat\rightarrow\Sigma$ a branched covering map that lifts to a branched covering map $\phat:\Vhat \rightarrow V$. Consider now the pull-back flat contact fiber bundle $(\pihat\co \Vhat\rightarrow\Sigma,\etazerohat)$ induced by $p$, i.e. $\etazerohat := \phat^{\,*}\eta_0$. If $\eta$ is a contact deformation of $\eta_0$, then there is a contact branched covering $\etahat$ of $\eta$ to $\Vhat$ that is a contact deformation of $\etazerohat$.
\end{prop}

\begin{proof}
	%This follows essentially from the explicit formula for the contact branched covering in the proof of \Cref{LemmaExistUniqContactizationInFamilies}.
	%
	%More precisely, 
	By definition of contact deformation, there is a smooth family of $1$-forms $\left(\alpha_t\right)_{t\in[0,1]}$ on $V$ interpolating between $\eta = \ker\alpha_1$ and $\eta_0 = \ker\alpha_0$, such that $\eta_t=\ker\alpha_t$ is contact for $t>0$ and the fibers of $\pi:V\rightarrow \Sigma$ have induced contact structures independent of $t$.
	For $t\in[0,1]$, define $f_t,h_t\co V \to \real$ by $\alpha_t\wedge d\alpha_t^{n-1}\vert_{\ker(d\pi)} = f_t \alpha_1 \wedge d\alpha_1^{n-1}\vert_{\ker(d\pi)}$ and $\alpha_t\wedge d\alpha_t^n = h_t \alpha_1 \wedge d\alpha_1^n$.
	Notice that $f_t>0$ everywhere for every $t\in[0,1]$, whereas $h_t>0$ everywhere for $t>0$ and $h_0=0$.
	Moreover, as $h_t$ is a smooth function of $t$ and $V$ is compact, the function $k(t)\coeq \min\{h_t(p)\,\vert\, p \in V\}$ is smooth in $t$.
	
	%	We claim that, up to deformation, we can assume that $\alpha_t=td\alpha_1$ for $t$ near $0$.
	%	Indeed, let $\delta>0$ be very small, and consider $\chi\co \real_{\geq0}\to \real$ be a cutoff function equal to $\delta$ near $0$ and equal to $0$ near $\delta$. 
	%	Then, $(\alpha_t')_{t\in[0,1]}$ given by $\alpha_t'\coeq (1-\chi(t))\alpha_t + \chi(t)t\alpha_1$ gives a well defined contact deformation, starting at $\ker(\alpha_1')=\eta$ and ending at $\ker(\alpha_0')=\eta_0$, moreover satisfying the desired condition. 
	%	With a little abuse of notation, we denote $\alpha_t'$ again by $\alpha_t$.
	%	Notice in particular that $h_t$ is now constant for $t$ near $0$.
	
	According to the proof of \Cref{LemmaExistUniqContactizationInFamilies}, we can choose $\etahat$ on $\Vhat$ to be the kernel of $\alphahat = \, \phat^{\,*} \alpha_1 + \epsilon g(r) r^2 d\theta$, with the same notations as in that proof, using the particular choice of closed form $\gamma = d\theta$ as connection on the trivial unit normal bundle of $M$ in $V$. 
	Recall that $\epsilon>0$ can be chosen arbitrarily small here.
	\\
	Define $\left(\alphahat_t\right)_{t\in[0,1]}$ by $\alphahat_t = \phat^{\,*} \alpha_t + tk(t)\epsilon g(r) r^2 d\theta$.
	In particular, $\ker(\alphahat_1) = \etahat$ and $\ker(\alphahat_0) = \etahat_0$.
	We then claim that $\alphahat_t$ is a contact deformation of $\etahat_0$ to $\etahat$.
	\\
	Now, $\alphahat_t$ gives on each fiber a contact structure independent of $t$, hence the only thing we need to show is that $\alphahat_t$ defines a contact structure for $t>0$.
	We can explicitly compute
	\begin{align*}
	\alphahat_t\wedge d\alphahat_{t}^{n} \, \, = & \, \, C^{n+1} \phat^{\,*} \left( \alpha_t\wedge d\alpha_{t}^{n} \right) \, + \\ 
	& + \, C^n \, \epsilon  \,t \, k(t)\left(r g'\left(r\right) + 2 g\left(r\right)\right) \phat^{\,*} \left( \alpha_t\wedge d\alpha_{t}^{n-1} \right) \wedge r dr \wedge d\theta \\
	& + \, C^n\, \epsilon \,  t\, k(t)\, g(r) \,r^2 d\theta \wedge \phat^{\,*} d \alpha_t^n \text{ .}
	\end{align*}
	Notice that $\phat^{\,*} \left( \alpha_t\wedge d\alpha_{t}^{n} \right)=h_t\phat^{\,*} \left( \alpha_1\wedge d\alpha_{1}^{n} \right)$. 
	Moreover,
	$$\phat^{\,*} \left( \alpha_t\wedge d\alpha_{t}^{n-1} \right) \wedge r dr \wedge d\theta \, = f_t \, \phat^{\,*} \left( \alpha_1\wedge d\alpha_{1}^{n-1} \right) \wedge r dr \wedge d\theta \text{ .}$$
	%because $\alpha_t$ and $\alpha_1$ induce the same contact form on each fiber. 
	In particular, $\phat^{\,*} \left( \alpha_t\wedge d\alpha_{t}^{n-1} \right) \wedge r dr \wedge d\theta$ is bounded below by a positive volume form independent of $t$.
	Using the fact that $t\,k(t)/h_t \to 0$ for $t\to 0$, an argument analogous to the one in the proof of \Cref{LemmaExistUniqContactizationInFamilies} allows to conclude that, if $\epsilon>0$ is small enough, $\alphahat_t\wedge d\alphahat_{t}^{n}>0$ for every $t>0$. 		
\end{proof}

\section{Virtually overtwisted contact structures in high dimensions}
\label{SecVirtOTHighDim}

In \Cref{SubSecBourgContStrWeakFill}, we prove \Cref{PropIntroBourgContStrWeakFill} from \Cref{SecIntro}, stating that a Bourgeois contact structure on a fiber bundle with total space $M\times\torus$ is weakly fillable provided that the same is true for the fiber $(M,\xi)$. 
Then, \Cref{SubSecExVirtOTMfld} contains the proof of \Cref{ThmExistenceVirtOT}, also stated in \Cref{SecIntro}, about the existence of virtually overtwisted manifolds in all odd dimensions.

%%%%%%%%%%%%%%%%%%%%%%%%%%%%%%%%%%%%%%%%%%%%%%%%%%%%%%%%%%%%%%%%%%%%%%%%%%%%%%%%%%%%

\subsection{Bourgeois contact structure and weak fillability}
\label{SubSecBourgContStrWeakFill}

Let $(M^{2n-1},\xi)$ be a contact manifold and consider the flat contact bundle $(\pi\co M\times\torus\rightarrow\torus,\eta_0=\xi\oplus T\torus)$, where $\pi$ is the projection on the $\torus$-factor.

\begin{prop}
	\label{Prop1}
	Let $\eta$ be a Bourgeois contact structure on $(\pi,\eta_0)$. 
	If $(M,\xi)$ is weakly filled by $(X,\omega)$, then $(\Mtor,\eta)$ is weakly filled by $(X\times\torus, \omega+\omegat)$, where $\omegat$ is an area form on $\torus$.
\end{prop}

Recall that the result is already known in the case of $\eta$ obtained by the Bourgeois construction \cite{Bou02}. The statement and the idea of the proof in that case already appeared in Massot--Niederkr\"uger--Wendl \cite[Example 1.1]{MNW13}, and an explicit proof can be found in Lisi--Marinkovi\'c--Niederkr\"uger \cite[Theorem A.a]{LisMarNie18}.
Notice also that, in a similar vein, \cite[Theorem A.b]{LisMarNie18} studies the stability of (subcritical) Stein fillability under the Bourgeois construction.
%Notice also that in \cite{LisMarNie18} the authors focus on the study of the properties that are preserved under the construction in \cite{Bou02}, analyzing in particular the cases of weak and (subcritical) Stein fillability.

The proof of \Cref{Prop1} uses the following fact about polynomials:
\begin{fact}
	\label{LemmaPoly2}
	Let $P_1,P_2\in\polytau$ of degree $n$, with $P_1(\tau)>0 \;\forall\,\tau\geq0$ and with $P_2$ having positive leading coefficient. Then $\exists \,\epsilon_0>0$ such that $\forall\,0<\epsilon<\epsilon_0$, $P_1+\epsilon^2 P_2>0$ on $\real_{\geq0}$.
\end{fact}

\begin{proof}[Proof (\Cref{Prop1})]
	We first choose a convenient contact form for $\eta$.\\
	Let $\beta$ is a form on $M$ defining $\xi$.
	We can then write $\eta=\ker\left(\alpha\right)$, where $\alpha\coeq \beta + \phione d\theta_1 - \phitwo d\theta_2$, with $\phione,\phitwo\co M\times\torus\rightarrow \real$ and $(\theta_1,\theta_2)$ coordinates on $\torus$.
	%Indeed, $\eta$ can be written as $\ker\left( \gamma + fd\theta_1 + gd\theta_2\right)$, with $f,g\co M\times\torus \rightarrow\real$ and $\gamma\in\Omega^{1}\left(M\times\torus\right)$ that is zero on $\{0\}\oplus T\torus \subset T\left(M\times\torus\right)$. Now, because $\eta$ induces $\xi$ on each fiber, there is a positive function $h\co M\times\torus\rightarrow\real$ such that $\gamma=h \beta$. Then, we can choose $\phi_1=\frac{f}{h}$ and $\phi_2=\frac{g}{h}$.
	%%the division by $h$ gives a contact form for $\eta$ of the announced form.
	\\	
	Recall from \Cref{SubSubSecBourgContStrRev} that if $A$ denotes the potential of $\eta$ with respect to $\eta_0$ then, for each $\epsilon>0$, the family of potentials $A_\epsilon\coeq \epsilon A$ defines a family $\eta_\epsilon$ of Bourgeois contact structures that are all isotopic among Bourgeois contact structures (hence among contact structures). 
	%Notice that $\eta_\epsilon$ is just the kernel of $\aeps=\beta+\epsilon\phi_1 d\theta_1-\epsilon\phi_2 d\theta_2$: indeed, $\epsilon\phione$ and $-\epsilon\phitwo$ are the contact hamiltonians associated, respectively, to $-\left(A_\epsilon\right)_{\partial_x}= -\epsilon A_{\partial_{x}}$ and $-\left(A_\epsilon\right)_{\partial_y}= - \epsilon A_{\partial_{y}}$.
	%In other words, up to isotopy, we can suppose $\eta= \ker(\aeps)$ for a certain $\epsilon>0$ that will be chosen very small in the following.  
	Thus, up to isotopy, we suppose $\eta= \ker(\aeps)$, where $\aeps=\beta+\epsilon\phi_1 d\theta_1-\epsilon\phi_2 d\theta_2$, for a certain $\epsilon>0$ that will be chosen very small in the following.  
	
	The weak fillability condition for $M$ implies that $$\beta\wedge(\omegaM+\tau d\beta)^{n-1}>0 \text{ on }M \text{, for all }\tau\geq0 \text{ ,}$$ where $\omegaM$ denotes the restriction of $\omega$ to $M=\partial X$. 
	We want to verify that, for $\epsilon>0$ small enough, we also have 
	$$\aeps\wedge(\omegaM + \omegat +\tau d\aeps)^n>0 \text{ on }M\times\torus \text{, for all } \tau\geq0\text{ .}$$ 
	
	\begin{claim}
		\label{ClaimBourgContStr}
		Let $\Omega$ be an arbitrary volume form on $\Mtor$. We then have
		\begin{multline*}
		\aeps\wedge(\omegaM+\omegat+\tau d\aeps)^{n}\;\,=\,  \; n\beta\wedge\eqpiece\wedge\omegat \;  \\ 
		+ \; \epsilon^2\tau^n\alpha_1\wedge d\alpha_{1}^{n}\; +\; \epsilon^2 h \,\Omega \; \text{ ,}
		\end{multline*}
		where $h$ is independent of $\epsilon$ and polynomial in $\tau$, with $\deg_\tau(h)\leq n-1$. 
	\end{claim}  
	To improve readability, the proof of this claim is postponed. 
	
	Denote now $f$ and $g$ the functions defined by $f\,\Omega = n\beta\wedge\eqpiece\wedge\omegat$, $g\,\Omega = \tau^n\alpha_1\wedge d\alpha_{1}^{n}$.
	Then we need to show that $f+\epsilon^2(g+h)>0$ on $\Mtor$.
	\\
	Notice that for each $p\in\Mtor$, $f(p)$, $g(p)$ and $h(p)$ are \emph{polynomials} in $\tau$, by explicit computation in the case of $f$ and $g$, and by \Cref{ClaimBourgContStr} in the case of $h$. 
	Moreover, we have the following properties: for each $p\in M\times\torus$,
	\begin{enumerate}[label=(\alph*)]
		\item $f(p)>0$, because $(X,\omega)$ weakly fills $(M,\xi)$;
		\item $g(p)>0$, because $\aone$ is a contact form for $\eta$;
		\item $h(p)$ has degree in $\tau$ strictly less than $g(p)$, by \Cref{ClaimBourgContStr}.
	\end{enumerate}
	
	For each $p\in M \times\torus$, define $P_1=f(p)$ and $P_2=g(p)+h(p)$.
	\Cref{LemmaPoly2} then gives an $\epsilon_p>0$ such that $f(p) + \epsilon_p (g+h)(p)>0$.
	Thus, by compactness of $M\times\torus$, there is $\epsilon>0$ such that $f+\epsilon (g+h)>0$, as desired.
\end{proof}

%We now prove \Cref{ClaimBourgContStr} used above:
\begin{proof}[Proof (\Cref{ClaimBourgContStr})]
	We can compute 
	\begin{equation}
	\label{EqDaeps}
	d\aeps=d\beta+\epsilon \ddun-\epsilon \dddeux - \left(\frac{\partial \phione}{\partial {\theta_2}}+ \frac{\partial \phitwo}{\partial {\theta_1}}\right) d{\theta_1}\wedge d{\theta_2}\text{ .}
	\end{equation}
	
	\noindent
	Moreover, one has the following:
	\begin{equation}
	%\begin{claim}
	\label{ClaimDoAZero}
	\dnabla A = 0 \text{ if and only if } \frac{\partial \phione}{\partial {\theta_2}}+ \frac{\partial \phitwo}{\partial {\theta_1}}=0 \text{ .}
	%\end{claim}
	\end{equation}
	
	%\begin{proof}[Proof (\Cref{ClaimDoAZero})]
	\noindent
	Indeed, we have $A=X\otimes d\theta_1  Y \otimes d\theta_2$, with $X,Y$ the contact vector fields on $(M,\xi)$ with contact hamiltonians $-\phione,\phitwo$ w.r.t. $\beta$.
	By \Cref{ClaimDoAVectFields}, $\dnabla A = 0$ if and only if $ \frac{\partial}{\partial {\theta_1}} Y - \frac{\partial}{\partial {\theta_2}} X=0$.
	Now, because the latter is a contact vector field on each fiber $(M,\xi)$, it is zero if and only if its contact hamiltonian w.r.t.\ $\beta$ is zero, i.e.\ if and only if 
	\begin{equation*}
	0=  \frac{\partial}{\partial {\theta_1}} \beta\left(Y\right) - \frac{\partial}{\partial {\theta_2}} \beta\left(X\right) =  \frac{\partial \phitwo}{\partial {\theta_1}} + \frac{\partial \phione}{\partial {\theta_2}} \text{ ,} %\qedhere
	\end{equation*}
	thus giving the equivalence in \Cref{ClaimDoAZero}.
	%\end{proof}

	Because $\eta$ is a Bourgeois contact structure, \Cref{EqDaeps} then becomes 
	\begin{equation*}
	d\aeps=d\beta+\epsilon \ddun-\epsilon \dddeux \text{ .}
	\end{equation*}
	For dimensional reasons, we thus get 
	\begin{align*}
	\left(\omegam +  \omegat \right. & \left.  +\tau d\aeps\right)^n\, \,=\,\, \\ = & \, \, n\left(\omegam+\tau d\beta\right)^{n-1}\wedge\left(\omegat+\tau\epsilon \ddun - \tau\epsilon \dddeux\right)+ \\
	&+\tau^2 \epsilon^2 n(n-1)\left(\omegam+\tau d\beta\right)^{n-2}\wedge d\phi_1\wedge d\phi_2\wedge \dxdy \text{ ,}
	\end{align*}
	and
	\begin{multline}
	\label{EqComputation1}
	\aeps\wedge (\omegam + \omegat +\tau d\aeps)^n= \\ 
	\shoveleft{=\,\, n(\beta+\epsilon\dun-\epsilon\ddeux)\wedge\eqpiece\wedge} \\	
	\shoveright{\wedge(\omegat+\tau\epsilon\ddun-\tau\epsilon\dddeux)+} \\
	\shoveleft{\;\,\;+\,\tau^2 \epsilon^2 n(n-1)\left(\beta+\epsilon\dun-\epsilon\ddeux\right)\wedge\left(\omegam+\tau d\beta\right)^{n-2} \,\wedge } \\
	\shoveright{ \wedge d\phi_1\wedge d\phi_2\wedge \dxdy = }\\
	%\shoveleft{=\,\, n\beta\wedge\eqpiece\wedge\omegat \;\,+\;\, n\tau\epsilon^2(\phi_1 d\phi_2 - \phi_2 d \phi_1)\;\wedge} \\ 
	%\shoveright{\wedge\eqpiece\wedge\dxdy \,+\,}\\
	\shoveleft{=\,\, n\beta\wedge\eqpiece\wedge\omegat}\\
	\shoveleft{\;\;\;\,+\, n\tau\epsilon^2(\phi_1 d\phi_2 - \phi_2 d \phi_1)\;\wedge\eqpiece\wedge\dxdy}\\
	\shoveleft{\;+\, \tau^2 \epsilon^2 n(n-1)\beta\wedge\left(\omegam+\tau d\beta\right)^{n-2}\wedge d\phi_1\wedge d\phi_2 \wedge \dxdy \text{ .}}
	\end{multline}
	A similar explicit computation (using again \Cref{ClaimDoAZero}) shows that
	\begin{multline*}  
	\aone\wedge  d\aone^n 
	\,\,=\,\,  n\left(\phione d\phitwo - \phitwo d\phione\right)\wedge d\beta^{n-1}\wedge\dxdy\,+ \\ +\, n(n-1)\beta\wedge d\beta^{n-2}\wedge d\phione \wedge d\phitwo\wedge \dxdy \text{ ,}
	\end{multline*}
	so that the second and third term in the right hand side of the last equality in \Cref{EqComputation1} give $\epsilon^2\tau^n\aone\wedge d\aone^{n-1}+\epsilon^2 h \, \Omega$, where $h$ is as in the statement.
	This conclude the proof of \Cref{ClaimBourgContStr}.
\end{proof}   

Even if we will not use it in the following, we remark that the local nature of the condition $\dnabla A=0$ and of all the computations in the above proof actually gives the following more general result:
\begin{prop}
	\label{Prop1General}
	Let $(M^{2n-1},\xi)$ be a contact manifold weakly filled by $(X^{2n},\omega)$. Suppose that a representation $\rhotilde$ of $\pi_1(\surfg)$ in the group of symplectomorphisms of $(X,\omega)$ gives, by restriction to the boundary, a representation $\rho$ of $\pi_1(\surfg)$ in the group of contactomorphisms of $(M,\xi)$. Let also $\eta$ be a Bourgeois contact structure on the flat contact bundle $(\pi\co M\times_{\rho}\surfg\rightarrow\surfg,\eta_0)$ (as constructed in \Cref{SubSubSecFlatContFibBund}). Then, there is a symplectic form $\Omega$ on $X\times_{\rhotilde}\surfg$ that weakly fills $\eta$ on $M\times_{\rho}\surfg$.\\
	More precisely, if $\real^2\rightarrow\surfg$ denotes the universal covering map, $\Omega$ can be chosen to be the symplectic form on $X\times_{\rhotilde}\surfg$ induced by $\omega + \omega_{\real^{2}}^{g}$ on $X\times\real^2$ and where $\omega_{\real^{2}}^{g}$ is a symplectic form on $\real^2$ invariant by the action of $\pi_1(\surfg)$ on $\real^2$ by deck transformations.
\end{prop}

Let's now come back to the results we need in order to exhibit examples of virtually overtwisted manifolds in all odd dimensions. \Cref{Prop1} and (the proof of) \Cref{PropBranchCovPresWeakFill} have the following immediate corollary:
\begin{prop}
	\label{CorWeakFillBranchCovBourgContStr}
	Consider a branched covering $\Sigma_g\rightarrow\torus$, where $\Sigma_g$ is the closed genus $g\geq2$ surface, and the naturally induced branched covering $\Msurfg\rightarrow\Mtor$. Let $\eta_g$ on $\Msurfg$ be a contact branched covering of a Bourgeois contact structure $\eta$ on the the contact bundle $(\pi\co\Mtor\rightarrow \torus,\xi\oplus T\torus)$, where $\xi$ is a contact structure on the fiber $M$.
	Then, if $(M,\xi)$ admits a weak filling $(X,\omega)$, there is a symplectic form $\Omega$ on $\Xsurfg$ weakly dominating $\eta_g$ on $\Msurfg=\partial \Xsurfg$. More precisely, $\Omega$ can be chosen to be of the form $\omega+\omegag$, for a certain area form $\omegag$ on $\surfg$.
\end{prop}

%We point out that the explicit proof of \Cref{PropBranchCovPresWeakFill} actually shows that, up to isotopy, $\Omega$ can be chosen to be of the form $\omega+\omegag$, for a certain area form $\omegag$ on $\surfg$.

%%%%%%%%%%%%%%%%%%%%%%%%%%%%%%%%%%%%%%%%%%%%%%%%%%%%%%%%%%%%%%%%%%%%%%%%%%%%%%%%%%%%

\subsection{High dimensional virtually overtwisted manifolds}
\label{SubSecExVirtOTMfld}

Let $\pi\co\surfg\rightarrow\torus$ be a branched covering map, branched along two points, and consider $(\Id,\pi)\co\Msurfg\rightarrow M\times \torus$.
Notice that $g$ is the branching index along each of the two connected components of the upstairs branching locus of $(\Id,\pi)$.

\begin{prop}
	\label{PropVirtOT}
	Let $\eta$ be a Bourgeois contact structure on the flat contact bundle $(M\times\torus,\torus,M,\eta_0)$ and consider a contact branched covering $\eta_g$ of $\eta$ with respect to $(\Id,\pi)\co\Msurfg\rightarrow M\times \torus$.
	If $(M,\xi)$ is weakly fillable and virtually overtwisted, 
	then, for $g\geq2$ big enough, $\left(M\times\surfg,\eta_g\right)$ is weakly fillable and virtually overtwisted.
\end{prop}

Starting for instance from the case of a holomorphically fillable virtually overtwisted contact structure on lens spaces, that exist by Gompf \cite[Proposition 5.1]{Gom98} (see also Giroux \cite[Theorem 1.1]{Gir00}), and using the construction in \cite{Bou02}, a proof by induction on the dimension $2n-1$ of $M$ gives then the following: 

\begin{cor}
	\label{CorVirtOT}
	Virtually overtwisted manifolds exist in all odd dimensions $\geq3$.
\end{cor}

For the proof of \Cref{PropVirtOT} we will need the following result, which is essentially just a rephrasing of the discussion in Niederkr\"uger--Presas \cite[Page $724$]{NiePre10} for the local situation near the branching locus:
\begin{lemma}%[\cite{NiePre10}]
	\label{LemmaSizeNeigh}
	For $k\in\nat_{>1}$, let $\pi_k\co\Vhat_k\rightarrow V^{2n+1}$ be a branched covering map of branching index $k$. 
	Suppose that all $\pi_k$'s have same downstairs branching locus $M$ and that the upstairs branching locus $\Mhat_k$ of $\pi_k$ is connected (in particular, $\pi_k\vert_{\Mhat_{k}}$ induces a diffeomorphism between $\Mhat_k$ and $M$).
	Suppose also that there is a tubular neighborhood $\neigh\coeq M\times\disk^2$ (where $\disk^2$ is the $2-$disk centered at $0$ and of radius $1$) of the downstairs branching locus $M$ over which \emph{all} the $\pi_k$'s are trivialized at the same time, i.e. such that $\pi_k\co M\times\disk^2\rightarrow M \times \disk^2$ is just $(p,z)\mapsto(p,z^k)$ for all $k$. 
	Finally, let $\eta$ be a contact structure on $V$ inducing a contact structure $\xi$ on $M$ and $\etahat_k$ on $\Vhat_k$ be a contact branched covering of $\eta$. 
	
	Then, there is $\epsilon>0$ such that, for all $k\geq2$, the upstairs branching locus $(\Mhat_k,\xihat_k=\ker\pi_k^*\alpha)\overset{\pi_k}{\simeq}(M,\xi=\ker\alpha)$ has a neighborhood of the form $\left(M\times\disk_{\sqrt{k}\epsilon}^2,\ker\left(\alpha+r^2d\varphi\right)\right)$ inside $(\Vhat_k,\etahat_k)$ (here, by $\disk^2_{r}$ we denote the open disk centered in $0$ and of radius $r$ inside $\real^2$).
\end{lemma}

\begin{proof}[Proof (\Cref{PropVirtOT})]
	\Cref{CorWeakFillBranchCovBourgContStr} tells us that $\left(M\times\surfg,\eta_g\right)$ is weakly fillable for all $g\geq2$. We then have to show that, for $g$ sufficiently big, this contact manifold admits a finite cover which is overtwisted.
	
	By hypothesis, there is a finite cover $p\co\Mline\rightarrow M$ such that $(\Mline,\xiline\coeq p^*\xi)$ is overtwisted.
	Consider then the following commutative diagram of smooth maps:
	\begin{center}
		\begin{tikzcd}
			\Mline\times\surfg \ar[r, "{(p, \Id)}"] \ar[d, "{(\Id,\pi)}", swap] 
			& M\times\surfg \ar[d, "{(\Id,\pi)}"] \\
			\Mline\times\torus \ar[r, "{(p,\Id)}"] & M\times\torus
		\end{tikzcd}
	\end{center}
	Consider also $\etaline\coeq (p,\Id)^*\eta$ on $\Mline\times\torus$ and $\zeta_g\coeq (p,\Id)^*\eta_g$ on $\Mline\times\surfg$.
	Notice that the restriction of $\zeta_g$ to the upstairs branching locus of $(\Id,\pi)\co \Mline\times\surfg\rightarrow \Mline \times\torus$ is exactly $\xiline$.
	
	We now claim that $(\Mline\times\surfg,\zeta_g)$ is a branched contact covering of $(\Mline\times\torus,\etaline)$.
	Indeed, we can see that $\zeta_g$ is a contact deformation of the confoliation $(\Id,\pi)^*\etaline$ on $\Mline\times\surfg$ as follows.
	If $(\eta_g^t)_{t\in[0,1]}$ is a path of confoliations adapted to the upstairs branching locus of $(\Id,\pi)\co M\times\surfg\rightarrow M\times\torus$ starting at $\eta_g^0=(\Id,\pi)^*\eta$, ending at $\eta_g^1=\eta_g$ and such that $\eta_g^t$ is contact for $t\in(0,1]$, then $(p,\Id)^*\eta_g^t$ is the path of confoliations on $\Mline\times\surfg$ which shows that $\zeta_g$ is a contact deformation of $(\Id,\pi)^*\etaline$.
	
	%	We now use the following result, which is essentially just a rephrasing of the discussion in \cite[Page $724$]{NiePre10} for the local situation near the branching locus:
	%	\begin{lemma}%[\cite{NiePre10}]
	%		\label{LemmaSizeNeigh}
	%		For $k\in\nat_{>1}$, let $\pi_k\co\Vhat_k\rightarrow V^{2n+1}$ be a branched covering map of branching index $k$. 
	%		Suppose that all $\pi_k$'s have same downstairs branching locus $M$ and that the upstairs branching locus $\Mhat_k$ of $\pi_k$ is connected (in particular, $\pi_k\vert_{\Mhat_{k}}$ induces a diffeomorphism between $\Mhat_k$ and $M$).
	%		Suppose also that there is a tubular neighborhood $\neigh\coeq M\times\disk^2$ (where $\disk$ is the $2-$disk centered at $0$ and of radius $1$) of the downstairs branching locus $M$ over which \emph{all} the $\pi_k$'s are trivialized at the same time, i.e. such that $\pi_k\co M\times\disk^2\rightarrow M \times \disk^2$ is just $(p,z)\mapsto(p,z^k)$ for all $k$. 
	%		Finally, let $\eta$ be a contact structure on $V$ inducing a contact structure $\xi$ on $M$ and $\etahat_k$ on $\Vhat_k$ be a contact branched covering of $\eta$. 
	%		
	%		Then, there is $\epsilon>0$ such that, for all $k\geq2$, the upstairs branching locus $(\Mhat_k,\xihat_k=\ker\pi_k^*\alpha)\overset{\pi_k}{\simeq}(M,\xi=\ker\alpha)$ has a neighborhood of the form $\left(M\times\disk_{\sqrt{k}\epsilon}^2,\ker\left(\alpha+r^2d\varphi\right)\right)$ inside $(\Vhat_k,\etahat_k)$ (here, by $\disk^2_{r}$ we denote the open disk centered in $0$ and of radius $r$ inside $\real^2$).
	%	\end{lemma}
	
	Notice that, letting $g\geq2$ vary, we get a sequence of branched coverings $\Mline\times\surfg$ of $\Mline\times\torus$, together with contact branched coverings $\zeta_g$ of $\etaline$. 
	\Cref{LemmaSizeNeigh} then tells that each of the fibers $(\overline{M},\xiline=\ker(\overline{\alpha}))$ that belong to the (upstairs) branching set has a contact neighborhood of the form $\left(\Mline\times\disk^2_{R_{g}},\ker\left(\overline{\alpha}+r^2d\theta\right)\right)$, with $R_g\rightarrow +\infty$ for $g\rightarrow + \infty$. 
	Because $\xiline$ on $\Mline$ is overtwisted, this implies, according to Casals--Murphy--Presas \cite[Theorem $3.2$]{CMP15}, that if $g$ is big enough then the upstairs branching set has an overtwisted neighborhood, so that $(\Mline\times\surfg,\zeta_g)$ is also overtwisted. 
	In other words, we just proved that, for $g$ big enough, $(M\times\surfg,\eta_g)$ has a finite cover which is overtwisted.
\end{proof}

Notice that taking $g=1$ in the statement of \Cref{PropVirtOT}, i.e. working directly on $M\times\torus$ without taking a branched covering, is in general not enough to ensure the same conclusion.\\
For instance, this follows from \Cref{SecTightNeighDim3}, where we will show that for each contact manifold $(M^{3},\xi)$, with $\pi_1(M)\neq0$, there is an open book decomposition of $M$ supporting $\xi$ such that the construction in \cite{Bou02} yields a hypertight contact form $\alpha$ on $M\times\torus$.
In particular, even if $(M,\xi)$ is virtually overtwisted, with $(\Mline,\xiline)$ an overtwisted finite cover, the pullback $\overline{\alpha}$ of $\alpha$ to $\Mline\times\torus$ will still define a tight contact structure $\etaline=\ker \overline{\alpha}$:
indeed, if by contradiction $\etaline$ is overtwisted, according to Casals--Murphy--Presas \cite{CMP15} and Albers--Hofer \cite{AlbHof09}, $\overline{\alpha}$ admits a contractible Reeb orbit in $\Mline\times\torus$, which then projects to a contractible Reeb orbit of $\alpha$ in $M\times\torus$, contradicting the hypertightness of $\alpha$. 
\\

We also point out that we preferred to take a very big $g$ in \Cref{PropVirtOT} in order not to enter too much in technical details and to keep the construction simple, but actually $g=2$ is already enough.
Indeed, it's enough to apply the following observation to the overtwisted cover $(\Mline,\xiline)$ in the proof of \Cref{PropVirtOT} above:
\begin{obs}[Massot--Niederkr\"uger]
	\label{ClaimDoubleCovOT}
	If $(M,\xi)$ is overtwisted, the contact manifold $(M\times\surfg,\eta_g)$ 
	is overtwisted already for $g=2$.
\end{obs}
The argument, which we now sketch, is due to Massot and Niederkrüger, and relies on the idea from Presas \cite{Pre07} of using monodromy on contact fiber bundleswith overtwisted fibers in order to find embedded Plastikstufes.
\begin{proof}[Proof (sketch)]
	Take an arc $\gamma$ on $\torus$ going from one (downstairs) branching point of the cover $\Sigma_2\rightarrow\torus$ to the other and such that it is radial in a local model (trivializing the smooth branched covering) around the two branching points, in such a way that its double cover $\delta$ in $\Sigma_2$ is a smooth closed curve. 
	%Denote also $p\in \Sigma_2$ one of the two upstairs branching points, and see $\delta$ as a loop based at $p$. 
	\\
	The monodromy of the contact fiber bundle $M\times\Sigma_2\rightarrow\Sigma_2$ over $\delta$ is trivial. 
	Indeed, as the proof of \Cref{LemmaExistUniqContactizationInFamilies} shows, the contact branched covering $\eta_2$ of a Bourgeois contact structure $(M\times\torus\rightarrow\torus,\eta=\ker\beta)$ can be 
	chosen to be defined by a form $\beta_2$ on $M\times\Sigma_2$ which is invariant under deck transformations of the branched covering $\pi:M\times\Sigma_2\rightarrow M\times\torus$ and $C^\infty$-close to $\pi^*\beta$. 
	Then, one can see that the monodromy of $\left(M\times\Sigma_2\rightarrow\Sigma_2,\eta_2\right)$ over $\delta$ is obtained as the concatenation of the monodromy $f_\gamma$ of $\left(M\times\torus\rightarrow\torus,\eta=\ker\left(\beta\right)\right)$ over $\gamma$, plus a $C^{\infty}-$little perturbation $h$, and the monodromy $\left(f_\gamma\right)^{-1}$ over $-\gamma$, plus the inverse $h^{-1}$ of the same perturbation.
	\\
	Using the techniques from \cite{Pre07}, we can then find an embedded plastikstufe inside $M\times\delta\subset M\times\torus$. 
	In practice,
	if $p\in \Sigma_2$ denotes one of the two upstairs branching points,
	this PS is obtained by parallel-transporting (w.r.t. the connection defined by $\eta_2$) an overtwisted disk in $M\times\{p\}\simeq M$  
	%via the monodromy of $\left(M\times\Sigma_2,\Sigma_2,M,\widehat{\eta}_2\right)$ 
	along $\delta$. 
	This procedure actually gives an embedded PS because the monodromy along the loop $\delta$ is the identity. % as map $M\times\{p\}\rightarrow M\times\{p\}$.
	Finally, Huang \cite{Hua16} tells that each PS-overtwisted manifold is also overtwisted, which concludes.
\end{proof}

%%%%%%%%%%%%%%%%%%%%%%%%%%%%%%%%%%%%%%%%%%%%%%%%%%%%%%%%%%%%%%%%%%%%%%%%%%%%%%%%%%%%
%%%%%%%%%%%%%%%%%%%%%%%%%%%%%%%%%%%%%%%%%%%%%%%%%%%%%%%%%%%%%%%%%%%%%%%%%%%%%%%%%%%%
%%%%%%%%%%%%%%%%%%%%%%%%%%%%%%%%%%%%%%%%%%%%%%%%%%%%%%%%%%%%%%%%%%%%%%%%%%%%%%%%%%%%

\section{Bourgeois construction and Reeb dynamics}
\label{SecTightNeighDim3}

The main aim of this section is to give a proof of \Cref{ThmEmbeddings} stated in \Cref{SecIntro}.
In order to do this, starting from a contact manifold $(M^{2n-1},\xi)$ and an open book $(B,\varphi)$ adapted to $\xi$, we consider in \Cref{SubSecReebVFBouContStr} a strong Bourgeois contact structure $\eta$ on the flat contact bundle $(M\times\torus\rightarrow\torus,\xi\oplus T\torus)$ which admits a contact form $\alpha$ with very specific Reeb vector field. 
This $\eta$ is actually one of the examples described in \cite{Bou02}.
We then show that the Reeb dynamics of $\alpha$ on $M\times\torus$ is strictly related to the Reeb dynamics on the binding $B$ of the open book $(B,\varphi)$. 
This will give a criterion for the existence of closed contractible Reeb orbits of $\alpha$ on $M\times\torus$.
Then, we show in \Cref{SubSecAppl} how to deduce \Cref{ThmEmbeddings} as a corollary of this study in the case of $3$-dimensional $M$.

%%%%%%%%%%%%%%%%%%%%%%%%%%%%%%%%%%%%%%%%%%%%%%%%%%%%%%%%%%%%%%%%%%%%%%%%%%%%%%%%%%%%

\subsection{Bourgeois structures and contractible Reeb orbits}
\label{SubSecReebVFBouContStr}

\begin{prop}
	\label{PropReebVFBouContStr}
	Let $(M,\xi)$ be a $(2n-1)$--dimensional contact manifold, $(B,\varphi)$ an open book decomposition on $M$ supporting $\xi$ and $\beta_0$ a contact form for $\xi$ adapted to $(B,\varphi)$.
	Then, there is a strong Bourgeois contact structure $\eta$ on the flat contact bundle $(M\times\torus\rightarrow\torus,\xi\oplus T \torus)$, which is obtained as in \Cref{ThmBourgeoisBis} and admits a contact form $\alpha$ with associated Reeb vector field of the form 
	%	Let $(M,\xi)$ be a $(2n-1)$-dimensional contact manifold and $(B,\varphi)$ an open book decomposition on $M$ supporting $\xi$.
	%	Then, there are a contact form $\beta$ on $M$ adapted to the open book $(B,\varphi)$, and a strong Bourgeois contact structure $\eta$ on the flat contact bundle $(M\times\torus\rightarrow\torus,\xi\oplus T \torus)$, which is obtained as in \Cref{ThmBourgeoisBis} and admits a contact form $\alpha$ with associated Reeb vector field of the form 
	$$\Reebba = Z+f\, \partial_x - g\,\partial_y\text{ .}$$
	Here, $Z$, $f$ and $g$ are as follows:
	\begin{enumerate}[label=\alph*.]
		\item \label{Item1PropReebVFBouContStr} $Z$ is a smooth vector field on $M$ such that:
		\begin{enumerate}[label=\roman*.]
			\item on $M\setminus B$, it is tangent to the fibers of $\varphi$,
			\item on the binding $B$, it is 
			% equal, up to a non-zero constant factor, to 
			a (non--zero) multiple of
			the Reeb vector field $\ReebB$ of the restriction of $\beta_0$ to $B$;
		\end{enumerate} 
		\item \label{Item2PropReebVFBouContStr} $f,g\co M \rightarrow \real$ are smooth functions such that $(f,g)\co M \rightarrow \real^2$ is positively proportional to $(\cos\varphi,\sin\varphi)$ on $M\setminus B$ and $f=g=0$ on $B$.
	\end{enumerate}
\end{prop}

For the proof of \Cref{PropReebVFBouContStr} we will need the following result, whose proof can be found for instance in D\"orner--Geiges--Zehmisch \cite[Section 3]{DGZ14}:
\begin{lemma}[Giroux]
	\label{PropLocalFormNearBinding}
	Let $\diskone\subset\real^{2}$ be the disk centered at the origin with radius $1$ and $\beta$ be a contact form on $B\times\diskone$ with the following properties:
	\begin{enumerate}
		\item \label{Item1PropLocalFormNearBinding} $\betaB:=\beta\vert_{TB}$ is a contact form on $B=B\times\{0\}$.
		\item \label{Item2PropLocalFormNearBinding} For each $\varphi\in\cercle$, $d\beta\vert_{T\Sigma_{\varphi}}$ is a symplectic form on $\Sigma_\varphi\setminus B$, where $$\Sigma_\varphi\,\,=\,\,\{\,(p,r,\varphi)\in B\times\diskone \, \vert \, p\in B , 0 \leq r \leq 1\,\}\text{ .}$$
		\item \label{Item3PropLocalFormNearBinding} With the orientations of $B$ and $\Sigma_\varphi$ induced,  respectively, by $\betaB$ and $d\beta\vert_{T\Sigma_\varphi}$,  $B$ is oriented as the boundary of $\Sigma_\varphi$.
	\end{enumerate}
	Then, for a sufficiently small $\delta>0$, there is an embedding $B\times\diskdelta\rightarrow B\times\diskone$ (here, $\diskdelta\subset\real^{2}$ denotes the disk centered at the origin of radius $\delta>0$) which preserves the angular coordinate $\varphi$ on the second factor, is the identity on $B\times\{0\}$ and pulls back a convenient isotopic modification $\beta'$ of $\beta$ (with an isotopy between contact forms that satisfy Hypothesis \ref{Item1PropLocalFormNearBinding}, \ref{Item2PropLocalFormNearBinding} and \ref{Item3PropLocalFormNearBinding} above) to a $1$-form $\hone(r)\,\betaB\,+\,\htwo(r)\, d\varphi$, where:
	\begin{enumerate}[label=\roman*.]
		\item \label{Item1ConclPropLocalFormNearBinding} $\hone(0)>0$ and $\hone(r)=\hone(0) + O(r^2)$ for $r\rightarrow0$,
		\item \label{Item1bisConclPropLocalFormNearBinding} $\htwo(r)\sim r^2$ for $r\rightarrow0$,
		\item \label{Item2ConclPropLocalFormNearBinding} if $H:=\hone^{n-1}\,(\hone\htwo'-\htwo\hone')$, then $\frac{H}{r}>0\,\, \forall r\geq 0$ (contact condition);
		\item \label{Item3ConclPropLocalFormNearBinding} $\hone'(r)<0$ for $r>0$, (symplectic condition on $\Sigma_\varphi$). 
	\end{enumerate}
\end{lemma}

\begin{proof}[Proof (\Cref{PropReebVFBouContStr})]
	We start by finding a convenient isotopic modification of the adapted form $\beta_0$ in the statement as well as a particular normal neighborhood $\neigh$ of the binding and a particular smooth map $\phi\co M\rightarrow \real^2$ defining $(B,\varphi)$.
	Take a normal neighborhood $B\times\disk^2$ of the binding $B$ in $M$ such that $\varphi\co B\times\left(\disk^2\setminus\left\{0\right\}\right)\rightarrow \cercle$ becomes the angular coordinate of $\disk^2\setminus\left\{0\right\}$.
	Such neighborhood exists by definition of open book decomposition.  
	%Also, take a contact form $\beta_0$ which defines $\xi$ and is adapted to the open book $(B,\varphi)$ on $M$. 
	Then, \Cref{PropLocalFormNearBinding} gives an isotopic modification $\beta$ of $\beta_0$, still adapted to the same open book, and of the form $\hone \betaB + \htwo d\varphi$ in the neighborhood $\neigh\coeq B\times \disk^2_\delta\subset B\times\disk^2$.
	Moreover, $\beta\vert_{TB}=\beta_0\vert_{TB}$, so that they induce the same Reeb vector field on $B$; as in the statement, we denote it $\ReebB$.
	\\
	%To define $\phi\co M \rightarrow \real^2$, 
	Consider now a function $\rho:M\rightarrow \real$ which is smooth away from $B$, equal to the radial coordinate $r$ of $\disk^2_\delta$ on the neighborhood $\{r\leq \sfrac{\delta}{3}\}$ of $B\times\{0\}$ inside $\neigh = B\times \disk^2_\delta$, equal to $1$ on the complement in $M$ of the open set $\{r<\sfrac{2\delta}{3}\}\subset \neigh$, and depending only on $r$ in a strictly increasing way on the set ${\sfrac{\delta}{3}<r<\sfrac{2\delta}{3}}$.
	Then, we define $\phi \coeq \rho\, (\cos\varphi,\sin\varphi)$. 
	Remark that such a $\phi$ is indeed well defined and smooth on all $M$, and defines the open book $(B,\varphi)$.
	
	We now define two functions $\lambda,\mu \co M\rightarrow \real$ as follows:
	\begin{equation*} \lambda \, =\, \begin{cases} \frac{\rho'}{\rho'\hone-\rho \hone'} & \text{inside } \neigh \\  0 & \text{outside }\neigh \end{cases} 
	\, \, \text{ and }\,\, 
	\mu \, = \begin{cases} \frac{-\hone'}{\rho' \hone-\rho\hone'} & \text{inside } \neigh \\  1 & \text{outside }\neigh \end{cases} \text{ .}
	\end{equation*}
	Notice that they are well defined smooth functions on all $M\times\torus$.
	Indeed, $\rho'$ smoothly extends as $1$ at $r=0$, $\hone'=O(r)$ near $r=0$ (by point \ref{Item1ConclPropLocalFormNearBinding} of \Cref{PropLocalFormNearBinding}) and $\rho' \hone - \rho \hone'$ is positive for $r>0$ and smoothly extends as $\hone(0)$ at $r=0$.
	\\
	Consider then $Z \coeq \lambda \ReebB$ and $(f,g)\coeq \mu (\cos\varphi,\sin\varphi)$.
	Here, $\ReebB$ is seen as as a vector field on $\neigh=B\times\diskdelta$ tangent to the first factor and $\lambda$ has support contained inside $\neigh$, hence $\lambda\ReebB$ is well defined on all $M$.
	Similarly, $f,g$ are well defined because $\mu$ is zero on $B$. 
	It is also easy to check that such $Z,f,g$ satisfy points \ref{Item1PropReebVFBouContStr} and \ref{Item2PropReebVFBouContStr} of \Cref{PropReebVFBouContStr}.
	
	Finally, we have to choose a contact form $\alpha$ defining a strong Bourgeois contact structure $\eta$ on the flat contact bundle $(M\times\torus\rightarrow\torus,\xi\oplus T\torus)$, as desired in the statement of \Cref{PropReebVFBouContStr}. Let $\alpha:= \beta + \phione dx - \phitwo dy$, i.e. the one obtained from \Cref{ThmBourgeoisBis}.\ref{Item2ThmBourgeoisBis} with the choices of $\phi$ and $\neigh$ made above.
	\\
	We already know that the contact structures given by \Cref{ThmBourgeoisBis} are in particular strong Bourgeois structures. 
	An explicit computation also shows that $Z+f\partial_x - g \partial_y$ is indeed the Reeb vector field associated to $\alpha$, as desired. %, which concludes the proof of \Cref{PropReebVFBouContStr}.
\end{proof}

%%%%%%%%%%%%%%%%%%%%%%%%%%%%%%%%%%%%%%%%%%%%%%%%%%%%%%%%%%%%%%%%%%%%%%%%%%%%%%%%%%%%

We have the following immediate consequence on the Reeb dynamics:
\begin{cor}
	\label{LemmaClosContrOrbit}
	Let $\alpha$ on $M\times\torus$ be the contact form given by \Cref{PropReebVFBouContStr}. 
	Then, the closed contractible orbits of $\Reebba$ in $M\times\torus$ are of the form $\mathcal{O}\times\{pt\}$, where $pt\in\torus$ and $\mathcal{O}$ is a closed orbit of $\ReebB$ in $B$ which is contractible in $M$.   
	%	Then, the closed contractible orbits of $\Reebba$ in $M\times\torus$ are of the form $\OrbBqzero\times\{(x_0,y_0)\}$, where $(x_0,y_0)\in\torus$ and $\OrbBqzero$ is a closed orbit of $\ReebB$ in $B$ which is contractible in $M$.   
\end{cor}

Notice that, even if the closed orbits of $\ReebB$ are contained in $B$, we are interested here in their homotopy class as loops in $M$.

%%%%%%%%%%%%%%%%%%%%%%%%%%%%%%%%%%%%%%%%%%%%%%%%%%%%%%%%%%%%%%%%%%%%%%%%%%%%%%%%%%%%

\subsection{Embedding 3-manifolds in (hyper)tight 5-manifolds}
\label{SubSecAppl}

%%%%%%%%%%%%%%%%%%%%%%%%%%%%%%%%%%%%%%%%%%%%%%%%%%%%%%%%%%%%%%%%%%%%%%%%%%%%%%%%%%%%

We start with a proposition on (topological) open books of $3$-manifolds:

\begin{prop}
	\label{PropPosStab}
	Let $M$ be a $3$-manifold with $H_1\left(M;\rat\right)\neq\{0\}$. Then, every open book decomposition $(K,\varphi)$ of $M$ %different from $\sphere{3}$ 
	can be transformed, by a sequence of positive stabilizations, to an open book decomposition $(K',\varphi')$ with binding $K'$ having at most $2$ connected components, each of which has infinite order in $H_1\left(M;\integ\right)$.
\end{prop}

For the proof of \Cref{PropPosStab} we need the following embedded version of the stabilization procedure for open book decompositions, which, as explained in Giroux--Goodman \cite{GirGoo06}, essentially follows from Stallings' study in \cite{Sta78}:
\begin{thm}
	\label{ThmEmbStabOBD}
	Let $\Sigma$ be a compact surface with boundary in a manifold $M$ and $\delta_0$ a properly embedded arc in $\Sigma$.
	Let also $\Sigma'\subset M$ be obtained by plumbing a positive Hopf band to $\Sigma$, i.e. $\Sigma' = \Sigma \cup A$ where $A$ is an annulus in $M$ such that
	\begin{enumerate}
		\item the intersection $A\cap\Sigma$ is a tubular neighborhood of $\delta_0$,
		\item the core curve $\delta$ of $A$ bounds a disk in $M\setminus \Sigma$ and the linking number of the boundary components is $+1$.
	\end{enumerate}
	If $\Sigma$ is a page of an open book decomposition $(B,\varphi)$ of $M$, then $\Sigma'$ is also a page of an open book $(B',\varphi')$ of $M$. 
\end{thm}

\begin{proof}[Proof (\Cref{PropPosStab})]
	We start by applying a sequence of stabilizations to reduce the number of connected components of the boundary of the pages to one. 
	We can thus suppose that the open book decomposition $(K, \varphi)$ has connected binding $K$. 
	Notice that if $[K]\in H_1(M;\integ)$ is of infinite order, then we have nothing to prove.
	We can hence suppose that it is a torsion element.
	%	\\
	
	We now consider the following set of generators for $H_1(M;\integ)$.\\
	Let $p\in K$, $\Sigma$ be a page of $(K,\varphi)$ inside $M$ and $\alpha_1,\beta_1,\ldots,\alpha_g,\beta_g,K$ be curves on $\Sigma$ that generate $\pi_1(\Sigma,p)$, as in \Cref{FigGenerators}.
	We can then use Van-Kampen theorem (see for instance Etnyre--Ozbagci \cite{EtnOzb08}) with the following two open sets: $\mathcal{U}$ given by an open neighborhood $K\times\disk$ of the binding $K$ and $\mathcal{V}\coeq M \setminus K$.
	Because $\mathcal{V}$ is just the mapping torus of the monodromy $\varphi\co \Sigma\rightarrow\Sigma$, we then get that the inclusion $\Sigma\hookrightarrow M$ induces a surjection at the $\pi_1$-level. 
	Moreover, by Hurewicz theorem, the same is true at the $H_1$-level, i.e. $H_1(M,\integ)$ is generated by the homology classes of (the images in $M$ of) $\alpha_1,\beta_1,\ldots,\alpha_g,\beta_g, K$.

	\begin{figure}[htb]
		\centering
		\def\svgwidth{200pt}
		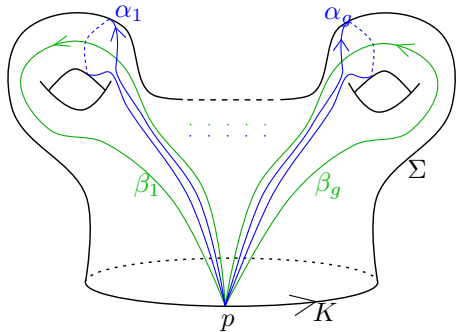
		\caption{Curves on a page $\Sigma$ of $(K,\varphi)$ which give a set of generators of $\pi_1(M,p)$.} %the presentation in \Cref{EqFondGroupMWithSigmaOld}.}
		\label{FigGenerators}
	\end{figure}
	
	Now, by the hypothesis $H_1(M;\rat)\neq\{0\}$ and by the universal coefficients theorem, at least one of the generators $[\alpha_i],[\beta_i],[K]$ of $H_1(M;\integ)$ is non-torsion; we can w.l.o.g. assume that this is the case for $[\beta_g]$ (as we are assuming that $[K]$ is torsion).
	
	%	We recall that, as explained in Giroux--Goodman \cite{GirGoo06}, we have the following embedded version of the stabilization procedure for open book decompositions, which essentially follows from Stallings' study in \cite{Sta78}:
	%	\begin{thm}
	%		Let $\Sigma$ be a compact surface with boundary in a manifold $M$ and $\delta_0$ a properly embedded arc in $\Sigma$.
	%		Let also $\Sigma'\subset M$ be obtained by plumbing a positive Hopf band to $\Sigma$, i.e. $\Sigma' = \Sigma \cup A$ where $A$ is an annulus in $M$ such that
	%		\begin{enumerate}
	%			\item the intersection $A\cap\Sigma$ is a tubular neighborhood of $\delta_0$,
	%			\item the core curve $\delta$ of $A$ bounds a disk in $M\setminus \Sigma$ and the linking number of the boundary components is $+1$.
	%		\end{enumerate}
	%		If $\Sigma$ is a page of an open book decomposition $(B,\varphi)$ of $M$, then $\Sigma'$ is also a page of an open book $(B',\varphi')$. 
	%	\end{thm}
	
	Consider then a surface $\Sigma'$ obtained, as surface embedded in $M$, by plumbing a positive Hopf band $A$ to $\Sigma$ along a properly embedded arc $\delta_0$ which is in the same class as $\beta_g$ in $\pi_1(\Sigma,\partial\Sigma)$, as shown in \Cref{FigArc,FigSigmaPrime}. 
	According to \Cref{ThmEmbStabOBD}, $\Sigma'$ is the page of an open book of $M$.
	\begin{figure}[t]
		\centering
        \def\svgwidth{190pt}
		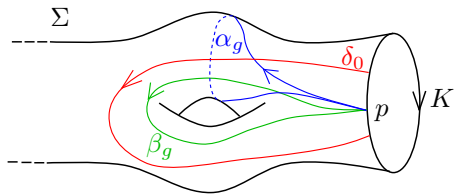
		\caption{Stabilization arc $\delta_0$. Only a part of the surface is shown.}
		\label{FigArc}
	\end{figure}
	In particular, we know that the core $\delta$ of $A$ bounds a disk $\Delta$ in $M\setminus \Sigma$. 
	
	Now, 
	$\delta$ is homotopic (in $\Sigma'$) to the concatenation $(K_1')^{-1}\ast \beta_g$ of the inverse of the boundary component $K_1'$ of $\Sigma'=\Sigma\cup A$ that passes through $p$ and $\beta_g$; see  \Cref{FigSigmaPrime}. 
	The existence of the disk $\Delta$ then tells that $[K_1']=[\beta_g]$ in $H_1(M;\integ)$ and, because $[\beta_g]$ is non-torsion, the same is true for $[K_1']$.
	\\
	%this null-homotopy gives a homotopy between one of the two boundary components of $\Sigma'$, which we can call $K_1'$, and $\beta_g$ (seen as a curve on $\Sigma'$ via the natural inclusion $\Sigma\subset\Sigma'$ as surfaces embedded in $M$). 
	%Hence, because $\beta_g$ is homotopically non-trivial, the same is true for $K_1'$.\\
	Moreover, $K=\partial \Sigma$ is cohomologous (in $\Sigma'$) to $K_1'\sqcup K_2'$; see again \Cref{FigSigmaPrime}.
	Working in $H_1(M;\integ)$ modulo torsion, it is then easy to check that $[K]$ torsion and $[K_1']$ non-torsion imply that $[K_2']$ is also non-torsion, as desired.
	%Similarly, the fact that the core of $A$ bounds a disk implies that the other connected component of $\partial\Sigma'$, which we call $K_2'$, is homotopic to $K\ast \beta_g^{-1}$; here, $\ast$ is the concatenation of paths and $.^{-1}$ is the inverse of a path. In particular, $K_2'$ is homotopically non-trivial, because $K$ is trivial and $\beta_g$ is not.
\end{proof}%\\
%\begin{figure}[htb]
%    \begin{subfigure}{0.40\textwidth}
%        \includegraphics[width=\linewidth]{Fig2.eps}
%		\caption{Stabilization arc $\delta_0$. Only a part of the surface is shown.}
%    \end{subfigure}
%    \hspace*{\fill} % separation between the subfigures
%    \begin{subfigure}{0.55\textwidth}
%        \includegraphics[width=\linewidth]{Fig3.eps}
%    	\caption{Stabilized surface $\Sigma'$; the representation is abstract and not embedded for simplicity (the boundary components of $A$ have trivial linking number in the picture).}
%    \end{subfigure}
%    \caption{Page of the open book before and after stabilization.} 
%        \label{Fig23}
%\end{figure}

\begin{figure}[t]
	\centering
	\def\svgwidth{210pt}
	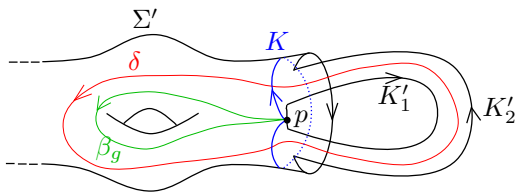
	\caption{Stabilized surface $\Sigma'$. The representation is abstract and not embedded for simplicity (the boundary components of $A$ have trivial linking number in the picture).}
	\label{FigSigmaPrime}
\end{figure}

%\begin{figure}[htb]
%	\begin{minipage}{0.35\textwidth}
%		\centering
%		%\def\svgwidth{160pt}
%		%\input{img1.pdf_tex}
%		\includegraphics[width=150pt]{Fig2.eps}
%		\caption{Stabilization arc $\delta_0$. Only a part of the surface is shown.}
%		\label{FigArc}
%	\end{minipage}\hfill
%	\begin{minipage}{0.55\textwidth}
%		\centering
%		%\def\svgwidth{190pt}
%		%\input{img1_bis.pdf_tex}
%		\includegraphics[width=180pt]{Fig3.eps}
%		\caption{Stabilized surface $\Sigma'$; the representation is abstract and not embedded for simplicity (the boundary components of $A$ have trivial linking number in the picture).}
%		\label{FigSigmaPrime}
%	\end{minipage}
%\end{figure}

%Consider now a $3-$dimensional contact manifold $(M,\xi)$..
\Cref{LemmaClosContrOrbit} and \Cref{PropPosStab} have the following direct consequence:
\begin{prop}
	\label{PropHypTightBouContStr}
	Let $(M^3,\xi)$ be a contact manifold with $H_1(M;\rat)\neq\{0\}$. 
	\\
	Then, there is a hypertight strong Bourgeois contact structure on $(M\times\torus\rightarrow\torus, \xi\oplus T \torus)$. 
	More precisely, given any open book $(K,\varphi)$ supporting $\xi$, there is another supporting $(K',\varphi')$, obtained from  $(K,\varphi)$ by a sequence of positive stabilizations, such that the strong Bourgeois contact structure on $M\times\torus$ obtained as in \Cref{ThmBourgeoisBis} from $(M,\xi,(K',\varphi'))$ is hypertight.
\end{prop}

Recall that a contact structure is called hypertight if it admits a contact form with non--contractible Reeb orbits.

Let's denote by $\disk^2_R$ the disk of radius $R>0$ centered at the origin in $\real^2$ and by $(r,\varphi)$ the polar coordinates on it. 
\Cref{ThmEmbeddings} and \Cref{CorIntroEmbeddings} then follow from \Cref{PropHypTightBouContStr}:

\begin{thm}
	\label{PropContEmb}
	Every closed $3$-dimensional contact manifold $(M,\xi)$ with non-trivial $H_1(M;\rat)$ can be embedded, with trivial conformal symplectic normal bundle, in a hypertight closed $5$-dimensional contact manifold $(N,\eta)$.\\ 
	In particular, for each contact form $\alpha$ defining $\xi$ on $M$, there is an $\epsilon>0$ such that $\left(M\times\disk^2_\epsilon,\ker\left(\alpha+r^2d\varphi\right)\right)$ is tight.
\end{thm}

As already remarked in the introduction, Hern{\'a}ndez-Corbato -- Mart{\'{\i}}n-Merch{\'a}n -- Presas \cite{HerMarPre18} deal with the higher dimensional case.
More precisely, they give a generalization of the second part of this result, as well as an analogue (with less control on the codimension) of the first part of it.

\begin{proof}[Proof (\Cref{PropContEmb})]
	Consider an arbitrary contact $3$-manifold $(M,\xi)$ with $H_1(M;\rat)\neq\{0\}$ and take one of the hypertight contact manifolds $(M\times\torus,\eta)$ given by \Cref{PropHypTightBouContStr}.
	\\
	Each $\left(M\times\{pt\},\eta\cap T\left(M\times\{pt\}\right)\right))$ is then exactly $(M,\xi)$ and it has topologically trivial normal bundle, hence trivial conformal symplectic normal bundle. Indeed, a symplectic vector bundle of rank $2$ is symplectically trivial if and only if it is topologically trivial. 
	
	As far as the second part of the statement is concerned, according to the standard neighborhood theorem for contact submanifolds (see for instance Geiges \cite[Theorem 2.5.15]{Gei08}), the contact submanifold $(M,\xi=\ker(\alpha))=\left(M\times\{pt\},\eta\cap T\left(M\times\{pt\}\right)\right))$ of $(M\times\torus,\eta)$ has a contact neighborhood of the form $\left(M\times\disk^2_\epsilon,\ker\left(\alpha+r^2d\varphi\right)\right)$, for a certain real $\epsilon>0$.
	Moreover, each hypertight high dimensional contact manifold is in particular tight, according to Albers--Hofer  \cite{AlbHof09} and Casals--Murphy--Presas \cite{CMP15}.
	In particular, $(M\times\torus,\eta)$ is tight.
	Then, $\left(M\times\disk^2_\epsilon,\ker\left(\alpha+r^2d\varphi\right)\right)$ is tight too, because it embeds (in codimension $0$) in a tight contact manifold.
\end{proof}

%%%%%%%%%%%%%%%%%%%%%%%%%%%%%%%%%%%%%%%%%%%%%%%%%%%%%%%%%%%%%%%%%%%%%%%%%%%%%%%%%%%%
%%%%%%%%%%%%%%%%%%%%%%%%%%%%%%%%%%%%%%%%%%%%%%%%%%%%%%%%%%%%%%%%%%%%%%%%%%%%%%%%%%%%
%%%%%%%%%%%%%%%%%%%%%%%%%%%%%%%%%%%%%%%%%%%%%%%%%%%%%%%%%%%%%%%%%%%%%%%%%%%%%%%%%%%%

\bibliographystyle{halpha.bst}
\bibliography{my_bibliography}

\begin{thebibliography}{MNW13}

\bibitem[AH09]{AlbHof09}
Peter Albers and Helmut Hofer.
\newblock On the {W}einstein conjecture in higher dimensions.
\newblock {\em Comment. Math. Helv.}, 84(2):429--436, 2009.

\bibitem[BEM15]{BorEliMur15}
Matthew Borman, Yakov Eliashberg, and Emmy Murphy.
\newblock Existence and classification of overtwisted contact structures in all
  dimensions.
\newblock {\em Acta Math.}, 215(2):281--361, 2015.

\bibitem[BGM19]{BowGirMor}
Jonathan {Bowden}, Fabio {Gironella}, and Agustin {Moreno}.
\newblock {Bourgeois contact structures: tightness, fillability and
  applications}.
\newblock {\em arXiv e-prints}, page arXiv:1908.05749, Aug 2019, 1908.05749.

\bibitem[Bou02]{Bou02}
Fr{\'e}d{\'e}ric Bourgeois.
\newblock Odd dimensional tori are contact manifolds.
\newblock {\em Int. Math. Res. Not.}, (30):1571--1574, 2002.

\bibitem[CE12]{CieEliBook}
Kai Cieliebak and Yakov Eliashberg.
\newblock {\em From {S}tein to {W}einstein and back}, volume~59 of {\em
  American Mathematical Society Colloquium Publications}.
\newblock American Mathematical Society, Providence, RI, 2012.
\newblock Symplectic geometry of affine complex manifolds.

\bibitem[CMP19]{CMP15}
Roger Casals, Emmy Murphy, and Francisco Presas.
\newblock Geometric criteria for overtwistedness.
\newblock {\em J. Amer. Math. Soc.}, 32(2):563--604, 2019.

\bibitem[CPS16]{CPS14}
Roger Casals, Francisco Presas, and Sheila Sandon.
\newblock Small positive loops on overtwisted manifolds.
\newblock {\em J. Symplectic Geom.}, 14(4):1013--1031, 2016.

\bibitem[DGZ14]{DGZ14}
Max D\"orner, Hansj\"org Geiges, and Kai Zehmisch.
\newblock Open books and the weinstein conjecture.
\newblock {\em The Quarterly Journal of Mathematics}, 65(3):869, 2014.

\bibitem[EF17]{EtnFur17}
John Etnyre and Ryo Furukawa.
\newblock Braided embeddings of contact 3-manifolds in the standard contact
  5-sphere.
\newblock {\em J. Topol.}, 10(2):412--446, 2017.

\bibitem[EL19]{EtnLek17}
John~B. Etnyre and Yank\i Lekili.
\newblock Embedding all contact 3-manifolds in a fixed contact 5-manifold.
\newblock {\em J. Lond. Math. Soc. (2)}, 99(1):52--68, 2019.

\bibitem[Eli92]{Eli92}
Yakov Eliashberg.
\newblock Contact {$3$}-manifolds twenty years since {J}. {M}artinet's work.
\newblock {\em Ann. Inst. Fourier (Grenoble)}, 42(1-2):165--192, 1992.

\bibitem[EO08]{EtnOzb08}
John Etnyre and Burak Ozbagci.
\newblock Invariants of contact structures from open books.
\newblock {\em Trans. Amer. Math. Soc.}, 360(6):3133--3151, 2008.

\bibitem[EP11]{EtnPan11}
John Etnyre and Dishant Pancholi.
\newblock On generalizing {L}utz twists.
\newblock {\em J. Lond. Math. Soc. (2)}, 84(3):670--688, 2011.

\bibitem[EP16]{EtnPan16}
John~B. Etnyre and Dishant~M. Pancholi.
\newblock Corrigendum: {O}n generalizing {L}utz twists [ {MR}2855796].
\newblock {\em J. Lond. Math. Soc. (2)}, 94(2):662--665, 2016.

\bibitem[ET98]{EliThu98}
Yakov Eliashberg and William Thurston.
\newblock {\em Confoliations}, volume~13 of {\em University Lecture Series}.
\newblock American Mathematical Society, Providence, RI, 1998.

\bibitem[Gei97]{Gei97}
Hansj{\"o}rg Geiges.
\newblock Constructions of contact manifolds.
\newblock {\em Math. Proc. Cambridge Philos. Soc.}, 121(3):455--464, 1997.

\bibitem[Gei08]{Gei08}
Hansj\"org Geiges.
\newblock {\em An introduction to contact topology}, volume 109 of {\em
  Cambridge Studies in Advanced Mathematics}.
\newblock Cambridge University Press, Cambridge, 2008.

\bibitem[GG06]{GirGoo06}
Emmanuel Giroux and Noah Goodman.
\newblock On the stable equivalence of open books in three-manifolds.
\newblock {\em Geom. Topol.}, 10:97--114, 2006.

\bibitem[Gir91]{Gir91}
Emmanuel Giroux.
\newblock Convexit\'e en topologie de contact.
\newblock {\em Comment. Math. Helv.}, 66(4):637--677, 1991.

\bibitem[Gir99]{Gir99}
Emmanuel Giroux.
\newblock Une infinit\'e de structures de contact tendues sur une infinit\'e de
  vari\'et\'es.
\newblock {\em Invent. Math.}, 135(3):789--802, 1999.

\bibitem[Gir00]{Gir00}
Emmanuel Giroux.
\newblock Structures de contact en dimension trois et bifurcations des
  feuilletages de surfaces.
\newblock {\em Invent. Math.}, 141(3):615--689, 2000.

\bibitem[Gir02]{Gir02}
Emmanuel Giroux.
\newblock G\'eom\'etrie de contact: de la dimension trois vers les dimensions
  sup\'erieures.
\newblock In {\em Proceedings of the {I}nternational {C}ongress of
  {M}athematicians, {V}ol. {II} ({B}eijing, 2002)}, pages 405--414. Higher Ed.
  Press, Beijing, 2002.

\bibitem[Gir12]{GirouxAIM}
Emmanuel Giroux.
\newblock The existence problem, contact geometry in high dimensions.
\newblock http://www.aimath.org/WWN/contacttop/aim12-giroux.pdf, May 2012.

\bibitem[Gir18]{MyPhDThesis}
Fabio Gironella.
\newblock {\em On some constructions of contact manifolds}.
\newblock PhD thesis, Université Paris-Saclay, Palaiseau, 7 2018.

\bibitem[Gom98]{Gom98}
Robert Gompf.
\newblock Handlebody construction of {S}tein surfaces.
\newblock {\em Ann. of Math. (2)}, 148(2):619--693, 1998.

\bibitem[Gro86]{GroPartDiff}
Mikhail Gromov.
\newblock {\em Partial differential relations}, volume~9 of {\em Ergebnisse der
  Mathematik und ihrer Grenzgebiete (3) [Results in Mathematics and Related
  Areas (3)]}.
\newblock Springer-Verlag, Berlin, 1986.

\bibitem[HMP18]{HerMarPre18}
Luis {Hern{\'a}ndez-Corbato}, Luc{\'\i}a {Mart{\'\i}n-Merch{\'a}n}, and
  Francisco {Presas}.
\newblock {Tight neighborhoods of contact submanifolds}.
\newblock {\em arXiv e-prints}, page arXiv:1802.07006, Feb 2018, 1802.07006.

\bibitem[Hof93]{Hof93}
Helmut Hofer.
\newblock Pseudoholomorphic curves in symplectizations with applications to the
  {W}einstein conjecture in dimension three.
\newblock {\em Invent. Math.}, 114(3):515--563, 1993.

\bibitem[Hon00]{Hon00}
Ko~Honda.
\newblock On the classification of tight contact structures. {I}.
\newblock {\em Geom. Topol.}, 4:309--368, 2000.

\bibitem[Hua17]{Hua16}
Yang Huang.
\newblock On plastikstufe, bordered {L}egendrian open book and overtwisted
  contact structures.
\newblock {\em J. Topol.}, 10(3):720--743, 2017.

\bibitem[KM97]{KriMicConvSett}
Andreas Kriegl and Peter Michor.
\newblock {\em The convenient setting of global analysis}, volume~53 of {\em
  Mathematical Surveys and Monographs}.
\newblock American Mathematical Society, Providence, RI, 1997.

\bibitem[KN05]{NieVKo05}
Otto~van Koert and Klaus Niederkr\"uger.
\newblock Open book decompositions for contact structures on {B}rieskorn
  manifolds.
\newblock {\em Proc. Amer. Math. Soc.}, 133(12):3679--3686, 2005.

\bibitem[KN07]{NieVKo07}
Otto~van Koert and Klaus Niederkr\"uger.
\newblock Every contact manifolds can be given a nonfillable contact structure.
\newblock {\em Int. Math. Res. Not. IMRN}, (23):Art. ID rnm115, 22 pages, 2007.

\bibitem[Ler04]{Ler04}
Eugene Lerman.
\newblock Contact fiber bundles.
\newblock {\em J. Geom. Phys.}, 49(1):52--66, 2004.

\bibitem[LMN18]{LisMarNie18}
Samuel {Lisi}, Aleksandra {Marinkovi{\'c}}, and Klaus {Niederkr{\"u}ger}.
\newblock {On properties of Bourgeois contact structures}.
\newblock {\em arXiv e-prints}, page arXiv:1801.00869, Jan 2018, 1801.00869.

\bibitem[Lut79]{Lut79}
Robert Lutz.
\newblock Sur la g\'eom\'etrie des structures de contact invariantes.
\newblock {\em Ann. Inst. Fourier (Grenoble)}, 29(1):xvii, 283--306, 1979.

\bibitem[MNW13]{MNW13}
Patrick Massot, Klaus Niederkr\"uger, and Chris Wendl.
\newblock Weak and strong fillability of higher dimensional contact manifolds.
\newblock {\em Invent. Math.}, 192(2):287--373, 2013.

\bibitem[Nie13]{NieThesis}
Klaus Niederkr\"uger.
\newblock {\em On fillability of contact manifolds}.
\newblock M\'emoire d'habilitation \`a diriger des recherches, Universit\'e
  Paul Sabatier, Decembre 2013.
\newblock https://tel.archives-ouvertes.fr/tel-00922320.

\bibitem[NP10]{NiePre10}
Klaus Niederkr\"uger and Francisco Presas.
\newblock Some remarks on the size of tubular neighborhoods in contact topology
  and fillability.
\newblock {\em Geom. Topol.}, 14(2):719--754, 2010.

\bibitem[ON07]{OztNie07}
Ferit \"Ozt\"urk and Klaus Niederkr\"uger.
\newblock Brieskorn manifolds as contact branched covers of spheres.
\newblock {\em Period. Math. Hungar.}, 54(1):85--97, 2007.

\bibitem[Pre07]{Pre07}
Francisco Presas.
\newblock A class of non-fillable contact structures.
\newblock {\em Geom. Topol.}, 11:2203--2225, 2007.

\bibitem[Sta78]{Sta78}
John Stallings.
\newblock Constructions of fibred knots and links.
\newblock In {\em Algebraic and geometric topology ({P}roc. {S}ympos. {P}ure
  {M}ath., {S}tanford {U}niv., {S}tanford, {C}alif., 1976), {P}art 2}, Proc.
  Sympos. Pure Math., XXXII, pages 55--60. Amer. Math. Soc., Providence, R.I.,
  1978.

\bibitem[Vog16]{Vog16b}
Thomas Vogel.
\newblock On the uniqueness of the contact structure approximating a foliation.
\newblock {\em Geom. Topol.}, 20(5):2439--2573, 2016.

\end{thebibliography}

\Addresses

\end{document}